\documentclass[12pt]{extarticle}
\usepackage{amssymb,amsfonts,amsthm,amsmath,bbold}
\usepackage{mathrsfs}
\usepackage{mathabx}
\usepackage[all]{xy}
\usepackage{array}
\usepackage{transparent}

\usepackage[top=1.2in, bottom=1.2in, left=0.5in,
right=0.5in]{geometry}

\usepackage{tikz}
\usepackage{tikz-cd}
\usetikzlibrary{matrix,arrows}

\usepackage{wela}
\usepackage{miama}
\usepackage{pifont}
\usepackage[T1]{fontenc}

\newcommand{\cE}{\mathscr{E}}
\newcommand{\cA}{\mathscr{A}}

\usepackage{bm}

\newcommand\Tate[1]{\bm{(}#1\bm{)}}
\newcommand\sTate[1]{\left|\!\left[#1\right]\!\right|}

\usepackage{hyperref}
\hypersetup{
    colorlinks=true,
    linkcolor=blue,
    filecolor=blue,      
    urlcolor=blue,
    }
\usepackage[backrefs,msc-links]{amsrefs}

    \definecolor{lgray}{rgb}{0.8, 0.8, 0.8}
    \definecolor{blue1}{rgb}{0,0.3,1}
    \definecolor{peach1}{rgb}{1,0.592,0.557}

\newcommand{\C}{\mathbb{C}}

\newcommand{\Gm}{\mathbb{G}_\mathbf{m}}

\newcommand{\Ct}{\mathbb{C}^\times}
\newcommand{\Q}{\mathbb{Q}}
\newcommand{\Qlb}{\overline{\Q}_\ell}

\newcommand{\A}{\mathbb{A}}
\newcommand{\F}{\mathbb{F}}
\newcommand{\bbf}{\mathbb{f}}
\newcommand{\bO}{\mathbb{O}}
\newcommand{\bD}{\mathbb{D}}

\newcommand{\Z}{\mathbb{Z}}

\newcommand{\X}{\mathbb{X}}
\newcommand{\R}{\mathbb{R}}

\newcommand{\bA}{\mathsf{A}}

\newcommand{\bff}{\mathsf{f}}
\newcommand{\bG}{\mathsf{G}}
\newcommand{\bH}{\mathsf{H}}

\newcommand{\sP}{\mathsf{P}}

\newcommand{\bC}{\mathsf{C}}

\newcommand{\bY}{\mathsf{Y}}

\newcommand{\bL}{\widehat{\mathsf{L}}}

\newcommand{\bM}{\mathsf{M}}

\newcommand{\bU}{\mathsf{U}}
\newcommand{\Gr}{\mathsf{Gr}}
\newcommand{\bB}{\mathsf{B}}

\newcommand{\bX}{\mathsf{X}}
\newcommand{\bK}{\mathsf{K}}
\newcommand{\bN}{\mathsf{N}}

\newcommand{\bP}{\mathbb{P}}
\newcommand{\cK}{\mathscr{K}}

\newcommand{\cT}{\mathscr{T}}
\newcommand{\cL}{\mathscr{L}}

\newcommand{\cP}{\mathscr{P}}
\newcommand{\cU}{\mathscr{U}}
\newcommand{\cQ}{\mathscr{Q}}

\newcommand{\cF}{\mathscr{F}}
\newcommand{\cM}{\mathscr{M}}

\newcommand{\cN}{\mathscr{N}}

\newcommand{\cV}{\mathscr{V}}
\newcommand{\cW}{\mathscr{W}}

\newcommand{\bpt}{\mathbf{p}}

\newcommand{\bfL}{\mathbf{L}}

\newcommand{\bmu}{\boldsymbol{\mu}}

\newcommand{\bph}{{\boldsymbol{\varphi}}}

\newcommand{\be}{\mathsf{e}}
\newcommand{\bh}{\mathsf{h}}

\newcommand{\cD}{\mathscr{D}}

\newcommand{\bS}{\mathsf{S}}
\newcommand{\bbS}{\mathbf{S}}

\newcommand{\eqdef}{\overset{\textup{\tiny def}}=}

\newcommand{\Aroof}{\widehat{A}}
\newcommand{\omhat}{\widehat{\omega}}

\newcommand{\cOh}{\widehat{\cO}} 

\newcommand{\fst}{\Bbbk(\!(t)\!)} 
\newcommand{\fTst}{\Bbbk[\![t]\!]} 
\newcommand{\pst}{(\!(t)\!)} 
\newcommand{\pTst}{[\![t]\!]}

\newcommand{\tdash}{\textup{---}}

\newcommand{\cO}{\mathscr{O}}
\newcommand{\Hd}{{H}^{\raisebox{0.5mm}{$\scriptscriptstyle \bullet$}}}
\newcommand{\bhd}{{\mathsf{h}}^{\raisebox{0.5mm}{$\scriptscriptstyle \bullet$}}}
\newcommand{\Sd}{{\bS}^{\raisebox{0.5mm}{$\scriptscriptstyle \bullet$}}}

\newcommand{\vir}{\textup{vir}}
\newcommand{\hor}{{\textup{hor}}}
\newcommand{\ver}{{\textup{ver}}}

\newcommand{\fsl}{\mathfrak{sl}}
\newcommand{\Heis}{\mathfrak{Heis}}

\newcommand{\Attr}{\mathsf{Attr}}

\newcommand{\mov}{\textup{mov}}

\newcommand{\Lbar}{\widehat{L}}

\newcommand{\buu}{^{\sqbullet}}
\newcommand{\Fix}{\mathsf{Fix}}

\DeclareMathOperator{\Coh}{Coh}
\DeclareMathOperator{\ch}{char}

\DeclareMathOperator{\Hom}{Hom}
\DeclareMathOperator{\Ext}{Ext}
\DeclareMathOperator{\Ker}{Ker}

\DeclareMathOperator{\Aut}{\mathsf{Aut}}

\DeclareMathOperator{\Lie}{Lie}

\DeclareMathOperator{\Res}{Res}

\DeclareMathOperator{\const}{const}

\DeclareMathOperator{\pt}{pt}

\DeclareMathOperator{\rk}{rk}
\DeclareMathOperator{\Pic}{Pic}
\DeclareMathOperator{\tr}{tr}

\DeclareMathOperator{\Spec}{Spec}

\DeclareMathOperator{\pFix}{Fix}

\DeclareMathOperator{\spec}{Spec}

\DeclareMathOperator{\supp}{supp}
\DeclareMathOperator{\ev}{ev}

\DeclareMathOperator{\QM}{\mathsf{QM}}
\DeclareMathOperator{\Nil}{\mathsf{Nil}}
\DeclareMathOperator{\QMN}{\mathsf{QMN}}
\DeclareMathOperator{\Maps}{\mathsf{Maps}}
\DeclareMathOperator{\Sect}{\mathsf{Sect}}

\DeclareMathOperator{\End}{End}
\DeclareMathOperator{\Image}{Image}
\DeclareMathOperator{\Thom}{Thom}
\DeclareMathOperator{\diag}{diag}

\DeclareMathOperator{\codim}{codim}

\DeclareMathOperator{\virdim}{vir\ \!dim}

\DeclareMathOperator{\Cone}{Cone}

\DeclareMathOperator{\Bun}{Bun}

\DeclareMathOperator{\Ad}{Ad}

\DeclareMathOperator{\Bl}{Bl}
\DeclareMathOperator{\CT}{CT}

\DeclareMathOperator{\Fr}{\mathsf{Fr}}

\newcommand{\Eis}{\textup{Eis}}

\newcommand{\Ld}{{\Lambda}^{\!\raisebox{0.5mm}{$\scriptscriptstyle
      \bullet$}}\!}

\newcommand{\of}[1]{\,\langle #1 \rangle}

\newtheorem{Theorem}{Theorem}

\newtheorem{Lemma}{Lemma}[section]
\newtheorem{Proposition}[Lemma]{Proposition}
\newtheorem{Corollary}[Lemma]{Corollary}

\theoremstyle{definition}

\newcommand{\cNrm}{\textup{\usefont{T1}{wela}{eb}{n}Mov}}
\newcommand{\Flag}{\textup{\usefont{T1}{wela}{eb}{n}\large Flag}}
\DeclareMathOperator{\cBun}{\textup{\usefont{T1}{wela}{eb}{sl}\large Bun\hspace{0.3mm}}}

\DeclareMathOperator{\Chars}{\textup{\usefont{T1}{wela}{eb}{sl}Chars\hspace{0.3mm}}}
\DeclareMathOperator{\Cochars}{\textup{\usefont{T1}{wela}{eb}{sl}Cochars\hspace{0.3mm}}}
\DeclareMathOperator{\AK}{\textup{\usefont{T1}{wela}{eb}{sl}AK\hspace{0.3mm}}}
\DeclareMathOperator{\SH}{\textup{\usefont{T1}{wela}{eb}{sl}\large Hecke\hspace{0.5mm}}}
\DeclareMathOperator{\Cliff}{\textup{\usefont{T1}{wela}{eb}{sl}\large Cliff\hspace{0.5mm}}}
\DeclareMathOperator{\Spin}{\textup{\usefont{T1}{wela}{eb}{sl}\large Spin\hspace{0.5mm}}}
\DeclareMathOperator{\NN}{\textup{\usefont{T1}{wela}{eb}{sl}N}}

\begin{document}

\title{L-function genera and applications} 
\author{David Kazhdan and Andrei Okounkov} 
\maketitle

\hfill \textsl{To Enrico Arbarello}

\setcounter{tocdepth}{2}
\tableofcontents

\section{Overview}

\subsection*{1.1}

Our goal in these notes is to produce an accesible, leisurely and, hopefully,
engaging introduction to the ideas, objects, and techniques that are used in our papers
\cites{KO1,KO2}. The problem studied in these papers has to do with
spectral analysis of automorphic forms, and thus fundamentally
a problem in analysis. Using well-established techniques in spectral
theory,
see Appendix \ref{s_app} for an elementary discussion, 
the sought spectral decomposition 
may be reduced to the study of certain contour integrals involving
$\zeta$-functions and more general L-functions. These special
functions of number theory appear in what may be described as the 
Eisenstein series analog of the reflection coefficients from the theory
of $1$-dimensional many-body systems. Only very basic properties of L-functions play a role in what
happens, thus reducing the problem further to an essentially combinatorial
problem of dealing with Weyl groups, roots, and many residues in
contour integrals.

This sounds like a great progress until one
realizes how much complexity the intermediate stages of
this analysis involve. Langlands devotes pages 181--195 of his
monumental treatise \cite{L} to the analysis of the basic contour
integral for the group $G_2$. In this case, the integral is
2-dimensional, corresponding to the rank of the group, the integrand
is a sum of 12
terms indexed by the elements of the Weyl group, and singularities of
the integrand are 
on 6 hyperplanes where the positive roots take value 1. The reader
will be no doubt dazzled by Langlands ability organize the computation
using carefully chosen notation, tables, etc., and to identify one
unexpected point in the spectrum. In light of subsequent insights by
Langlands, Arthur, and many other prominent mathematicians, this
unexpected point is explained by the existence of an
interesting nilpotent conjugacy class for type $G_2$, which is its own
Langlands dual $G_2^\vee = G_2$. 

At this point,
it may be worth mentioning, that the group $E_8$ has Weyl group of
order 696729600 and 120 polar hyperplanes, which intersect in
basically all possible ways inside an 8-dimensional space. By
contrast, it has only 70 nilpotent conjugacy classes
in total. It may be therefore highly desirable to have a more direct
link between spectral decomposition and nilpotent elements in the
Langlands dual Lie algebra. The first objective of this paper is
to explain in elementary terms such direct link found by us in
\cite{KO1}. 

\subsection*{1.2}
One can interpret the contour integral in question as the pairing of a
certain distribution $\Psi_{\Eis}$ with  a pair of holomorphic test
function. In this way, the spectral decomposition may be recast into
a problem of finding a special additive decomposition of the
distribution 
$\Psi_\Eis$. We solve this problem with geometric tools.

First, we interpret the distribution $\Psi_{\Eis}$ as a certain
nonproper pushforward in equivariant K-theory, or cohomology,
depending on the characteristic of the global field $\F$.  
Recall that for a group $\bH$ acting on a space $\bX$,
the equivariant K-theory $K_\bH(\bX) = K(\bX/\bH)$ is a ring over
\begin{equation}
K(\pt/\bH)=\textup{conjugation-invariant functions on
  $\bH$}\,,\label{eq:7}
\end{equation}
and proper pushforwards $\bX/\bH \to \pt/\bH$ produce elements of
\eqref{eq:7}. It should be considered well-known, see e.g.\
\cite{Atiyah}, but is nonetheless explained at length in Section
\ref{s1}, that some nonproper
pushforwards are well-defined as linear functionals on
\eqref{eq:7}, that is, as conjugation-invariant \emph{distributions} on $\bH$.
In particular, $\Psi_{\Eis}$ was identified in \cite{KO1} as a certain
equivariant \emph{genus} of the quotient stack $\cT$ from
\eqref{eq:252} below.

Since it is a genus, any decomposition of $\cT$ in a suitable 
equivariant cobordism group results in an additive decomposition
of $\Psi_\Eis$. By construction, the stack $\cT$ maps to the stack of nilpotent
elements in $\Lie (\bG^\vee)$, where the group $\bG^\vee$ is Langlands
dual to the reductive group $\bG$ for which the original
spectral problem was posed. We show in \cite{KO1} that the
additive decomposition of $\cT$ and $\Psi_\Eis$ induced by the decomposition of
$\Lie (\bG^\vee)$ into conjugacy classes
is precisely the right one for the spectral decomposition.

\subsection*{1.3}
The analysis in \cite{KO1} was about Eisenstein series induced from the trivial
automorphic form on the minimal parabolic $\bB \subset \bG$. 
The genus that appears in $\Psi_{\Eis}$ in this special case encodes the 
$\zeta$-function of the field $\F$. In the more general situation \cite{KO2},
$\zeta$-functions are replaced by L-functions.

The second objective of this note is to introduce and explain the notion of
L-genera associated to algebraic actions of Galois groups $\Gamma$ on smooth
algebraic varieties, or, more generally, virtually smooth algebraic stacks.
This occupies the bulk of Section \ref{s2}.
While L-genera may be given a number of definitions applicable in
slightly different contexts, the following one may have the
largest geometric appeal.

Let a group $\Gamma$ act on a smooth manifold $M$. By passing to
cotangent bundles if necessary, it is convenient to assume that
$\Gamma$ preserves a sympletic form on $M$.
The $\Gamma$-equivariant cobordism class of
$M$ stores a lot of data. For every closed subgroup $\Gamma'\subset
\Gamma$, it remembers the cobordism class of the fixed
locus $M^{\Gamma'}$, together with the restriction of the tangent
bundle $TM\big|_{M^{\Gamma'}}$ and the action of the normalizer of $\Gamma'$ on
both. The groups $\Gamma$ considered in this paper have a
distinguished normal subgroup $\Gamma' \triangleleft \Gamma$, called
the inertia subgroup. We denote the $\Gamma'$-fixed locus by $\Fix \subset
M$. On this fixed locus, we may consider the vector bundle formed by
\[
  T^1_\vir \Fix \eqdef H^1(\Gamma', TM\big|_\Fix) \,. 
\]
As the notation indicates, this should be viewed as a part of the
virtual tangent bundle to the fixed locus. In fact, by duality \eqref{eq:cup}
for $H^i(\Gamma',\textup{---})$, this is the only interesting part of the
virtual tangent bundle.

Multiplication in cohomology and the
symplectic form combine to give $T^1_\vir \Fix$ a nondegenerate
symmetric pairing, which we can use to define a bundle $\Cliff$ of
Clifford algebras on $\Fix$. Assuming the existence of the
corresponding self-dual spinor bundle $\Spin$, we define
\[
\bL_{\frac12}(M) = \textup{index of $\Spin$} = \chi(\Fix, \Spin_0 -
\Spin_1)\,, 
\]
where we have explicitly separated the $\Z/2$-graded bundle $\Spin$
into its even and odd parts. The index above is equivariant with
respect to the action of $\Gamma/\Gamma'$ and whatever commutes
with the action of $\Gamma$, hence defines a distribution on the
corresponding groups.

It is curious to note that even for the trivial action of $\Gamma$ on
$M$ this procedure outputs a
nontrivial genus of $M$. For instance, for a trivial action of 
$\Gamma=\pi_{1,\textup{alg}}(C)$, where $C$ is curve over a finite
field, this gives the genus corresponding to the completed
$\zeta$-function of the curve
$C$.

We believe L-genera is a useful notion beyond their use in spectral
analysis of Eisenstein series. Some potential applications are briefly
discussed in Section \ref{s3}.  

\subsection*{1.4}

Theorem \ref{t1} in Section \ref{s3} identifies the distribution
$\Psi_\Eis$ with a certain L-genus for general split unramified Eisenstein
series over a function field. It has a formal analog of number fields
which may be proven and applied without referring to the action of $\Gamma$.
While, fundamentally, it is little more
than a repackaging of classical formulas, see \cites{L, Labesse, MW}, it certainly provides a very useful angle of approach
to spectral analysis.

It is natural to explore the deeper meaning of this
identification, including its possible categorification, by which we
mean identifying the role played by the bundle $\Spin$ itself, without
passing to its K-theory class. 

We explore a basic example of
such categorification in Section \ref{s_curves}, where we deal with
Eisenstein series for the group $\bG=PSL(2)$ over the field
$\F=\Bbbk(C)$ of functions on an algebraic curve $C$. This theory
is ultimately about special divisors on $C$. While it gets
wrapped in many layers of reformulations in the process, we hope that both
its geometric origins and the geometric output contained in
Theorem \ref{t2} will appeal to Enrico's taste in mathematics,
and his passion for everything concerning
the geometry of algebraic
curves. 

Theorem \ref{t2} categorifies Theorem \ref{t1} for $PSL(2,\Bbbk(C))$
and may be also viewed as the first baby step in broadening the
fruitful interactions of between classical enumerative geometry and modern
high-energy physics in the direction of arithmetic enumeration. 
A very brief and superficial discussion of
these ideas may be found in Section \ref{s_mirror}.

In principle, the ideas presented in Section \ref{s_curves} apply to 
the Eisenstein series itself, not just its inner product with another
Eisenstein series (which is what is captured by the distribution
$\Psi_\Eis$). We find, however, that the moduli spaces that enter the
categorification of $\Psi_\Eis$, constant terms, and related
quantities are better-behaved than their fibers responsible for the values
of the
Eisenstein series at points. This may be the geometric explanation of
why it is easier to access the Eisenstein series via harmonic analysis
than directly from definitions.

\subsection*{1.5}
To our surprize and delight, following the logic of Theorem \ref{t2},
we came into a very close proximity of several classical geometric
representations theory results related to Eisenstein series. In particular, the conjecture of \cite{Semiinf2}, proven in \cite{Gaits}, expresses
the stalks of Eisenstein series for the curve $C=\bP^1$ in terms of
semi-infinite cohomology of tilting modules for quantum groups at
roots of unity. The latter can also be described as local cohomology of an irreducible perverse coherent sheaf on the nilpotent cone by \cites{Bez,BezPos}. It seems clear that the
powerful techniques developed in these and related papers
may turn out to be very handy for the eventual
generalization of our Theorem \ref{t2} to groups $\bG$ other
than $PGL(2)$.

While a technical investigation of these connections is perhaps best
left for future research, it may be worthwhile to highlight some aspects in
which our line of approach may be shifting the center of mass of the argument
from representation theory to geometry. Since we deal with the curve
$C$ of arbitrary genus, the group $H^1(C)$, with its Frobenius action,
has to play the central role for us. The descendants of
$H^1(C)$ and $H^2(C)$ generate a Clifford algebra 
correspondences, which we denote by $\cA^{\ge 1}(C)$ in this paper.
The algebra $\cA^{\ge 1}(C)$, on the
one hand, uniquely determines the modules in questions, while on
the other hand, is as close to commutative as one can get in
representation theory. In other words, one can work with $\cA^{\ge
  1}(C)$ and its modules as if they were usual objects of algebraic
geometry. 

\subsection*{1.6}
In this text, we try to intertwine several related threads on
very different levels of abstraction. While some sections of this narrative
may be rather heavy on technicalities, we hope to have provided a
sufficient number of easy reading interludes and examples to keep the
reader engaged. We have even written a completely elementary
introductory discussion of some aspects of automorphic forms for
readers lacking such background. See Appendix \ref{s_app} for that. 

\subsection*{1.7}
We would like to dedicate these modest notes to Enrico Arbarello as a
token of our great appreciation for the influence he had on algebraic
geometry, as well as on the people who work in this field and its 
applications. We imagine Enrico would agree with the point of view
that a good way to solve
a mathematical problem, be it purely theoretical or applied, is to interpret and
solve it geometrically. This is precisely what we tried to do in the context
of spectral analysis of Eisenstein series.

\subsection*{1.8}

We would like to thank many people for very interesting and enlightening
conversations about the topics discussed here. Among them, Roman
Bezrukavnikov, Alexander Braverman, Misha Finkelberg, Johan de Jong,
Mikhail Kapranov, Davesh Maulik,
Hiraku Nakajima, Nikita Nekrasov, 
Martin Olsson, Alexander Polishchuk, Will Sawin, and Leon Takhtajan.
A.O.\ thanks
the Simons Foundation for being supported as Simons Investigator.

We 
think about Igor Krichever, who's untimely passing left an enormous void
in mathematics and life for us. We cherish the memory of many related
and unrelated discussions with Igor in both recent and distant past.

\section{Distributions from geometry}\label{s1}

\subsection{Characters as distributions}

\subsubsection{}

Let $V$ be a finite-dimensional complex representation of a group $\bG$. Its
character
\[
\chi_V(g) = \tr_V g
\]
is a conjugation-invariant function on $\bG$, which satisfies
\begin{equation}
  \label{eq:128}
\chi_V
=\chi_{V_1}+\chi_{V_2}\,, 
\end{equation}
for any exact sequence
\[
0 \to V_1 \to V \to V_2 \to 0
\]
of $\bG$-modules.

\subsubsection{}

If $\bG$ is a complex reductive group, equivalently, the complexification of
a compact Lie group $\bK$, then $\chi_V$ determines $V$
completely. Indeed, the multiplicity of an irreducible representation
$V^\mu$ in $V$ is given by
\begin{equation}
  \label{eq:129}
  [V: V^\mu]
  =
  \int_{\bK} \chi_V(k) \, \overline{\chi_{\mu}(k)} \, d_\textup{Haar} k\,, 
\end{equation}
where $\chi_{\mu} = \chi_{V^\mu}$ and the Haar measure on $\bK$ is normalized to have total mass $1$.

\subsubsection{}

Fundamental to study of infinite-dimensional representations is the
idea that the character of a infinite-dimensional representation $V$,
while not defined as a function, may be well-defined as a
distribution so that, in particular, the equality \eqref{eq:129} is
preserved as a pairing
\begin{equation}
  \label{eq:129b}
  \left\langle \chi_V ,\chi_{\mu} \right\rangle = \left[V: \left(V^{\mu}\right)^*\right]\,, 
\end{equation}
of a test function $\chi_{\mu}$
with the distribution $\chi_V$.

To avoid a possible confusion, we should stress that 
in this paper we are not interested in traces of unitary and
other truly infinite-dimensional representations of reductive
groups. Representations that we consider are direct sums of
finite-dimensional irreducible representations with finite
multiplicities. In particular, test functions $\chi_{\mu}$ are
analytic functions on $\bG$.

\subsubsection{}

For example, if $V$ is the regular representation $\C[\bG]$ of a complex
reductive group $\bG$ then the decomposition 
\[
\C[\bG] = \bigoplus_{\textup{irreducible $V^\mu$}} \End(V^\mu) 
\]
gives $[\C[\bG]: V^\mu]=\dim V^\mu = \chi_{\mu}(1)$. Thus
\begin{equation}
\chi_{\C[\bG]}=
\delta_1\label{eq:131}\,, 
\end{equation}
where $\delta_1$ is the $\delta$-function at the identity $1\in \bG$.

\subsubsection{}
More generally, the multiplicity
\[
[V: V^\mu ] = \dim \Hom_\bG(V^\mu, V)\,, 
\]
which we assume to be finite, 
may be upgraded to a representation of the centralizer of $\bG$ in
$GL(V)$. If $\bG \times \bH$ acts on $V$ and
$(g,h) \in \bG \times \bH$ then 
\begin{equation}
  \label{eq:129bb}
  \left\langle \chi_V(gh) ,\chi_{\mu^*} \right\rangle =
  \tr_{\Hom_\bG(V^\mu, V)} h \,, 
\end{equation}
defines a linear functional on $\C[\bG]^{\Ad}$ with values in $
\C[\bH]^{\Ad}$.

For example, the group $\bG \times \bG$ acts in $\C[\bG]$ by left
and right shifts and
\[
\left\langle \tr_{\C[\bG]} g \times h, \chi_{\mu}(g) \right
\rangle = \chi_{\mu}(h) \,. 
\]
Since we pair only with conjugation-invariant functions of $g$, we get
\begin{equation}
\tr_{\C[\bG]} g \times h = \delta_{g \sim h}\,, \label{eq:132}
\end{equation}
where $g \sim h$ means that $g$ and $h$ have the same
eigenvalues, that is, the same image in $\Spec \C[\bG]^{\Ad}$. In
other words, $\delta_{g \sim h}$ is the delta-function at the point
$h$ in $\Spec \C[\bG]^{\Ad}$. 
This generalizes \eqref{eq:131}. 

\subsubsection{}

In the example \eqref{eq:132},
one can observe the following general phenomenon. If $h$ lies in a
fixed neighborhood $\cU_\bH$ of a maximal compact subgroup of $\bH$ then
the traces \eqref{eq:129bb} grow with a fixed exponential rate as
function of the highest weight $\mu$. Therefore, $\chi_V(gh)$ may be
paired with test functions $f(g)$ analytic in a sufficiently large
neighborhood $\cU_\bG$ of $\bK \subset \bG$. The sizes of $\cU_\bH$
and $\cU_\bG$ 
are correlated in general and simply equal in the particular
example \eqref{eq:132}.

In number-theoretic
applications, $\cU_\bH$ is bounded in terms of the data like the
cardinality $q$ of a finite field etc. We will thus use the word 
\emph{distributions} to describe linear functionals on functions
analytic in a certain fixed neighborhood $\cU_\bG$ of $\bK \subset
\bG$.

A typical example of such linear functional for $\bG = (\Ct)^r \owns z$ is
\begin{equation}
f(z) \mapsto \int_{\gamma} f(z) \, w(z,h) \, \prod \frac{dz_j}{2\pi i
  z_j} \label{eq:133}
\end{equation}
where $w$ is meromorphic in $z$ and
$\gamma\in H_r(\bG\setminus \{\textup{poles of $w$}\})$. In
particular, $\delta$-functions and their derivatives can be written in
this form by the Cauchy integral formula. 
Clearly, in this example, $\cU_\bG$ has to be large enough to contain a
representative of $\gamma$.

\subsubsection{}

It is possible and important to also consider the generalizations in
which the trace \eqref{eq:129bb}, while being infinite,
converges for a given $h\in \bH$. This means that the multiplicities
of irreducibles for $\bG \times \bH$ are finite, and furthermore
satisfy certain bounds in the $\bH$-direction.

\subsection{Euler characteristics as distributions}

\subsubsection{}

The regular representation of $\bG$ in
\[
\C[\bG] = H^0(\cO_\bG) = \Hd(\cO_\bG) \,, 
\]
where $\cO_\bG$ is the sheaf of functions on the algebraic variety $\bG$,
belongs to the very broad class of representations arising in the
cohomology groups 
$H^i(\bX,\cF)$, where $\bX$ is a scheme or an algebraic stack with the
action of $\bG$ and $\cF$ is a $\bG$-equivariant coherent sheaf on
$\bX$.

Informally, one can say that we replace a finite-dimensional 
representation of $\bG$ by a family $\cF_x$ of vector spaces (in generally, of
varying dimension) parametrized by the points of $\bX$, on which the
group $\bG$ acts by linear isomorphisms $\cF_{x} \xrightarrow{\,
  \sim\, } \cF_{gx}$. Taking $\bX$ to be a point recovers the previous
setup.

\subsubsection{}
In order to maintain the basic additivity \eqref{eq:128} over short
exact sequences
\begin{equation}
0 \to \cF_1 \to \cF \to \cF_2 \to 0\,, \label{eq:138}
\end{equation}
one has to combine the individual cohomology groups into the Euler characteristic
\begin{equation}
\chi(\cF) = \sum_{i} (-1)^i H^i(\cF)\,,\label{eq:134}
\end{equation}
which is a virtual representation of $\bG$. We will identify this
representation with its characters, viewed as a distribution on
$\bG$. Here we assume that the multiplicities of irreducibles in $\chi(\cF)$
are finite.

\subsubsection{}

Despite the fact that we are only interested in
the Euler characteristic $\chi(\cF)$, for the technical development of
the subject it is convenient to assume that all groups $H^i(\bX,\cF)$
have finite multiplicities of irreducibles for all $\bG$-equivariant
coherent sheaves $\cF$. Since we can always replace $\cF$  by
a tensor product $\cF \otimes V^\mu$, this is equivalent to
requiring that 
\[
\dim H^i(\bX,\cF)^{\bG} <
\infty 
\]
for all $i$ and all $\cF \in \Coh_\bG(\bX)$. This, in turn, can be
phrased as saying the the pushforward map
\begin{equation}
\bX/\bG \to \pt \,, \label{eq:135}
\end{equation}
from the quotient stack $\bX/\bG$ to the point takes complexes of
coherent sheaves to complexes with coherent cohomology. By definition,
this means that the map \eqref{eq:135}, equivalently the stack
$\bX/\bG$ itself, is \emph{cohomologically
  proper} \cite{DHLloc}.

Cohomological properness is a very convenient technical assumption satisfied by
many practically important examples. It may be thought as an algebraic
extension of the notion of \emph{transversally elliptic operators}
introduced and studied by Atiyah \cite{Atiyah} and many authors since.
As already observed by Atiyah, the index of a transversally elliptic
operator is naturally a distribution, of which
$\chi(\cF)$ is a direct algebraic analog.

\subsubsection{}
Evidently, the map $\cF \to \chi(\cF)$ factors through
the K-theory $K_\bG(\bX)$ of $\bG$-equivariant coherent sheaves on
$\bX$, which is a module over
\[
K_\bG(\pt) = \Z[\bG]^{\Ad} \,.
\]
For concrete spaces $\bX$ that we have in mind, there is no
difference between algebraic and topological K-theory. In any
case, one can always pass to the \emph{numerical} K-theory, in which
one quotients out by classes $\cF$ such that $\chi(\cF\otimes \cV)=0$
for any locally free sheaf $\cV$.

\subsubsection{}

The pairing of $\chi(\cF)$ with the test functions
$\chi_\mu \in K_\bG(\pt)$ may be phrased in terms of the diagram
\begin{equation}
  \label{eq:136}
  \xymatrix{
    \bX/\bG \ar[dr]^{p_2} \ar[d]_{p_1}\\
    \pt / \bG \ar[r] & \pt
}
\end{equation}
as follows
\begin{equation}
  \label{eq:137}
  \left \langle \chi(\cF), \chi_\mu \right \rangle = p_{2*} (\cF
  \otimes p_1^*(\chi_\mu))\,. 
\end{equation}
In this formulation, $\bX/\bG$ may be replaced by any other stack
mapping to $\pt / \bG$ so that the composed map $p_2$ is
cohomologically proper. In the presense of a commuting action of
a group $\bH$, the pairing \eqref{eq:137} takes values in
$K_\bH(\pt)$. 

\subsubsection{}
Note that the pairing \eqref{eq:137} is defined over $\Z$ regardless
of the field of definition of $\bX$. To apply continuity and other
analytic arguments to distributions, it is convenient to extend
the scalars to
\[
K_\bH(\pt) \otimes_\Z \C = \C[\bH]^{\Ad} \,.
\]
As to the field of definition of $\bX$, it is convenient to assume
it has characteristic zero,
but it may very well be the algebraic closure $\overline{\Q}_\ell$ of
the field 
$\Q_\ell$ of $\ell$-adic rationals, the typical ground field in the study of Galois
representations.

We also assume that
$\bG$ is a split reductive group over the same ground field, although the
reader should bear in mind that $\bG$ in the ongoing discussion is
different from reductive groups considered in the automorphic and
Galois representation contexts. For our concrete applications, the we will have
$\bG=\bA \times \bA$, where $\bA$ is a maximal torus is the
centralizer of a Galois representation, $\bX$ will be the fixed locus
of the inertia subgroup, and 
$\bH$ will be the reductive group
generated topologically by the Frobenius automorphism. 

\subsection{Scissor relations}\label{sec_sciss}

\subsubsection{}

In addition to being additive over short exact sequences
\eqref{eq:138},
Euler characteristics $\chi(\bX,\cF)$ are also additive over $\bX$ in
the following sense.
Given a $\bG$-stable closed $\bY \subset \bX$, there is
a long exact sequence
\begin{equation}
  \label{eq:139}
  \dots \to H^i_\bY(\bX,\cF) \to H^i(\bX,\cF) \to H^i(\bU,\cF) \to H^{i+1}_\bY(\bX,\cF) \to \dots \,, 
\end{equation}
in which $\bU = \bX \setminus
  \bY $ and the groups $H^i_\bY(\bX,\cF)$ are \emph{local cohomology}
groups. These are algebraic analog of the relative cohomology and
one should view the long exact sequence \eqref{eq:139} as coming from
a distinguished triangle of spaces
\begin{equation}
  \label{eq:130}
  \xymatrix{ && \bX  \ar[dr] \\
    \bU \ar[urr]^{\iota_\bU} & & & 
   \bX/\bU=\Cone(\iota_\bU) \ar[lll]_{[1]}\,.} 
\end{equation}
Local cohomology is easiest to compute when $\bX$ and $\bY$ are
smooth, in which case $\Cone(\iota_\bU)$ may be identified with the Thom space
$\Thom(\bY \to \bX)$ of the embedding of $\bY$ into its normal
bundle.

\subsubsection{} 

It is convenient to write, by definition, 
\[
\chi(\Thom(\bY \to \bX), \cF) = \sum (-1)^i H^i_\bY(\bX,\cF) \,, 
\]
and one can, in principle, declare the Cech complex that computes
the local cohomology of $\cO_\bX$ to be the structure sheaf of $\Thom(\bY \to
\bX)$. 
It follows from \eqref{eq:139} that if any two among $\bX/\bG$,
$\bU/\bG$, and $\Thom(\bY \to \bX)/\bG$ are cohomologically proper
then so it the third.

Tautologically, we have
\[
\chi(\bX,\cF) = \chi(\bU,\cF) + \chi(\Thom(\bY \to \bX), \cF)\,,
\]
which means that from the point of view of Euler characteristics,
$\bX$ is the sum of $\bX \setminus \bY$ and $\Thom(\bY \to \bX)$. It
is
therefore convenient
to introduce equivalence classes $[\bX]$ of schemes or stacks that satisfy
\[
\left[\bX\right] = \left[\bX \setminus \bY\right] + \left[\Thom(\bY \to \bX)\right]
\,.
\]

\subsubsection{Example}

Many key features of the general story are already clearly visible in
the following simplest example. Let $z\in \Ct=\bG$ act on
$\bX = \C^n$ by multiplication by $z$. Then $z$ acts by multiplication
by $z^{-d}$ on homogeneous polynomials of degree $d$, and therefore
\begin{equation}
  \label{eq:140}
  \chi(\cO_\bX) = 1 + n z^{-1} + \binom{n+1}{2} z^{-2}+ \dots =
  (1-z^{-1})^{-n} \,, \quad |z|> 1\,. 
\end{equation}
Equivalently,
\begin{equation}
  \label{eq:141}
  \left\langle\chi(\cO_\bX),
  \phi(z) \right\rangle = \oint_{|z|=c > 1}
  \frac{\phi(z)}{(1-z^{-1})^n} \, \frac{dz}{2\pi i z} \,. 
\end{equation}
The fact that \eqref{eq:140} converges for $|z|>1$ and that,
correspondingly, the integral in \eqref{eq:141} goes outside of the
pole at $z=1$ can be traced to the fact that $\Ct$ contracts $\bX$ to
a proper variety (here, the point at the origin) as $z\to 0$.

\subsubsection{Example, continued}\label{s_examp2} 

Now take $\bY= \{0\} \subset \bX$. The Cech complex computing the
local cohomology is the product of the elementary complexes
\begin{equation}
\C[x_i] \to \C[x_i^{\pm 1}]\label{eq:142}
\end{equation}
over all coordinates $i=1,\dots,n$. The map in \eqref{eq:142} is the
obvious embedding. We conclude
\[
H^i_\bY(\cO_\bX) =
\begin{cases}
  \left(\prod x_i^{-1} \right) \C[x_1^{-1},\dots,x_n^{-1}] \,, \quad &i =n \,,\\
  0 \,, \quad & \textup{otherwise} \,. 
\end{cases}
\]
Therefore
\begin{alignat}{2}
\chi(\Thom(0\to \C^n),\cO) &= (-1)^n\cdot\left(
                             z^n+ n z^{n+1} + \binom{n+1}{2} z^{n+2}+ \dots\right)&&
  \notag \\
&=
(1-z^{-1})^{-n}\,, && \quad |z|<1 \,.\label{eq:143}
\end{alignat}
The equality of \eqref{eq:140} and \eqref{eq:143} as
rational functions follows from localization theorem in equivariant
K-theory. The fact that now the trace converges for $|z|<1$ can
again be traced to the fact that the normal bundle $N_{\bX/\bY}$ is
has positive $z$-weights.

\subsubsection{Example, continued}

The action of $\bG=\Ct$ on $\bU = \C^n \setminus \{0\}$ is free and
\[
\bU/\bG = \bP^{n-1} \,.
\]
From definitions, we have
\begin{equation}
  \label{eq:230}
  \langle \chi(\C^n \setminus \{0\},\cO), z^{m} \rangle = \chi(\bP^{n-1}, \cO(m))\,,
\end{equation}
which we can rephrase as
\begin{equation}
  \label{eq:230bis}
  \langle \chi(\C^n \setminus \{0\},\cO), \phi(z) \rangle = \chi(\bP^{n-1},
  \phi(\cO(1)))\,. 
\end{equation}
In \eqref{eq:230bis}, $\phi(\cO(1))$ is computed from a
Taylor series of $\phi$ at $z=1$. This is well-defined
due to the relation $(\cO(1) -1)^{n}=0$ in $K(\bP^{n-1})$ and
makes \eqref{eq:230bis} a distribution supported at $z=1$.

\subsubsection{Example, summarized}\label{ex_summ}

We see that the decomposition 
\begin{equation}
  \label{eq:144}
  [\C^n/\Ct] = [ \bP^{n-1}] + [(\Thom(0\to \C^n)/\Ct] \,. 
\end{equation}
amounts to the
equality
\begin{equation}
  \label{eq:248}
  \oint_{|z|=c > 1}
  \frac{\phi(z)}{(1-z^{-1})^n} \, \frac{dz}{2\pi i z}  =
  \chi(\bP^{n-1},
  \phi(\cO(1))) + \oint_{|z|=c' < 1} (\textup{same as in LHS})   \,. 
\end{equation}
%
%
This means that the scissor relation \eqref{eq:144} replaces
picking up residue while moving the contour of
integration.

\subsubsection{}

Of course, the operation of picking up residues while moving the contour of
integration appears all over mathematics and natural sciences. In
particular, a very popular approach to spectral decompositions,
including spectral decomposition of Eisenstein series, uses this step
in an essential way. See, for instance,  Appendix \ref{s_app} for an elementary
example.

 The advantages of lifting an additive relation between distributions
 to an additive relation between spaces becomes
 particularly apparent for groups of rank more than one.
Distributions of the form \eqref{eq:133} coming from
 $\chi(\bX,\cF)$ have the schematic form 
 \begin{equation}
 w(z,h) \, dz = \frac{p(z,h)}{\prod_i (1-h^{\alpha_i} z^{\beta_i})}
 \prod \frac{dz_j}{2\pi i z_j} \,, 
\label{eq:146}
\end{equation}
where $p(z,h)$ is a polynomial and $z^{\beta_i}$
and
$h^{\alpha_i}$ are weights of $\bG$ and $\bH$, that is,
characters of their maximal tori. Each 
exponent in \eqref{eq:146} is a integer vector of dimension
equal to the rank of the corresponding group.

The monomials $z^{\beta_i} h^{\alpha_i}$
are the weights of the
maximal torus of $\bG \times \bH$ acting in the tangent
space to the fixed locus $\bX^{\bG \times \bH}$. This fixed locus is
often 0-dimensional, which means that is has
$\chi_{\textup{top}}(\bX)$
many components and each contributes $\dim \bX$
terms to the denominator. The weight $w(z,h)$ has
usually a much simpler conceptual form when written as a sum over
fixed loci, but this sum can very large indeed (think order of the
Weyl
group many terms any time we see a flag variety) and have many more
poles than the collected expression.
All this leads to a combinatorial explosion in the
group $H_{\textup{middle}}(\bG \setminus
\{\textup{poles of $w$}\})$, which in turn 
results in unwieldy outputs from the
direct application of the residue formula. The eventual dramatic
simplification of these outputs appears miraculous in such
approach.

Scissor relations, in general, cross many poles at once and
automatically account for a multitude of geometric
relations satisfied by distributions of the form $\chi(\bX,\cF)$.
In addition to the additivity in both $\cF$ and $\bX$,
these relations include the basic principle that
\begin{equation}
  \label{eq:147}
  \textup{$\supp \cF$ is proper} \quad \Rightarrow
  \quad \textup{$\chi(\cF)$ has no singularities}\,, 
\end{equation}
which in practical terms amounts to cancellation of all poles in
rather complex rational functions that give an equivariant
localization computation of $\chi(\cF)$.

\subsection{Remarks}

\subsubsection{}

A given $\bX$ can be cut in pieces in many different ways and
it is natural to ask if there is some guiding principle to finding 
good decompositions.

If $\bG$ preserves an ample line bundle $\cL$ on $\bX$ then one
can use geometric invariant theory to partition $\bX$ into
the unstable and the semistable locus (which is what we did
in Example \eqref{s_examp2}), and then further
partition the unstable locus into various strata according
to their degree of instability. Such stratification has been
introduced and used by many authors in slightly different
settings, see Chapter 5 in \cite{VinPop}  for a textbook
exposition.

Different choices of $\cL$ produce interesting linear
relations between distributions. For instance, in the example
above, the unstable locus can be the origin or all of $\C^n$,
depending on the weight of $\Ct$-action on $\cL=\cO_{\C^n}$.
However, what we have found especially illuminating, is
to compare the GIT stratifications for $\bG$ and $\bH$.

\subsubsection{}

One can study the diagram \eqref{eq:136} in multiplicative cohomology
theories $\bhd(\tdash)$ other than equivariant K-theory.
The cohomology $\bhd(\pt / \bG)$ becomes the
ring of test functions for the distributions $p_{1,*} \cF$, 
$\cF \in \bhd(\bX / \bG)$.

The key ingredient of $\bhd(\tdash)$ is $\spec \bhd(\pt/\Ct)$, which
is a commutative algebraic group since $\pt/\Ct$ is a group up to
homotopy. What we have found in \cite{KO1} is that the natural choice for
this group 
in the automorphic context is
\begin{equation}
\spec \bhd(\pt/\Ct) = \Hom (\cQ_\F, \Ct) \,,\label{eq:258}
\end{equation}
where $\cQ_\F$ is the image of the norm map $\| \, \cdot \,\|$ on the
adeles of $\F$. Concretely,  
\begin{equation}
  \label{eq:257} 
  \cQ_\F=\textup{image of $\| \, \cdot \,\|$ in $\R_{>0}$}
  =
  \begin{cases}
     q^{\Z} \,, & \ch F >0\,, \\ 
    \R_{>0}\,, & \ch \F = 0 \,. 
  \end{cases}
\end{equation}
Here $q$ is the cardinality of the constant subfield $\Bbbk \subset
\F=\Bbbk(C)$ and the number field case may be viewed as the 
$q\downarrow 1$ limit. Formula \eqref{eq:258} means that 
$\bhd(\tdash)$ is the equivariant K-theory for
a global function field and the equivariant cohomology
for a global number field. 

\subsubsection{}

As linear $q$-difference equations are growing in importance in
various areas of mathematics and theoretical physics, so does
the interest in integral solutions of these equations.

Formally, an integral is a pushforward of some other $q$-difference
module, and in this way one aims to describe an interesting $q$-difference
module as an image of a simpler one. The simplest possible
linear $q$-difference equation is a rank $1$ equation
\begin{equation}
f(qz) = r(z) f(z) \,,\label{eq:250}
\end{equation}
where $r(z)$ is a rational function of the variable $z\in \C$. The
equation \eqref{eq:250} may be solved explicitly in terms of
the $q$-analog
\[
\Gamma_q(z) = \frac1{(1-z)(1-q z) (1-q^2 z) (1-q^3 z) \dots} \,,
\quad |q|< 1\,,
\]
of the $\Gamma$-function. The product converges for all $z \ne
1, q^{-1}, q^{-2}, \dots$ and solves
\[
\Gamma_q(qz) = (1-z) \Gamma_q(z) \,.
\]
One notes that
\begin{equation}
\Gamma_q(z) = \chi(\cO_{\bX})\,,
\quad \bX=\Maps(\C\otimes q \to \C \otimes z^{-1})\,,\label{eq:23}
\end{equation}
where we consider polynomial maps (that is, just polynomials in this
case), and the group $\left(\Ct\right)^{2} \owns (q,z)$ act in the
source and target of these maps with 
indicated weights. Here and in what follows we denote
by $\otimes q$, $\otimes z^{-1}$, etc.,  the twist of an action on a linear space by
a character of the group. 

It is easy to guess from \eqref{eq:23} that one can solve 
interesting $q$-difference equations via K-theoretic
computations on 
\[
\bX=\Maps(\C\otimes q \to \textup{some target})\,, 
\]
which is indeed the case, see for instance \cite{O2}. There are
natural stratifications of the spaces of maps coming from
a stratification of their target or from some notion of stability
for maps, and all these play a certain role in the development
of the subject.

\section{$\zeta$-function and $L$-function genera} \label{s2}

\subsection{Genera}

\subsubsection{}

The dependence of $\chi(\bX,\cF)$ on two separate arguments is 
artificial as, with minimal assumptions about $\bX$, every such
quantity may be naturally expressed as a linear
combination of $\chi(\cO_{\bY})$, where $\bY\subset \bX$ is an
invariant closed subvariety or a substack. These satisfy a version of scissor
relations from Section \ref{sec_sciss}, as well as relations
imposed by equivariant deformations and flat limits.

\subsubsection{}

If $\bX$ is smooth, one can put $\chi(\cO_{\bX})$ into the richer
framework of \emph{genera} as follows. Let $\psi(x)$ be a
function of one variable $x$, which we will take to be
a Laurent polynomial with coefficients in a ring $R$ to start. Denoting by $x_1,\dots,x_n$, where
$n=\dim \bX$, the Chern roots of the tangent bunde $T\bX$, we
can define
\[
\psi(T\bX) = \prod_{i=1}^n \psi(x_i) \in K_\bG(\bX) \otimes_\Z R \,. 
\]
Recall that Chern roots of a bundle $\cV$
or rank $r$ are the coordinates in the target of the map
\[
\Spec K_\bG(\bX) \to S^r \Ct = \Spec K(\pt/GL(r))\,,
\]
induced on K-theories by the map $\bX \to \pt/GL(r)$ that
classifies $\cV$.


\subsubsection{}
For cohomologically proper $\bX$, we define 
\[
\psi(\bX) = \chi(\psi(T\bX)) \in \Hom(K_\bG(\pt),R) \,.
\]
In addition to the obvious deformation invariance,
these are preserved by an arbitrary $\bG$-equivariant bordism.
Further, we have
\begin{equation}
  \label{eq:148}
  \psi(\bX_1 \times \bX_2) = \psi(\bX_1) \times \psi(\bX_2)
\end{equation}
provided one of the factors is proper, which means that in the right-hand side
of \eqref{eq:148} we have a product of a function and a distribution.
We also have 
\begin{equation}
  \label{eq:149}
  \psi(\bX) = \psi(\bX \setminus \bY) + \psi(\Thom(\bY \to \bX)) \,, 
\end{equation}
when $\bY$ is smooth and either (equivalently, both) spaces on
the right are cohomologically proper. 

\subsubsection{}
Many spaces in algebraic geometry, while not smooth, have a
virtual\footnote{While the literature on virtual
  tangent bundles and virtual structure sheaves is really vast and
  advanced, old references like \cite{Fan} or Section 5.3 in
  \cite{OP} may still be useful as
  an introductory discussion of the subject.} 
 tangent bundle $T_\vir \bX \in K_\bG(\bX)$. For example,
$\bX$ could be cut out by a section of a vector bundle $\cV$ in
some ambient smooth space $\bX'$, in which case
\[
T_\vir \bX = \left(T \bX' - \cV\right) \Big|_{\bX} \,.
\]
Since
\begin{equation}
\psi(\cV_1 - \cV_2) = \psi(\cV_1) \big/ \psi(\cV_2)\,, \label{eq:151}
\end{equation}
the quantity $\psi(T_\vir\bX)$ is a rational function on
the spectrum of $K_\bG(\bX)$. To push it forward to a distribution, 
we need to make sure it is regular on a appropriate open set. It is
convenient to discuss this issue in the context of analytic K-theory
classes. 

\subsubsection{}\label{s_AK}

In general, let 
\[
\cU \subset \Spec K_\bG(\pt) \otimes_\Z \C= \textup{maximal torus} \big/
\textup{Weyl group} 
\]
be an open neighborhood of the image of the compact torus,
and let
\[
\AK_\bG(\pt) \supset K_\bG(\pt)\otimes_\Z \C 
\]
be some algebra of analytic
functions on $\cU$. The notation is short for analytic K-theory
classes. The projection $K_\bG(\pt) \to \Z$
onto $\bG$-invariants extends to the map $\AK_\bG(\pt) \to \C$ given by
the Weyl integration formula. This gives a natural embedding 
\[
\AK_\bG(\pt)  \hookrightarrow \Hom(\AK_\bG(\pt) , \C)
\]
of functions in distributions.

For any $\bG$-space $\bX$, we can define
\begin{equation}
  \label{eq:150}
 \AK_\bG(\bX) = K_\bG(\bX) \otimes_{K_\bG(\pt)} \AK_\bG(\pt) \,. 
\end{equation}
Since push-forward is a $K_\bG(\pt)$-linear operation, these push
forward just like ordinary K-theory classes. That is, if $\cF \in
\AK_\bG(\bX)$ then $\chi(\cF) \in \AK_\bG(\pt)$ if $\bX$ is proper and
\[
\chi(\cF) \in \Hom(\AK_\bG(\pt) , \C)
\]
if $\bX$ is cohomologically proper. In particular,
\[
  \psi(\bX) = \chi(\psi(T_\vir \bX))
\]
is a well-defined linear functional on $\AK_\bG(\pt)$ if $\psi(T_\vir
\bX)$ is regular on the preimage of $\cU$.



\subsection{$\zeta$-function genera}

\subsubsection{}
This paper is about certain topological
computations which have applications to arithmetic
questions. Interestingly, we use very little arithmetic information 
as an input, making our results applicable to genera associated with
arbitrary functions 
that share a few basic properties with $\zeta$-functions and
L-functions. 
Since, at the moment, we don't see any application beyond the arithmetic
situation, we will refer to a function like \eqref{eq:269} below as a
$\zeta$-function of some curve $C$, even though the actual proof
in \cite{KO1} only uses the fact that $\alpha_i$ come in complex
conjugate pairs and that $1< |\alpha_i| < q$.

\subsubsection{}

In this paper, will
be concerned with algebraic actions of groups $\Gamma$ of the
form 
 \begin{equation}
   \label{eq:256}
   1 \to \Gamma' \to \Gamma \to 
   \cQ_\F \to 1  
 \end{equation}
 on algebraic varieties over $\bbf$. Here $\bbf=\C$ in the number
 field case, while in the function field case $\bbf$ can be any algebraically
closed field of characteristic zero.  The group $\cQ_\F\subset \R_{>0}$
was defined in \eqref{eq:257}. It is  formed
by the norms of the adeles of $\F$.

 The exact sequence \eqref{eq:256} makes
$\Gamma$ look like a Weil group of a global field. The only
further arithmetic input that we will need to define L-genera is the duality pairing 
\begin{equation}
  \label{eq:cup}
  H^i(\Gamma',\pi)\otimes H^{2-i}(\Gamma', \pi^\vee) \xrightarrow{\quad
    \cup \quad } H^{2}(\Gamma', 1_\bbf) = 1_\bbf \otimes \hbar 
\end{equation}
of $\cQ_\F$-modules, where $\pi$ is a representation of $\Gamma$ over
$\bbf$, $1_\bbf$ denotes the trivial representation, and
\begin{equation}
  \label{eq:75}
   \hbar: \Gamma/\Gamma' \to GL(1,\bbf)
\end{equation}
is the tautological character. Concretely
\begin{equation}
  \label{eq:75_}
  \hbar(\Fr) =q = 
  \begin{cases}
    \textup{cardinality of the constant field}\,, & \ch F > 0 \,, \\
    1 \,, & \ch F = 0 \,, 
  \end{cases}
\end{equation}
for any element $\Fr \in \Gamma$ that projects to the generator of
$\cQ_F$ in the function field case and to the canonical infinitesimal
generator $1\in \Lie \cQ_\F$  in the number field. To save on
notation, after these introductory pages, we will identify the
character $\hbar$ with its value $q$ on the generator of $\cQ_F$.
This is harmless since
all equivariant computations are in the end evaluated at an element
$\Fr\in \Gamma$ as above.

By definition, an action of $\Gamma$ on an algebraic variety $\bX$
over $\bbf$ factors through a homomorphism from $\Gamma$ to
a linear algebraic group over $\bbf$ acting on $\Aut(\bX)$, see the
diagram \eqref{eq:263} below. Without loss of generality, we may
assume that the character $\hbar$ extends to a character of the
algebraic group in question and denote it by the same letter.



\subsubsection{}
After these preliminaries, we recall that
a $\zeta$-function of the field $\F = \Bbbk(C)$ is obtained by
substituting $t=q^{-s}$ into
\begin{equation}
\zeta_C(t) = \prod_{x\in C}  (1-t^{\deg x})^{-1} =
\frac{\prod_{i=1}^g (1-t \alpha_i)(1- t q/\alpha_i)}{(1-t)(1-tq)} \in \Z[[t]]
\,, \label{eq:269}
\end{equation}
where $\{\alpha_i,q/\alpha_i\}$ are pairs of eigenvalues of the Frobenius
action on $H^1(C)$. A slight modification
\begin{equation}
\xi_C(t) = (q^{1/2}t)^{1-g} \zeta_C(t)  \label{eq:269b}
\end{equation}
satisfies the functional equation
\begin{equation}
  \label{eq:270}
  \xi_C(q^{-1}t^{-1}) = \xi_C(t) \,. 
\end{equation}
The two poles in \eqref{eq:269} and the functional equation
\eqref{eq:270} suggest that $\xi_C(t)$ can be interpreted as an equivariant
genus of a 2-dimensional vector space on which
a group acts with weights $\{t, q^{-1} t^{-1}\}$. 
Concretely, let 
\[
 \bY=\A^1\otimes t\,,
\]
be a one-dimensional representation of $\bG$ with
character $t$ and trivial action of $\bH=GL(1,\bbf)$.  Let
$\hbar: \Gamma \to \bH$ be the tautological character as before.
We consider 
\[
\bX = T^*\bY
\]
and let the group $\bH$ act on it by scaling the cotangent fibers with
weight $-1$. The genus 
\[
\psi(\bX) \Big|_{\{\Fr\} \times \bG}= \frac{\psi(t) \psi(q^{-1}t^{-1})}{(1-t^{-1})(1-qt)}\,,
\]
has the same poles and the same symmetry as $\xi_C(t)$. 
We can thus find a Laurent polynomial $\psi_\F(t)$ such that 
\begin{equation}
  \label{eq:249}
  \psi_\F(\bX) \Big|_{\{\Fr\} \times \bG} = \xi_C(t) \,. 
\end{equation}
Note that while the polynomial $\psi_\F(t)$ involves splitting
Frobenius eigenvalues and taking 
square roots, these ambiguities cancel out of 
\eqref{eq:249} and its generalization \eqref{eq:271} below. 

A parallel definition works for number fields in cohomology once
one introduces analytic equivariant cohomology  classes as
in Section \ref{s_AK}.

\subsubsection{}

More generally, suppose that $\bY$ is smooth but otherwise
arbitrary and consider $\bX = T^*\bY$. Let $\bG$ and $\bH$ act on
$\bY$ and let us twist the induced action on $\bX$ by a character 
\[
\hbar^{-1}: \bH \to GL(1,\bbf) 
\]
that scales all cotangent directions. This scales the
canonical symplectic form on $\bX$ with weight $\hbar$.

The Chern roots of $T\bX$ form pairs of the form
$\{y_i, \hbar^{-1} y_i^{-1}\}$, where $y_i$ are the Chern
roots of $T \bY$. Clearly, $\psi(\bX)$ depends only on
the combination $\psi(y)\psi(\hbar^{-1} y^{-1})$ which is
a function symmetric with respect to $y \mapsto \hbar^{-1} y^{-1}$.
Identifying $\Fr$ with its image in $\bH$, let 
$\Fr\in \bH$ be an element such that $\hbar(\Fr)=q$. 
We define
\begin{equation}
  \label{eq:271}
  \xi_\F (\bX ) = \psi_\F(\bX) \Big|_{\{\Fr\} \times \bG}  \,. 
\end{equation}
This defines a  distribution on $\bG$ provided
the traces of $\Fr$ in the $\bG$-multiplicity spaces converge as
in \eqref{eq:129bb}.

\subsubsection{}

The essense of our paper \cites{KO1} may be described as 
the following link
between genera and spectral problems. Let $\bG$ be a
split reductive group and let $\bG^\vee$ be the Langlands
dual reductive group. We have the Springer resolution
\begin{equation}
  \label{eq:251}
  T^*(\bG^\vee/\bB^\vee)  \to \textup{Nilpotent elements $\Nil^\vee$}
 \subset \Lie(\bG^\vee) \,, 
\end{equation}
where $\bB^\vee$ is the Borel subgroup. One can take
the symplectic reduction of \eqref{eq:251} by $\bB^\vee$
and arrive at the quotient stack $\cT$ with two
canonical maps 
\begin{equation}
  \label{eq:252}
  \xymatrix{
    & \cT = T^*(\bB^\vee \backslash \bG^\vee / \bB^\vee) \ar[dr]^{p_2}
    \ar[dl]_{p_1}\\
    \pt/(\bA^\vee \times \bA^\vee) && \Nil^\vee/ \Ad \bG^\vee \,. }
\end{equation}
Here $\bA^\vee$ is the maximal torus of $\bG^\vee$,
and the canonical map $\bB^\vee \to \bA^\vee$ is the quotient by
the unipotent radical. Since 
\[
K(\pt/(\bA^\vee \times \bA^\vee))= \Z[\bA^\vee \times \bA^\vee] \,, 
\]
the pushforward 
\[
\Psi_{\xi_\F} = p_{1,*} \xi_\F(\cT)
\]
defines a distribution on two copies of $\bA^\vee$.

\subsubsection{}\label{s_Eis_spec1}

The torus $\bA^\vee$ is the space of parameters for the spherical
Eisenstein 
series $\Eis(a)$  induced from the trivial representation of $\bB(\A) \subset
\bG(\A)$, where $\A$ denotes the adeles of the field $\F$.
By 
construction, spherical Eisenstein series are certain functions on 
\[
\Bun_\bG = \bG(\F) \backslash \bG(\A) /\textup{maximal compact} \,, 
\]
which are always eigenfunctions of the Hecke operators (and of the
Laplace operators in the number field context). However, they
are not square-integrable and the problem of the describing
the spectrum of Hecke operators in the intersection of
$L^2(\Bun_\bG)$ with the closure of the span of Eisenstein series
is highly nontrivial.

In his groundbreaking work \cite{L}, Langlands introduced integral
formulas and shift of the integration contours as tools of spectral
analysis. 
the spectrum. These techniques received further important contributions in
the work of many prominent mathematicians, including the work of 
C.~M\oe glin and J.-L.~Waldspurger \cites{Moe1,Moe2,MW}, who analyzed
the case of classical groups. Volker Heiermann, Marcelo de
Martino and Eric Opdam, starting in
\cite{DHO}   and completing their project in the recent
paper \cite{DHO2}, extended
the analysis to general split groups.
Their work is focused on the number field case and the minimal
parabolic subgroup. It is computer 
assisted in the exceptional cases.  

\subsubsection{}

Taking linear combinations of Eisenstein series, one can form the
function
\begin{equation}
\Eis_\omega = \int_{|a| = \textup{fixed} \gg 1} \Eis(a) \, \omega(a)
\,
d_\textup{Haar}(a) \,, \quad \omega \in \C[\bA^\vee]\,,\label{eq:253}
\end{equation}
known as the pseudo-Eisenstein series, or wave packets in other
contexts, see Appendix \ref{s_app} for an  elementary example.
These are always in
$L^2(\Bun_\bG)$ but never eigenfunctions of the Hecke operators.
Indeed, Hecke operators acts on \eqref{eq:253} as multiplication by
a $W$-invariant function of $a$ under the integral sign. In other
words, we have a commutative diagram
\begin{equation}
  \label{eq:24}
  \xymatrix{
    \C[\bA^\vee] \ar[rr]^{\Eis_\omega} \ar[d]_{\textup{multiplication by $\C[\bA^\vee]^W$}} && L^2(\Bun_\bG) \ar[d]^{\textup{Hecke
      operators}} \\
    \C[\bA^\vee] \ar[rr]^{\Eis_\omega}  && L^2(\Bun_\bG)\,,
    }
  \end{equation}
 in which a Hecke operator corresponds to its image in
 $\C[\bA^\vee]^W$ under the Satake isomorphism.

 \subsubsection{}
 The abstract spectral theorem asserts that a self-adjoint operator
 is unitarily equivalent to multiplication by a real-valued function
 $\lambda$ in $L^2(\bbS, \bmu)$, where $(\bbS,\bmu)$
 is a measure space. The spectrum of this operator is then the support
 of pushforward $\lambda_*(\bmu)$. The joint spectrum of the
 Hecke operators in the span of the Eisenstein series is easily seen
 to be simple, hence we can always take $\bbS = 
 \bA^\vee/W$, with coordinate functions $\lambda$. In this case,
 the joint spectrum is the support of $\bmu$.

 While \eqref{eq:24} conjugates Hecke operators to multiplication by
 coordinate functions, it gives no immediate information
 about the measure $\bmu$ and, hence, about the spectrum. By
 construction, this information is captured by the pullback
 of the norm $\| \cdot \|_{L^2(\Bun_\bG)}$ under $\Eis_\omega$,
 and concretely by distribution $\Psi_\Eis$
 defined by
\begin{equation}
  \label{eq:254}
  \left\langle \Psi_\Eis,  \omega_1 \otimes \omega^*_2 \right \rangle \eqdef
  \left(\Eis_{\omega_1}, \Eis_{\omega_2}
  \right)_{L^2(\Bun_\bG) },
\end{equation}
where $\omega^*(a) = \bar \omega(a^{-1})$.

Our geometric
language really captures the full spectral decomposition, including
the spectral measure and spectral projectors. This is more information
than the spectrum, which
only remembers the measure equivalence of $\mu$.
Even on the most basic level, that of the normalization of measures, 
formulas like \eqref{eq:255} below contain computations equivalent
to Langlands' computation of the Tamagawa numbers \cite{Lvol}.

\subsubsection{}

The main results of \cite{KO1} may be summarized as follows. First
\begin{equation}
  \label{eq:255}
   \Psi_\Eis = \Psi_{\xi_\F} \,. 
 \end{equation}
 Second, the spectral
 decomposition of $\Psi_\Eis$ is induced, via $p_2^*$ and the
 scissor relations, from the decomposition into conjugacy
 classes of nilpotent elements in \eqref{eq:252}. This works equally well
 for function fields and number fields.

 Note that by \eqref{eq:24} the distribution $\Psi_\Eis$ descends to
an element
\begin{equation}
  \label{eq:73}
  \Psi_\Eis \in \textup{linear dual} \Big(\underbrace{\C[\bA^\vee]
  \otimes_{\C[\bA^\vee]^W} \C[\bA^\vee]}_{K(\cT)} \Big)\,. 
\end{equation}
One may view this as a heuristic reason why it is natural to
expect a formula of the form \eqref{eq:255}. 

 An elementary example of how \eqref{eq:255} works may be found
 in Appendix \ref{s_app}. Of course, it may be much more interesting
 to explore the deeper meaning of \eqref{eq:255} and its
 generalizations.

 \subsection{L-function genera}

 \subsubsection{}

 Let $\bG$ be a reductive group over a global field $\F$, which for
 simplicity we will assume to be split. 
 The very general and powerful philosophy going back to Langlands
 aims to describe complex or $\ell$-adic automorphic forms for $\bG(\A)$ in term of
 representations
 \begin{equation}
   \label{eq:260}
   \phi: \Gamma \to \bG^\vee(\bbf)\,, 
 \end{equation}
 where $\bbf$ denotes the complex numbers $\C$ or the $\ell$-adic
 algebraic numbers $\overline{\Q}_\ell$. 
 Here $\Gamma$ is a
 certain enhanced version of the Galois group of $\F$ which fits into
 an exact sequence of the form \eqref{eq:256}. 
 
 In case of
 a function field $\F=\Bbbk(C)$, $\Gamma'$ is the geometric
 fundamental group, meaning that
 \begin{equation}
   \label{eq:259}
   \Gamma' = \pi_1(C \otimes \overline{\Bbbk})
 \end{equation}
 in the unramified situation. The extension \eqref{eq:256} includes
 the natural action of the Frobenius automorphism on
 \eqref{eq:259}. Galois-theoretic parametrization of 
 cuspidal automorphic representations in the function field
 case is a monumental achievement of V.~Lafforgue \cite{Laff}.
 The nature of $\Gamma'$ in the
 number field case is much more mysterious, see for instance \cite{Lmarch}.

 \subsubsection{}

 To associate an L-function to a homomorphism \eqref{eq:260}, one
 normally composes $\phi$ with a linear representation of
 $\bG^\vee$. Instead, and this is an important point for us,
 one can consider an action of $\bG^\vee(\bbf)$ on an algebraic variety
 $\bY$ over $\bbf$.

Let  $\bbf$ be 
an algebraically
 closed field of characteristic zero. 
We call an action of an abstract group $\Gamma$ on an
 algebraic variety $\bY$ over $\bbf$ algebraic if there is a triangle 
 \begin{equation}
   \label{eq:263}
   \xymatrix{
     \Gamma \ar[rr] \ar[dr]_\phi && \Aut(\bY) \,, \\
     & \bG^\vee(\bbf) \ar[ur]
     }
 \end{equation}
 in which $\bG^\vee$ is a linear algebraic group acting
 on $\bY$. We denote the ambient group by $\bG^\vee$ because it will
 be the Langlands dual group in the context of Eisenstein
 series. However, the notion of an algebraic action is abstract and it
 is always beneficial to choose $\bG^\vee$ as large as possible in
 \eqref{eq:263}. The larger the group, the more variables will the
 L-genus potentially involve. For simplicity, we will assume that $\bG^\vee$ and
 the Zariski closure $\overline{\phi(\Gamma)} \subset \bG^\vee$ are 
 reductive.

 For algebraic actions of groups of the form
 \eqref{eq:256} on a smooth algebraic variety $\bY$ we will define a nonlinear
 generalization $\bL(\bY)$  of an L-function as a certain equivariant
 genus of the fixed locus $\bY^{\Gamma'}$. Note that this fixed locus depends
 only on the Zariski closure
 \begin{equation}
 \overline{\phi(\Gamma')} \subset \bG^\vee\,. 
\label{eq:264}
\end{equation}
The closure \eqref{eq:264} is reductive by our assumptions, being a normal subgroup of
$\overline{\phi(\Gamma)}$. It follows that $\bY^{\Gamma'}$ is smooth.
In the case of the trivial
 action, the construction of $\bL(\bY)$ will recover the $\zeta$-genus
 \begin{equation}
   \label{eq:265}
   \phi=1 \quad \Rightarrow \quad \bL(\bY) = \xi(T^*\bY) \,. 
 \end{equation}

 \subsubsection{}\label{def_act} 

 The genus $\bL(\bY)$ will be defined as a distribution on a certain space
 of \emph{deformations} of the action \eqref{eq:263}. 
Homomorphisms $\phi$ in \eqref{eq:263} come in families in which the restriction
of $\phi$ to $\Gamma'$ stays the same, while the values of $\phi$ on
the quotient $\cQ_\F$ are deformed by a map 
\[
a: \cQ_\F \to \bG\buu\,, \quad \bG\buu
=(\bG^\vee)^{\phi(\Gamma')}_{\textup{connected}} \,. 
\]
The continuous variables $a$ will correspond, in particular, to the
continuous parameters of Eisenstein series. This  is why we denote them
by the same
letter $a$.

The group $\bG\buu$ has an action of a commutative reductive 
algebraic group
\begin{equation}
  \label{eq:262}
  \bH\buu = \overline{\phi(\Gamma)}/\overline{\phi(\Gamma')}\,, 
\end{equation}
%
The group $\bH\buu$ has a distinguished semisimple topological generator $\Fr$ in
the function field case and a distinguished infinitesimal
generator in the number field case, which we will also denote also by
$\Fr$. It is the $\bG\buu$-conjugacy class of
$\phi(\Fr) a(\Fr)$ that matters in the deformation, where $\Fr$ is an
arbitrary lift of the Frobenius to $\Gamma$. The element 
$\phi(\Fr) a(\Fr)\in \bG^\vee$ is understood infinitesimally in the number field
case.

The
$\bG\buu$-conjugacy classes of semisimple elements in
$\Aut(\bG\buu)$ have been studied classically, see Section 3.3.8 in
\cite{VinOnish}. In particular, up to $\bG\buu$-conjugacy, the map $a$ may
be chosen in the form
\begin{align}
a: \cQ_\F \to \bA\buu &=
                        \textup{maximal torus of
                        $(\bG\buu)^{\phi(\Fr)}$} \notag \\
                        &=
\textup{maximal torus of $(\bG^\vee)^{\phi(\Gamma)}$} \,. \label{eq:266}
\end{align}
While there can be a finite group action further identifying
equivalent deformations, we will focus for now on \eqref{eq:266} as the
basic space of deformations. We see from \eqref{eq:258}  that it is naturally identified
with $\spec \bhd(\pt/\bA\buu)$, or more precisely with the
fiber over $\Fr$ in the following diagram
\begin{equation}
  \label{eq:267}
  \xymatrix{
    \textup{deformations \eqref{eq:266}} \ar[rr] \ar[d]&&
    \spec \bhd(\pt/\bH\buu \times \bA\buu) \ar[d] \\
    \Fr \ar@{|->}[rr]&&  \spec \bhd(\pt/\bH\buu) \,. 
    }
\end{equation}
Concerning this diagram, we recall that $\spec \bhd(\pt/\bH\buu)$ is
the
group $\bH\buu$ itself in the function field case and its Lie algebra
$\Lie \bH\buu$ in the number field case. 

For example, if $\bY$ is an irreducible linear representation of
$\bG^\vee$ then $\bA\buu$ is reduced to the scalars and
$\phi(\Fr) a(\Fr)$ in the function field case is a scalar multiple of the Frobenius element
that enters familiar formulas for the L-functions.

\subsubsection{}
Consider the fixed locus
\[
\Fix=\bY^{\Gamma'} = \bY^{\overline{\phi(\Gamma')}}\,. 
\]
Since $\bY$ was assumed
smooth and $\overline{\phi(\Gamma')}$ is reductive, it follows that
$\Fix$ is smooth with the tangent bundle
\begin{equation}
  \label{eq:268}
  T \Fix = \left( T \bY \big|_{\Fix}
  \right)^{\Gamma'} \,. 
\end{equation}
We define
\begin{equation}
  \label{eq:268-2}
  T_\vir  \Fix = \Hd\left(\Gamma', T \bY \big|_{\Fix}
  \right) \in K_{\bH\buu \times \bA\buu}(\Fix)\,. 
\end{equation}
Instead of taking the K-theory class in \eqref{eq:268-2}, one can
consider the corresponding complex of $\bH\buu \times
\bA\buu$-equivariant coherent sheaves on $\Fix$. It is also 
easy to imagine a generalization in which $T \bY \big|_{\Fix}$ is by itself a
complex of $\Gamma'$-modules.

\subsubsection{}

In the function field case, Galois actions on $\bY$ correspond to local
$\bY$-systems over $C$ for which $\Fix \subset \bY$ are the global flat 
sections. Geometrically, global sections are obstructed by higher
cohomology groups, and the virtual tangent space in \eqref{eq:268-2}
is the natural virtual tangent space to the space of global 
sections.

In the number field case, the authors of this paper do not have
new insights to offer on the nature of the group $\Gamma'$.
It seems to us that it is more pragmatic to define the functor
$H^1(\Gamma',\textup{---})$ as a infinite-dimensional linear space
with the action of $\Fr$ given by the zeroes of the corresponding
L-function.

It may be worth
stressing that the final formulas used in applications show
remarkable uniformity between function fields and number fields
and can be proven and used without those geometric constructions that may or may not
be available in the number field case.

\subsubsection{}

In the function field case, we define 
\begin{equation}
  \label{eq:272}
  \bL(\bY) = \chi(\Fix, \cOh_{\Fix,\vir})\Big|_{\Fr}  \,. 
\end{equation}
Here 
$\cOh_{\Fix,\vir}$ is the symmetrized virtual structure sheaf corresponding
to the virtual tangent space \eqref{eq:268-2}, about which we
will say more below. The definition is arranged
so that
\begin{align}
 \bL( \textup{space of a representation $\pi$}) &=\Lbar(\pi,a) \notag\\
                                                &\overset{\textup{\tiny
                                                  def}}=
                                                  {\det}_{\Hd(\Gamma',\pi)}
                                                  \left(a^{1/2} \Fr^{1/2} - a^{-1/2} \Fr^{-1/2}\right)^{-1}                       \,,    \label{eq:273}
 \end{align}
 where $a\in \bA\buu$ is a scalar operator in $\pi$. 
 We take formula \eqref{eq:273} as the definition of the completed
 L-function $\Lbar(\pi,a)$ in the function field case. Observe that
 it can be interpreted as an equivariant $\Aroof$-genus
 \begin{align}
   \label{eq:276}
   \Lbar(\pi,a) = \Aroof\left(-\Hd(\Gamma',\pi)\right)
 \end{align}
 evaluated on the automorphism $a \Fr$. The $\Aroof$-genus here 
 corresponds to the function
 \begin{equation}
   \label{eq:277}
   \psi_{\Aroof}(x) = {x^{1/2} - x^{-1/2}} \,,
 \end{equation}
 and is closely related to indices of Dirac operators, see Section
 \ref{s_Spin} below. Since
 $\Hd(\Gamma',\pi)$ is even-dimensional, we conclude by Poincare
 duality that 
 \begin{equation}
   \label{eq:278}
    \Lbar(\pi,a)  =  \Lbar(\pi^*,q^{-1} a^{-1})  \,, 
 \end{equation}
 where $\pi$ denotes the dual representation.

 \subsubsection{}\label{s_epsilon}

 As it is always the case with $\Aroof$-genera and Dirac operators,
 the definition requires fixing some square root, namely the square
 root of the determinant line $\Lambda^{\textup{top}}
 H^1(\Gamma',\pi)$. This determinant line is, essentially, the
 $\varepsilon$-factor in the functional equation for the L-functions.
 One should treat this line as a line bundle on the space of all
 relevant parameters, and fix a square root of this line bundle. 
 
 While there is no canonical square root, finding and
 fixing a square root is aided by the fact that $\dim
 H^1(\Gamma',\pi)$ is even for any $\pi$. Indeed, taking a Jordan-Holder
 filtration of $\pi$, we may assume that  $\pi = \pi'
 \otimes V$, where $\pi'$ is irreducible and the multiplicity space $V$ is a representation of
 the centralizer of $\Gamma'$. This gives
 \[
 \det H^1(\Gamma',\pi) = \left(\det H^1(\Gamma',\pi')\right)^{\dim V}
   \otimes \left(\det V\right)^{\dim H^1(\Gamma',\pi')} \,, 
   \]
 and hence it suffices to fix the square roots of $\det
 \Fr|_{H^1(\pi)}$ for irreducible linear representations $\pi$ that 
appear in the tangent bundle to $\bY$. From our current
 viewpoint, these determinants are simply elements of the algebraically
 closed field $\bbf$, and their square roots may be chosen
 arbitrarily. Special considerations, for instance considerations of
 self-duality, may be required
 to make compatible choices when different actions of $\Gamma'$ interact.

 If a compatible choice of square roots cannot be made,
 one can replace the symmetrized sheaf $\cOh_{\Fix,\vir}$ by the
 usual virtual structure sheaf $\cO_{\Fix,\vir}$, which means
 replacing the $\Aroof$-genus by the Todd genus. While less
 symmetric, this genus still carries an equivalent amount of information.

 \subsubsection{}

In the number field case, the requirement that
\begin{equation}
  \label{eq:274}
  \bL( \textup{space of a representation $\pi$}) =
  \textup{completed L-function of $\pi$} 
\end{equation}
uniquely fixes an analytic characteristic class on $\Fix$, which we
push forward to $\pt/\bH\buu \times \bA\buu$ to
obtain the distribution $\bL(\bY)$. Here completed means that it
includes the Archimedean L-factors and a square root of the
$\varepsilon$-factor, which simplify the functional equation to
\[
  \Lbar(\pi,s)  =  \Lbar(\pi^*,1-s) \,.
\]
Here the traditional variable $s$ corresponds to minus the logarithm of
the coordinate on $\cQ_\F$.

\subsubsection{}

The symmetrized virtual structure sheaf used in \eqref{eq:272}
is defined by
\begin{equation}
  \label{eq:261}
  \cOh_{\vir} = (\det T_\vir)^{-1/2} \otimes \cO_{\vir} \,,
\end{equation}
provided a square root of the line bundle $\det T_\vir$ exists. Here
$\cO_{\vir}$ is the usual virtual structure sheaf. 
  
While, in general, there are obstacles to constructing $\cO_{\vir}$
when higher obstructions are present in
the virtual tangent space, in the case at hand case the construction of $\cO_\vir$
readily reduces to the special case of the $\zeta$-genus. Indeed, we have
the decomposition into $\overline{\phi(\Gamma')}$-modules
\begin{equation}
  \label{eq:275}
 T \bY =  T  \Fix \oplus T_{\mov} \bY 
\end{equation}
where the second term contains only nontrivial $\Gamma'$-modules.
Since nontrivial $\Gamma'$-modules have neither $H^0$ nor $H^2$, we
conclude that the contribution of $T_{\mov} \bY$ to $T_{\vir} \Fix$ 
\begin{equation}
\cNrm = - \Hd(\Gamma', T_{\mov} \bY ) = H^1(\Gamma', T_{\mov} \bY
)\label{eq:281}
\end{equation}
is a vector bundle of rank $(2g-2)\codim \Fix$. The
trivial representation $T  \Fix$, by construction, gives the pushforward of the
$\xi_\F$-genus under the projection
\begin{equation}
p: T^*\Fix \to \Fix \,.\label{eq:280}
\end{equation}
We thus conclude 
\begin{equation}
  \label{eq:279}
  \bL(\bY) = \chi\left(T^*\Fix,\xi_\F \otimes p^*\Aroof(\cNrm)
             \right)  \,. 
\end{equation}
Separating the contribution of $H^1$ to the $\xi_\F$-genus, we can
rewrite \eqref{eq:279} as follows

\begin{Proposition}
  We have
  \begin{equation}
    \label{eq:28}
    \bL(\bY) = (- q^{\frac12})^{\dim \Fix} \, \chi\left(T^*\Fix, p^*\Aroof\left(H^1\left(\Gamma',T\bY\big|_\Fix\right)\right)
             \right)  \,. 
  \end{equation}
\end{Proposition}

\begin{proof}
  From the equality
  $$
  \xi_{\bP^1}(t) = \frac{q^{1/2} t}{(1-t)(1-qt)} = \frac{(-q^{1/2}
    )}{(1-1/t)(1-qt)} \,, 
  $$
  we see that the contributions of $H^0$ and $H^2$ to the
 $\xi_\F$-genus combine into $(- q^{\frac12})^{\dim \Fix}
 \cO_{T^*\Fix}$. 
 \end{proof}

\subsubsection{}

What really helps in our situation is that Frobenius acts on the
fibers of \eqref{eq:280} even if it does not act on the base $\Fix$.
The additional weight $q^{-1}$ of the Frobenius action can make
infinite sums appearing the nonproper pushforward \eqref{eq:280}
converge.

Similarly, Frobenius action on \eqref{eq:281} helps in the situation
when the tangent spaces $T \bY$ and $T_{\mov} \bY$ are already
virtual tangent spaces. As their cohomology acquire an additional
Frobenius weight, it may be possible for $p^*\Aroof(\cNrm)$, which
in this case will be a ratio of two K-theory classes, to 
be a well-defined analytic K-theory class in the sense of
Section \ref{s_AK}.

\subsection{Categorification and other
  generalizations}\label{s_cat_gen} 

\subsubsection{Manifolds with a polarization}\label{s_L_polar}

Since the symplectic manifold $M=T^*\bY$ appears very prominently in
\eqref{eq:28}, it is natural to extend the applicability of
$\Lbar$-genera 
to other manifolds which may share some key features of $T^*\bY$
without being cotangent bundles to manifolds with $\Gamma$-action. 

In one possible direction, we may define the $\Lbar$-genus for any
manifold
$M$
that admits a $\Gamma\times \bA\buu$-equivariant invariant
polarization. As above, $\bA\buu$ is a torus that commutes with $\Gamma$
and forms the arguments of the $\Lbar$-genus. By definition, a
polarization $T^{1/2} M$ a class in 
$K_{\Gamma\times \bA\buu}(M)$ such that
\begin{equation}
T^{1/2} M + q^{-1} \left(T^{1/2} M  \right)^{*} = TM \,.\label{eq:27}
\end{equation}
For example, $p^* T \bY$ is a polarization of $M=T^*\bY$. We define 
  \begin{equation}
 \bL_{\frac12}(M) = (-q^{1/2})^{\frac12\dim M^{\Gamma'}} \,
 \chi\left(M^{\Gamma'}, \Aroof\left(H^1\left(\Gamma',T^{1/2} M \right)\right)
 \right) \,.\label{eq:25}
 \end{equation}
By the functional
equation \eqref{eq:278}, this is independent of the choice
of the polarization.

\subsubsection{Spinor bundles}\label{s_Spin}
Note that \eqref{eq:27} implies that 
\[
  TM = q^{-1} \, T^*M 
\]
in $K_\textup{eq}(M)$. In examples of interest to us, there exists 
a $\Gamma$-invariant 
symplectic form $\omega_M$ on $M$ of weight $q$ 
%
\begin{equation}
\omega_M: \Lambda^2 TM \to q^{-1} \otimes \cO_M \,. \label{eq:32}
\end{equation}
%
It induces 
a natural \emph{symmetric}
pairing
\begin{equation}
  \label{eq:29}
S^2 H^1\left(\Gamma',T M
\right)
\xrightarrow{\quad \cup\otimes  \omega_M 
  \quad} q^{-1} \otimes H^2(\Gamma',
\cO_{\Fix}) = \cO_{\Fix} \,, 
\end{equation}
where the cup product is as in \eqref{eq:cup}. 
In \eqref{eq:29} and above $\cO$ stands for the trivial line bundle on
the corresponding spaces. The pairing \eqref{eq:29} induces a Clifford algebra
structure on $\Ld \, H^1\left(\Gamma',T M
\right)$. We denote this sheaf of algebras by $\Cliff_{M,\Gamma}$ and denote by 
\begin{equation}
  \Spin_{M,\Gamma} = \Spin_0 \oplus \Spin_{1}\label{eq:30} 
\end{equation}
the corresponding $\Z/2$-graded self-dual spinor bundle, assuming it
exists. This is a weaker
assumption than the existence of a polarization and of the
corresponding square root in \eqref{eq:25}.
The self-duality of $\Spin$ implies that
\[
  \Cliff
  \cong \left(\Spin\right)^{\otimes 2}
\]
as $\Z/2$-graded
vector bundles. 
With these preparations, we may define the $\Lbar$-genus as the
index
\begin{equation}
  \label{eq:31}
  \bL_{\frac12}(M) = (-q^{1/2})^{\frac12\dim M^{\Gamma'}} \,
 \chi\left(M^{\Gamma'},  \Spin_0 - \Spin_1 
 \right) \,.
\end{equation}
Clearly, this agrees with \eqref{eq:25} on the common domain of
applicability.

\subsubsection{Categorification}

Without plunging into categorical depths, we will only discuss
the special situations of interest to us in Section \ref{s_curves}.
We assume that the structure sheaf $\cO_M$ has an even
cohomological degree grading such that the form \eqref{eq:32}
\begin{equation}
\omega_M: \Lambda^2 TM \to \cO_M\sTate{2} \,,  \label{eq:65}
\end{equation}
shifts the cohomological degree by $2$ with a Tate twist. 
This is
the case, for instance, for Coulomb branches as constructed in
\cites{ BFN1, BFN2, BFN3, NakCoul}, since they are defined as spectra of certain cohomology
rings, see Section \ref{s_curves} for an elementary example.  In
\eqref{eq:65} and it
what follows, we combine the degree shift and Tate twist if the
overall shift is pure and set, by definition
\begin{equation}
  \label{eq:64}
  \sTate{n} = [n] \Tate{\tfrac n2} \,. 
\end{equation}
It follows from \eqref{eq:65} that the pairing \eqref{eq:29} lifts to 
\begin{equation}
  \label{eq:29cat}
S^2 H^1\left(\Gamma',T M
\right)
\xrightarrow{\quad \cup\otimes  \omega_M 
  \quad} H^2(\Gamma',
\cO_{\Fix}) \sTate{2} = \cO_{\Fix} \,. 
\end{equation}
As a result, the Clifford algebra $\Cliff_{M,\Gamma}$ 
is still well-defined and we can ask for a self-dual
spinor bundle $\Spin$ on $M^{\Gamma'}$ which is $\Z$-graded cohomologically.
If such bundle exists, we define
\begin{equation}
  \label{eq:33}
  \cOh_{\Fix,\vir} = \Spin  \sTate{-\delta}\,, \quad \delta=\tfrac 12 \dim M^{\Gamma'} \,. 
\end{equation}
The cohomological shift
\[
  \sTate{-\delta} = \cK^{1/2}_\Fix \,, 
\]
which appears in \eqref{eq:33}, is the correct categorical
lift of the $(-q^{1/2})^\delta$-factor in \eqref{eq:31}.

\subsubsection{}\label{PoissonM} 
More generally, $M$ can be a Poisson manifold such that 
\begin{equation}
  \label{eq:40}
  \{ \, \cdot \, , \, \cdot \, \}: \Omega^2 M \to \cO_M\sTate{-2}
\end{equation}
where $\Omega^2 M = \Lambda^2 \Omega^1 M$. In this case, $\Omega^1 M
\sTate{2}$ can be used as a replacement on the tangent bundle in the
definition of $\Cliff_{M,\Gamma}$. We define 
\begin{equation}
  \label{eq:6}
  \Cliff_{M,\Gamma} = \textup{Clifford}(H\Omega_M), \quad
  H\Omega_M= H^1(\Gamma', \Omega^1 M) \sTate{2} \,. 
\end{equation}
Note that a $\Gamma$-equivariant
Poisson map
\[
  f: M_1 \to M_2
\]
induces a homomorphism $\Cliff_{M_2} \xrightarrow{f^*} \Cliff_{M_1}$
and, therefore, a pushforward $f_*$ for sheaves of Clifford modules.
If $f_*$ is a symplectic resolution of singularities then $f_*
\cOh_{\Fix_1,\vir}$ may be used as a replacement for the virtual
structure sheaf of $\Fix_2$, see an example in Section \ref{s_curves}.

\subsubsection{Connected components}\label{s_components} 

 Note that L-genera, as defined above, depend only on the first order
 neighborhood of $M^{\Gamma'}$ inside the ambient
 space $M$  with a $\Gamma$-action. In particular, there is no
 interaction\footnote{While one may imagine
 generalizations that may involve, for instance, $\Gamma'$-invariant 
 curves joining different components of $M^{\Gamma'}$, this is not something
 we consider in this paper.} between different connected components of
$M^{\Gamma'}$, beyond the fact that they may be permuted by the action
of $\Fr$. Thus, given a $\Fr$-invariant subset of component $M^{\Gamma'}_1 \subset M^{\Gamma'}$ we
 may define L-genera of $M_1$ by restricting any of the
 definitions above to $M^{\Gamma'}_1$. Here $M_1$ may be takes as a
 germ of $M$ along $M^{\Gamma'}_1$.

 By construction, L-genera are computed on the fiber over $\Fr$ in
 \eqref{eq:267}. Therefore, any component of $M^{\Gamma'}$ which is 
 nontrivially permuted by $\Fr$ makes no contribution to the
 L-genus. In other words, we can always restrict the computation to the subset of
 components that are $\Fr$-stable.

 The flexibility to specify a set of components is convenient for the
 following pragmatic definition of genera for quotient stacks.

\subsubsection{Quotient stacks}\label{s_fix_quo}

For applications to Eisenstein series and other needs, it is
convenient to have a definition of L-genera for stacks, in
particular, for quotient stacks. For quotient stacks, we propose
the following recipe.

Suppose an algebraic action of $\Gamma$ on $\bX$ considered above
normalizes an action of
an algebraic group $\bH$ on $\bX$. In particular, $\Gamma$ acts on
$\bH$ by automorphisms, the fixed points of which form the
subgroup $\bH^{\Gamma'}  \subset \bH$. For the quotient stack $\bX/\bH$, we
define
\begin{equation}
\Fix(\bX/\bH) = \bX^{\Gamma'} \big/ \bH^{\Gamma'} \,.\label{eq:292}
\end{equation}
From 
\[
T (\bX/\bH) = T \bX - \Lie \bH
\]
we see that $T_\vir \Fix$ gets an extra piece $H^1(\Gamma', (\Lie
\bH)_\mov)$.

While there may be other definitions of the fixed locus for quotient
stacks, for all of them \eqref{eq:292} will form a subset of the set
of 
components. This  subset plays a distinguished role in the
application we have in mind.

\subsubsection{Example}\label{ExamplePGP}

Let 
\begin{equation}
  \label{eq:88}
  \phi: \Gamma \to \sP^\vee 
\end{equation}
be a homomorphism and consider the induces action of $\Gamma$ on
the flag manifold $\Flag = \bG^\vee / \sP^\vee$, which we will view
as a quotient stack. In this example, we will investigate the
difference between the fixed locus of the quotient and the
quotient of fixed loci. 

Since a parabolic subgroup is its own normalizer, we have
\[
(\Flag)^{\Gamma'} = \{ \textup{conjugates $\sP'=g\, \sP^\vee\!g^{-1}$ that contain
  $\phi(\Gamma')$} \}\,. 
\]
By reductivity of $\overline{\phi(\Gamma')}$, 
\begin{equation}
  \label{eq:78}
  T_{\sP'} (\Flag)^{\Gamma'} = \Lie \left(\bG^\vee\right)^{\Gamma'}
  \big/ \Lie \left(
    \sP^\vee\right)^{\Gamma'}  \,. 
\end{equation}
It follows that the connected component 
$\bG\buu$ of $(\bG^\vee)^{\Gamma'}$ acts transitively on each connected component of
$\Flag^{\Gamma'}$. For each connected component, we have a homomorphism
\begin{equation}
  \label{eq:74&3}
  \phi_g = g^{-1} \, \phi \, g: \Gamma' \to \sP^\vee\,, 
\end{equation}
well-defined up to the conjugation by $\sP^\vee$. Clearly,
\begin{equation}
  \label{eq:75&4}
  \Fix(\Flag) = \left(\bG^\vee\right)^{\Gamma'}
  \big/ \left(
    \sP^\vee\right)^{\Gamma'} =
  \left\{
    \begin{matrix}
      \textup{$\Gamma'$-fixed flags such that}\\
      \textup{$\phi_g$ is equivalent to $\phi\big|_{\Gamma'}$}
    \end{matrix}
    \right\} \subset \Flag^{\Gamma'} \,. 
\end{equation}
Each component of this locus is isomorphic to $\bG\buu/\sP\buu$, where
$\sP\buu = \bG\buu \cap \left(
    \sP^\vee\right)^{\Gamma'}$ is a parabolic subgroup. Other
components can have different stabilizers, including stabilizers of
different dimensions. 

Since only $\Fr$-stable components contribute to the genus, we may
restrict our attention to the intermediate subgroup 
\begin{equation}
  \label{eq:74**}
  \bG\buu \subset \bG^\diamond \subset \left(\bG^\vee\right)^{\Gamma'}
  \,, 
\end{equation}
defined by 
\begin{align}
  \label{eq:76}
  \bG^\diamond &= \left\{ \left. g \in  \left(\bG^\vee\right)^{\Gamma'} \right| \,
                 g \, \phi(\Fr) \, g^{-1} \, \phi(\Fr)^{-1} \in \bG\buu \right\} \\
  & = \{ 
\textup{$g$ such that $g\phi g^{-1}$ is a deformation of $\phi$} \}
    \notag \\
  & = \textup{$\Fr$-stable components of
    $\left(\bG^\vee\right)^{\Gamma'} $} \,. \notag 
\end{align}

\section{Applications}\label{s3}

\subsection{Additive relations among L-genera}

By construction, L-functions behave like a genus, that is
\[
L(\pi_1 \oplus \pi_2) = L(\pi_1) L(\pi_2) \,,
\]
and similarly for the symmetrized L-functions \eqref{eq:273}. 
The genera $\bL(\bY)$ satisfy scissor relations, which are
\emph{additive} in the space $\bY$. In particular, the space $V$ of an
irreducible linear representation $\pi: \Gamma \to GL(V)$ is not irreducible in the
sense of $\Gamma$-actions on algebraic varieties. For example
it always has the trivial orbit $0\in V$, and will typically have
many other $\Gamma$-invariant subvarieties.

Note that these additive relations are understood in the sense of
distributions, which concretely means that they relate different
integrals, of different dimensions, involving L-functions. An
example of such additive relation is the spectral decomposition of the
distribution $\Psi_{\Eis}$.



\subsection{Eisenstein series}

\subsubsection{} 

In general, Eisenstein and pseudo-Eisenstein series on $\Bun_\bG$
are defined  by the
pull-push $\pi_{\bG,*} \pi_\bM^*$ of functions, which we will assume
$\C$-valued, in the following
diagram
of spaces
\begin{equation}
  \label{eq:187_1}
  \xymatrix{
    & \Bun_{\sP} \ar[ld]_{\pi_\bM} \ar[rd]^{\pi_\bG}\\
    \Bun_\bM && \Bun_\bG \,.}
\end{equation}
Here $\sP \subset \bG$ is a parabolic subgroup and
$\sP \to \bM$ is the quotient by the unipotent radical.

For Eisenstein series, one starts with an automorphic form $\phi$ on
$\Bun_\bM$, which we will assume a cuspidal Hecke eigenfunction. Since $\bM$ has a
nontrivial center and, hence, nontrivial characters, such forms come in
families of the form $\phi(m) \|\chi(m)\|^s$, where $\chi$ is
a character of $\bM$ and $s\in\C$ is a complex number. More
abstractly, one should view the function $a(m) = \|\chi(m)\|^s$
as an element of
\begin{align}
  \notag
  a \in \spec \bhd(\pt/\bA\buu) &= \Chars(\bM) \otimes \Hom(\cQ_\F,\Ct)  \\
  &= \Cochars(\bA\buu) \otimes \Hom(\cQ_\F,\Ct)  \,,\label{eq:283}
\end{align}
where $\bA\buu$ is the center of the Langlands dual group $\bM^\vee$.
We denote
\begin{equation}
  \Eis(\phi,a) = \pi_{\bG,*} \pi_\bM^* \, a \cdot \phi\label{eq:284}
  \,,\quad 
  a \in \spec \bhd(\pt/\bA\buu) \,. 
\end{equation}
In principle, there is a finite group $\pFix(\phi)$ that idenfies equivalent choices
of $a$ in \eqref{eq:284}. The quotient $\bhd(\pt/\bA\buu)/\pFix(\phi)$ 
is the true space of parameters for the Eisenstein series associated
with $\phi$.

The spectral analysis problem mentioned in Section \eqref{s_Eis_spec1}
corresponds to the trivial function $\phi=1$ on the minimal parabolic
subgroup $\sP=\bB$. The general spectral analysis problem can
be posed in the exactly the same way and rephrased, in the
exactly the same way, as finding a certain additive decomposition of
the distribution $\Psi_\Eis$ defined by \eqref{eq:254}  on two copies of the
parameter space $\bhd(\pt/\bA\buu)$.

\subsubsection{}
For function fields $\F=\Bbbk(C)$, the work of Lafforgue \cite{Laff} relates
cuspidal Hecke eigenfunctions $\phi$ to Galois representations. In the
unramified situation, these are irreducible
representations
\begin{equation}
\phi: \Gamma \to \bM^\vee(\Qlb)\label{eq:286}
\end{equation}
of the unramified Weil group
 \begin{equation}
   \label{eq:256bb}
   1 \to \Gamma'=\pi_1(C \otimes \overline{\Bbbk})
   \to \Gamma \xrightarrow{\,\, \| \, \cdot \, \|\,\, }
   \cQ_\F \to 1 \,. 
 \end{equation}
The irreducibility of \eqref{eq:286} means that the image of $\phi$
does not lie in any smaller Levi subgroup or, equivalently, that the
center $\bA\buu$ of $\bM^\vee$ satisfies 
\begin{equation}
  \label{eq:287}
  \bA\buu = \textup{maximal torus of } (\bM^\vee)^{\phi(\Gamma)}  
\end{equation}
as in \eqref{eq:266}, with the role of $\bG^\vee$ played by
$\bM^\vee$. Irreducibility of $\phi$ also implies that the group
\[
\bB\buu = \bG\buu \cap \sP^{\vee}
\]
is a Borel subgroup of the group $\bG\buu$. We recall that
$\bG\buu \subset \left(\bG^\vee\right)^{\Gamma'}$ is defined
as the connected component of the identity.

\subsubsection{} \label{s_t1}
As in \eqref{eq:254},
the $L^2$-norm of pseudo-Eisenstein series defines a distribution
on two copies of $\bA\buu$, which we want to understand with
geometric tools. Consider the quotient stack
\begin{equation}
  \label{eq:89}
\cT=   T^*(\sP^\vee \backslash \bG^\vee / \bB^\vee)
\end{equation}
with the induced action of $\Gamma$. As explained in Section
\ref{ExamplePGP} all components of the fixed locus $\Fix\subset \cT$ are copies
of $T^*(\bB\buu \backslash \bG\buu / \bB\buu)$. Therefore,
the L-genus in the following theorem defines a distribution on
$\bA\buu \times \bA\buu$. 

\begin{Theorem}{\cite{KO2}} \label{t1} We have 
 \begin{equation}
   \label{eq:288}
   \Psi_\Eis = 
   \bL_{\frac12}(T^*(\sP^\vee \backslash \bG^\vee / \sP^\vee))\,.
 \end{equation}
 \end{Theorem}

 We recall from Section
 \ref{ExamplePGP} that the components of $\Fix$ that are $\Fr$-stable, 
 and hence contribute to
 L-genus, are indexed by $\pi_0(\bG^\diamond)$,
where the group $\bG^\diamond$ is defined in \eqref{eq:76}.

 As before, the fixed locus $\Fix$ maps to nilpotent
 elements in $\Lie \bG\buu$ and one can decompose its genus
 into the corresponding pieces. When the
 the representation $\phi$ is \emph{tempered}, this is
 entirely parallel to what we did in \cite{KO1}, see \cite{KO2}.

 \subsubsection{}
 Note that complex Eisenstein series mesh seamlessly with $\ell$-adic
 representation in Theorem \ref{eq:288} because K-theoretic
 computations that are involved in the definition of the RHS in
 \eqref{eq:288} are really defined over integers, and not over the field
 of the definition of the varieties in questions. If one wants a more
 direct interaction between $\ell$-adic and complex manifolds, one
 may note that the L-genera are certain tensor functors of the
 $\Gamma'$-isotypic components of the normal bundle to the fixed locus
 $\Fix$. Both the fixed locus and the isotypic component
 depend only on the Zariski closure of
 $\phi(\Gamma')$. Since we assume this closure to be reductive, it can
 be uniquely transferred from the $\ell$-adic world to the complex
 world.

 \subsection{L-genera in enumerative problems}\label{s_mirror} 

 \subsubsection{} 

 A large portion of modern enumerative geometry studies various
 moduli spaces of maps $\Maps(C \xrightarrow{\, f\,} \bX)$  from a curve $C$, which may be fixed or
 move in moduli, to some target algebraic variety $\bX$. For a number
 of reasons, ranging from the needs of applications in theoretical
 physics to the desire to work with moduli spaces that enjoy certain 
 properness, one is lead to allow certain singularities for
 the maps $f$, about which we will be more specific below. For our
 goals, the ability to change the degree of $f$ by creating or
 removing singularities will play a very important role.

 One
 sets up things so that, whatever counts one wants to do, these
 counts have a well-defined answer for maps of fixed degree
 \[
   \deg f = [f(C)] \in H_2(\bX,\Z) \,,
 \]
 after which one forms a generating function over all
 degrees. Equivalently, one studies the Fourier transform of
 \eqref{eq:1} as function on the
 dual torus
 \begin{equation}
   \label{eq:74*}
   \bA\buu = H^2(\bX,\Z) \otimes \Gm \,. 
 \end{equation}
 The reason we denote this torus by $\bA\buu$ is the following.
 In the examples we have in mind, $\bX$ is obtained as a quotient
 by a parabolic subgroup $\sP$ and the resulting map 
 \[
   \Cochars(\bA\buu) = \Chars(\sP) \to H^2(\bX,\Z)
   \,,
 \]
 turns out to be an isomorphism. 

 \subsubsection{}

 In the traditional approach, see Part 4 in \cite{Mirror} for an excellent
 introduction, the counts are defined by some variation
 on the formula
   \begin{equation}
   \left \langle \lambda_1 \lambda_2 \dots  \right
   \rangle_{d} = \int_{\deg f = d} [\Maps]_\vir \cap \bigcup
   \lambda_i \label{eq:1}\,, 
   \end{equation}
   where the integral denotes
   pushforward to a point, 
$[\Maps]_\vir$ denotes a distinguished homology
   class called the virtual fundamental cycle, and 
$\lambda_1, \lambda_2, \dots$ are some cohomology classes on
   the moduli of maps, including those pulled back from $\bX$  via
   the evaluation maps 
   \begin{equation}
     \label{eq:3}
     \ev_x:  f \mapsto f(x) \in \bX\,, \quad x \in C \,. 
   \end{equation}
   Parallel constructions exist in other
   cohomology theories, notably equivariant K-theory.

   Because of the
   direct link between K-theory and indices of the Dirac operator and other 
   elliptic operators, computations in equivariant K-theory are particularly close
   to the indices which our colleagues in modern high energy physics are computing
  in supersymmetric quantum mechanics on
  the space of maps\footnote{If $\X$ represents a
     component of the moduli
     spaces of vacua in some supersymmetric quantum field theory on $C
     \times \{\textup{time}\}$, the dynamics of  the corresponding modulated vacua may
     be described by supersymmetric quantum mechanics on $\Maps(C \to
     \X)$. This is  a
     natural low-energy
     approximation to the full dynamics of the theory, which captures
     deformation-invariant quantities like indices.}. 

   \subsubsection{}\label{s_Hcrit}

   One may compare and contrast counts \eqref{eq:1} with
   counts defined as \emph{traces}, or more precisely supertraces, of certain operators
   acting in the critical cohomology $\Hd_{\textup{crit}}(\Maps)$ 
  or its
  extraordinary analogs.  These are defined when $\Maps$ can be
  described, globally or locally, as a critical
   locus of a certain function
   $\cW$, see for example \cite{Ionuts} and the many references therein. The counts \eqref{eq:1} are typically richer and, in
   particular, allow richer equivariance\footnote{For a function $\cW$,
     the construction of the
equivariant critical cohomology requires 
     $\cW$ to be invariant (as opposed to scaled by
     group action), which represents a very noticeable loss of equivariance in enumerative applications. A
     promising approach to restoring full equivariance is provided by
     the interpretation of the fully equivariant counts as
     Nekrasov-style
     $qq$-characters, for which one hopes to find a purely
     representation-theoretic definition.}. Traces over $\Hd_{\textup{crit}}(\Maps)$ have the advantage
   of being easily categorified.

   In some situations, $\Hd_{\textup{crit}}(\Maps)$
may be equated with ordinary cohomology of related moduli spaces
   by the mechanism of dimensional
   reduction, see e.g.\ \cite{DeformRed} and references therein for the current state of
   technology in this area.  In what follows,  we will assume, for simplicity
   of notation, that dimensional reduction has been already applied to
   pass from critical cohomology to ordinary cohomology. In the
   concrete situation considered in Section \ref{s_curves}, this
   means replacing critical cohomology of  quasimaps to $\X =T^*
   \bP^1$ with the ordinary cohomology of quasimaps to
   $\bX=\bP^1$. 

   \subsubsection{}

   For the computation of traces, it is beneficial to study
   $\Hd(\Maps)$ as a module over some algebra $\cA$
   which includes operators of cup product like $\cup \ev_x^*(\alpha)$,
   $\alpha\in \Hd(\bX)$, but 
   may also include raising and lowering
   correspondences which change the degree of $f$ by 
   creating or destroying singularities at points
   $x\in C$. 
   Various correspondences associated with $x\in C$
   can be assembled into a single correspondence
   \begin{equation}
     \label{eq:4}
     \cE \in \Hd(I \times C \times \Maps) \to  \Hd(\Maps) \,, 
   \end{equation}
   where $I$ is a discrete indexing set for the generators of $\cA$.
   Composing $\cE$ along the $\Maps$-factors, one
   gets an algebra of correspondences between $(I \times C)^n \times
   \Maps$ and $\Maps$. It can be studied locally, since
   modifications at different points of $C$ supercommute.

   By definition, the algebra $\cA_\cE(C)$ has generators
   $\be_i\of{\gamma}$  labelled by an index $i \in
   I$ and a cohomology class $\gamma \in \Hd(C)$. These act
   by the
   corresponding components
\[
   \be_i\of{\gamma} \mapsto \cE(\{i\} \otimes \gamma \otimes
   \, \textup{---}\,) \in \End \Hd (\Maps) 
 \]
of the K\"unneth decomposition of $\cE$ along the
   $I\times C$-factor. The generators $\be_i\of{\gamma}$ satisfy two
   sets of relations, discussed in \eqref{eq:81} and
   \eqref{eq:82} below. Since the
   correspondence $\cE$ will be fixed in what follows, we will drop it
   from the notation $\cA(C)$.

 It is convenient
   to use the term \emph{descendants} of $\gamma$ to refer to all
   generators labelled by $\gamma$. The change
   in degree provides a natural grading of
   $\cA(C)$ by the lattice $H_2(\bX,\Z)$. There is also a $\Z$-grading
   by the cohomological degree.

   \subsubsection{}

   At the expense of making the indexing set $I$ infinite, we may
   assume that the linear span of the correspondences $\cE$ is
   closed under two following operations.
First, consider the commutator of $\cE$ with itself 
   along the $\Maps$-factor. This is a
   correspondence between $(I \times C)^2\times \Maps$ and $\Maps$,
   supported on the diagonal in $C^2$. To say that this diagonal
   correspondence is in the linear
   span of the correspondences $\cE$ is to say that
   \begin{equation}
     \label{eq:79}
     \left[ \cE_{iC_1}, \cE_{j C_2}\right] = \sum_k c^k_{ij} \, \cE_{k,C} \circ
     \Delta^*\,, 
   \end{equation}
   for some $C$-independent coefficients $c^k_{ij}$ in the ground ring, where $i,j,k
   \in I$. In \eqref{eq:79}, all correspondences act in the
   $\Maps$-factor, $C_1$ and $C_2$ refer to the two
   factors in $C^2$, and 
  \[ 
   \Delta: C \to C_1 \times C_2
 \]
 is the diagonal map. Second, we can restrict the composition of the correspondences from
 $(I \times C)^2$ to the diagonal in $C$, resulting in a relation of
 the form
 \begin{equation}
 \cE_{iC_1} \cE_{j C_2} \circ \Delta_{*} = \sum_k m^k_{ij} \, 
 \cE_{k,C}\label{eq:80} \,, 
\end{equation}
with $C$-independent coefficients $m^k_{ij}$ in the ground ring.
In terms of the generators, this means the Lie superalgebra relations 
\begin{equation}
     \left[ \be_i\of{\gamma_1} , \be_j\of{\gamma_2} \right] =
     \sum c^k_{ij} \, \be_k\of{\gamma_1 \cup \gamma_2} \,,\label{eq:81}
   \end{equation}
   and the relations
   \begin{equation}
     \label{eq:82}
      \be_i \be_j \of{\Delta_* \gamma} = \sum_k m^k_{ij} \, 
 \be_{k}\of{\gamma} \,. 
\end{equation}
In \eqref{eq:82}, we use the convention that $\be_i \be_j \of{\gamma_1
  \otimes \gamma_2} = \be_i \of{\gamma_1} \, \be_j \of{\gamma_2}$.
We have 
\begin{align*}
  \Delta_* \bpt &= \bpt \otimes \bpt \\
  \Delta_* \alpha &= \alpha \otimes \bpt + \bpt \otimes \alpha \\
  \Delta_* 1 &= 1 \otimes \bpt + \textstyle{\sum} \, \alpha_i \otimes \alpha_i^\vee +
  \bpt \otimes 1 \,, 
\end{align*}
where $\alpha \in H^1(C)$ and $\alpha_i$ and $\alpha^\vee_i$ form a
dual set of bases of $H^1(C)$ with respect to the intersection
form. From this we conclude that: 
\begin{itemize}
\item[(1)] assuming we work
  nonequivariantly, the descendants of $\bpt\in H^2(C)$ generate a commutative
  subalgebra $\cA^{2} \subset \cA(C)$ which does not depend on
  $C$; 
\item[(2)] the subalgebra  $\cA^{\ge 1}(C) \subset \cA(C)$
generated by the descendants of $H^{\ge 1}(C)$ is the Clifford algebra
of $\Omega^1 \cA^{2}\otimes H^1(C)$, where $\Omega^1$
denotes the module of K\"ahler differentials. The symmetric pairing on
the generators of the Clifford algebra comes from the skew-symmetric
Poisson
structure $c^k_{ij}$ and the skew-symmetric multiplication in $H^1(C)$. 
\end{itemize}
   
One can relate the descendents of $1\in H^0(C)$, which will not receive a
   significant attention in this paper, to 
   Hamiltonian vector fields on the spectrum of $\cA^{2}$, with their
   natural action on functions and $1$-forms.
   For us, it is the descendant of classes $\alpha\in H^1(C)$ that are of
   maximal interest. They form the part of $\Hd(C)$ that
   is sensitive to the geometry of the curve $C$, via the Galois action 
   that they carry.

   \subsubsection{}\label{s_cAcE}

   While in the discussion above we explicitly avoided equivariant
   cohomology, the key to understanding $\cA^2$ lies in the study of the
   algebra $\cA_{\cE}(\A^1/\Gm)$
   generated by $\Gm$-equivariant correspondences for
   $C=\A^1$. Since $\bpt^2 = \tau \bpt$, where $\tau \in H^2(\pt/\Gm)$
   is the generator, this is a noncommutative algebra over
   $\Z[\tau]$. It becomes 
  $\cA^{2}$ modulo $\tau$. Computations modulo
  $\tau^2$ give $\cA^{2}$ its Poisson structure, see an example in Section
  \ref{s_cAcE2}.

   \subsubsection{}

   Enumerative geometry has experienced an
   enormous influx of ideas from modern high energy
   physics and, in particular, in recent years there has been a lot of
   progress in making so-called Coulomb branches of
   supersymmetric gauge theories act by correspondences
   as above
  \footnote{At the risk of a gross oversimplification, this phenomenon
    may be explained as follows. A supersymmetric gauge theory has 
    gauge fields, matter fields, and their supersymmetric
    partners. In a Higgs vacuum, the gauge fields vanish, the matter
    fields take a constant value at the bottom of their potential $\cW_\textup{mat}$, while
    the remaining
    constant gauge transformation act on them with a finite stabilizer.
   The moduli space $\X_\textup{Higgs}$ of such vacua is the quotient
   of the critical locus $\partial \cW_\textup{mat} =0$ by the gauge
   group. The dynamics of the modulated Higgs
   vacua is described by the supersymmetric quantum mechanics of the
   moduli spaces of maps, or more precisely quasimaps, from $C$ to
   $\X_\textup{Higgs}$.

   There is an opposite kind of phase, called Coulomb phase, in which
   the matter fields vanish, or take some constant value fixed by the whole
   gauge group, while the superpartners of the gauge fields are
   allowed constant commuting values. Nonabelian and
   nonperturbative effects in the Coulomb phase are due to the
   appearance of monopoles, which fill the space-time like a gas of
   particles of different species. The moduli space
   $\X_\textup{Coulomb}$  of Coulomb vacua, which should be a Poisson algebraic
   variety because of the supersymmetry constraints, may be interpreted as describing
   the equations
   of state for this monopole gas.  Inserted in
   the expectation, each monopole species gives a global
   function on $\X_\textup{Coulomb}$.

   The correspondence \eqref{eq:4} may be interpreted as
   the insertion
   of a monopole of species $i \in I$ at the spatial location $x\in C$ in
   the spacetime $C\times \R$.
   Thus, descendants of the point class $\bpt\in H^2(C)$ give global
   function on $\X_\textup{Coulomb}$. This logic can be inverted to
   analyze $\X_\textup{Coulomb}$ as in \cites{NakCoul}, thus 
   circumventing the formidable difficulties
   present in the direct analysis of the monopole gas. }. 
Among the
   really vast literature on the subject, we should mention
   \cites{BFN1, BFN2, BFN3 ,Vermas, NakCoul} as the references that particularly influenced our thinking
   about the subject.

   From such considerations, one expects a Poisson map
   \[
     \sigma: \X^\vee \to \Spec \cA^{2}\,, 
   \]
   where $\X^\vee$ is a certain nicer Poisson algebraic variety, such that
   \[
   \sigma^*: \cA^{\ge 1}(C) \to \Cliff\, (\Omega^1 \X^\vee
   \otimes H^1(C))\,,
 \]
 and similarly for the whole algebra $\cA(C)$. Here $\Cliff$ denotes 
 the corresponding sheaf of Clifford algebras. It is reasonable to
 expect that this sheaf of Clifford algebras is easier to study than $\cA^{\ge 1}(C)$ and
 that one can construct important $\cA(C)$-modules as pushforwards
 under $\sigma$.

  \subsubsection{}\label{s_C_gen}

 As an important additional degree of freedom, one may have a
 constructible sheaf 
   $\bph_{\Maps}$ on
   $\Maps$ and  a matching local system $\bph$ on $C$ with
   a correspondence 
   \begin{equation}
   \cE : \Hd(C \times \Maps, \bph \boxtimes
   \bph_{\Maps} ) \to \Hd(\Maps, \bph_{\Maps})\label{eq:5}\,. 
   \end{equation}
   The local system  $\bph$ has rank $|I|$, which 
   absorbs the $I$-factor from earlier formulas. As before, it is
   convenient to assume that this rank is countable. The coefficients
   $c^k_{ij}$ and $m^k_{ij}$ from \eqref{eq:79} and \eqref{eq:80}
   amount to bilinear operations
   \[
     \{\, \cdot\, , \, \cdot \,\}, m : \bph \otimes \bph \to \bph \,, 
   \]
  which make $\bph$ a local system of Poisson algebras. Fixing a
   base point $x \in C$, we may interpret $\bph$ as an algebraic action of the group
   $\Gamma=\pi_{1,\textup{alg}}(C)$ on a certain $\X^\vee$, which
   depends on $\bph$. To match the
   generality of our discussion to the simplifying assumptions made in 
   our treatment of  L-genera,
   we will assume that the image of $\Gamma$ is Zariski dense
   in a reductive subgroup $\overline{\Gamma} \subset \Aut \X^\vee$. 

   It is important to stress that while $\X^\vee$ depends on
   $\bph$, this dependence is only through the reductive group $\overline{\Gamma}$
   that is required to act on $\X^\vee$. We will write
   $\X^\vee_{\overline{\Gamma}}$ to indicate such
   dependence. 
   It is natural to require this assignment to be functorial in the sense of the diagram
   \begin{equation}
     \label{eq:84}
     \xymatrix{
       \overline{\Gamma}_1\ar@{~>}[rr] \ar[d] &&
       \X^\vee_{\overline{\Gamma}_1} \ar[d] \\
       \overline{\Gamma}_2 \ar@{~>}[rr]&&
       \X^\vee_{\overline{\Gamma}_2} 
     }\,.
   \end{equation}
   The vertical arrows in this diagram are maps in the category of
   algebraic groups and the category of schemes with group action,
   respectively.
   
   The grading by the degree of the map and by the cohomological degree
leads to the action of $\bA\buu \times \Gm$ on $\X^\vee$ which
commutes with the action of $\overline{\Gamma}$.


   \subsubsection{}

   We denote by $\cA_{\bph}$ the algebra generated by the K\"unneth
   components of \eqref{eq:5}, which we denote by 
   \[
   \be\of{\gamma} \in \cA_{\bph}\,, \quad \gamma \in
   \Hd(C,\bph)=\Hd(\Gamma',\bph) \,. 
   \]
   Here $\Gamma'=\pi_{1,\textup{geom}}(C)$. Since the groups $H^0$ and
   $H^2$ pick out the trivial component in $\bph$, we have
   \[
     \cA_{\bph}^{2} = \textup{$\overline{\Gamma}$-invariants} \big/ (
     \textup{noninvariants}) \xrightarrow{\quad \sigma^*\quad} \cO_\Fix\,, 
     \]
     where we mod out the invariants by its intersection with the
     ideal generated by noninvariant functions, and $\Fix \subset \X^\vee$
     denotes the scheme-theoretic fixed locus. Similarly, 
  \begin{equation}
   \sigma^*: \cA_{\bph}^{\ge 1} \to
    \Cliff_{\X^\vee, \Gamma}\,. \label{eq:35*2}
  \end{equation}
  where Clifford algebra in the right-hand side is generated by
  the sheaf
   \[
     H\Omega_{\X^\vee} = H^1(\Gamma', \Omega^1 \X^\vee \big|_{\Fix}
     )\sTate{2} \,. 
   \]
   Here we pay attention to cohomological degree and the Galois
   action, as in Section \ref{PoissonM}. See Section \ref{s_Galois_SL2} below for 
   the details on how this works in a concrete situation.

   \subsubsection{}

It is logical to
construct $\cA_{\bph}^{\ge 1}$-modules or $\cA_{\bph}$-modules 
as pushforwards of the Clifford modules, or Clifford modules with a
certain extra $\cD$-module structure, under $\sigma$. If $\Fix$ is smooth and has a spinor bundle
$\Spin$ then $\Spin$ is clearly the main building block for all
constructions of this kind. We recall from Section \ref{s_cat_gen}
that $\Spin$ is a natural categorification of the L-genus 
$\bL_{\frac12}(\X^\vee)$.

In the pushforward $\sigma_*$, one can in principle specify an
arbitrary support. We will see that modules of the form $\Hd(\Maps,
\bph_{\Maps})$
will indeed be constructed as pushforwards with support. What we do
in Section \ref{s_curves} in an example, and what we think is the right recipe in
general, is to take the pushforward with support in the total attracting manifold
for the $\bA\buu$-action. This automatically produces modules whose
grading by degree agrees with the effective cone in $H_2(\bX,\Z)$. 

\subsubsection{}

In Theorem \ref{t1}, L-genera appear as the answer to a 
counting problem which is controlled by the Frobenius action on
$\Hd(\bph_{\Maps})$ for a certain sheaf $\bph_{\Maps}$ on a
suitable moduli spaces of maps $C \to \bG/\sP$. More precisely, these
are sections of $\bG/\sP$-bundles over $C$ rather than maps, see 
Section \ref{s_ct} below. The cohomology
groups $\Hd(\bph_{\Maps})$  provide 
a natural categorification of the left-hand side in
\eqref{eq:288}. Thus, a
natural guess for the categorification of the right-hand side is that
\begin{equation}
  \label{eq:83}
  \X^\vee =   T^* (\sP^\vee \backslash \bG^\vee / \sP^\vee), 
\end{equation}
with the bundle $\Spin$ on its $\Gamma'$-fixed locus. 
%

In Section \ref{s_curves} we verify a version of this prediction in the simplest
example $\bG=PSL(2)$, with one strange exception which only
happens for curves of genus $2$, see Theorem \ref{t2}.

There are
various aspect of this verification that merit comments. The most
immediate one is that the spaces \eqref{eq:83} are, in general, \emph{not}
the same as the
conventional Coulomb branches\footnote{With current technology, those
 are defined by Nakajima and his collaborators for $\bG$ of classical type.} for
geometries like $T^*\bG/\sP$. Whence our expectation that more general spaces
$\X^\vee_{\overline{\Gamma}}$ from Section \ref{s_C_gen} can be
defined and studied\footnote{One may speculate that
  $\X^\vee_{\overline{\Gamma}}$  should
  become 
  relevant when the superpartners of the gauge fields take
  nongeneric values and leave the remaining gauge fields to vary in
  $\Bun_\bM$ for some Levi subgroup $\bM$.}.

\subsubsection{}\label{s_ct}

Let $\Bun_{\bX}$ denote the moduli
stack of $\bX$-bundles $\widetilde{\bX} \to  C$. For any $\widetilde{\bX}
\in \Bun_{\bX}$, we may consider the moduli
space $\Sect(C\to \widetilde{\bX} )$ of sections of
$\widetilde{\bX}$. These moduli spaces are very similar to $\Maps(C\to \bX)$ and
   the corresponding counts are studied in enumerative geometry in parallel
   with untwisted counts. Their arithmetic analogs provide very
   interesting functions on the stack $\Bun_{\bX}$. In the example
   $\bX=\bP^1$, $\bG=PSL(2)$ considered in Section \ref{s_curves}, we have
   $\Bun_{\bX} = \Bun_\bG$.

   In the classical enumerative context, all
   counts are invariant under equivariant deformation.
   Up to such deformation, 
   we may assume $\widetilde{\bX}$ is in the image of
   $\Bun_\bA \to \Bun_\bG$, where $\bA\subset \bG$ is a maximal
   torus. In the arithmetic context, deformation invariance is no longer valid, but we
   can still reduce the counts to $\Bun_\bA$ if we sum or integrate
   over all $\bB$-bundles that project to a given $\bA$-bundle. Such
   counts give rise to contant terms, Whittaker coefficients, and
   similar functions of the Eisenstein series. 

   Geometrically, this means we allow the bundle $\widetilde{\bX}$
   to vary in the unipotent directions. Concretely for
   $\bP^1$-bundles, this means we consider all bundles of the form
   $\bP(\cV)$, where rank 2 vector bundle $\cV$ fits into a sequence
   of the form
   \begin{equation}
     0 \to \cM \to \cV \to \cO_C \to 0 \,, \label{eq:86}
   \end{equation}
   with fixed $\cM \in \Pic(C)$. 
   Such moduli spaces turn out to be nicer and, in particular, more
   amenable to the analysis by $\bA$-equivariant localization.
   In general, equivariant localization plays an important role in the
   geometry of $\X^\vee$ as, for instance, different components of the
   fixed locus give different affine charts on it, see Section
   \ref{two_local}.

   \subsubsection{}\label{s_diagram3}
   Another technical point about Theorem \ref{t2} is that we work with
   quasimaps, or rather quasisections, in place of the regular
   maps. Quasimaps to $\bP^1$ provide a simple compactification of the
   space of maps, 
   for which the strata look like moduli spaces of maps of lower
   degree, with extra marked points. This allows one to, first, construct an
   algebra of correspondences from the inclusion of boundary strata 
   and, second, isolate the compactly supported cohomology of the
   moduli spaces of maps via the long exact sequence corresponding
   to the inclusion of the boundary. 

   This structure is reflected in the first correspondence in the
   following chain of correspondences 
   \begin{equation}
     \label{eq:85}
     \xymatrix{
       \X^\vee_{\QMN} \ar@{<->}[r]\ar@{=}[d]& \X^\vee_{\mathsf{MapsN}}
       \ar@{<->}[r]\ar@{=}[d]&
       \X^\vee_{\Psi_\Eis} \ar@{=}[d]\\
   T^* \bP^1 \ar@{<->}[r]& 
   T^* (\bN^\vee \backslash \bP^1) \ar@{<->}[r] &
     T^* (\bB^\vee \backslash \bP^1)\,.}  
   \end{equation}
   Here the correspondences are the conormals to the maps between the
   bases of the cotangent bundles, 
   $\QMN$ denotes the moduli spaces of quasisections of
   $\bP(\cV)$, where $\cV$ is allowed to vary as in \eqref{eq:86},
   while $\mathsf{MapsN} \subset \QMN$ corresponds to the open set of regular
  sections. We expect such structure to be present generally. 

   \subsubsection{}

   The second correspondence in the chain \eqref{eq:85} appears because we integrate the
   constant term, that is,  integrate over
   $\cM\in \Pic(C)$, to obtain the distribution $\Psi_\Eis$. The
   dependence on $\cM$ is Theorem \ref{t2} is introduced via
   a homomorphism 
   \begin{equation}
   \Pic(C)(\Bbbk) \xrightarrow{\quad \cM \mapsto \cO(\cM)\quad }
   \Pic_{\overline{\Gamma}\times \bA\buu}(\X^\vee)\label{eq:25*2}
   \end{equation}
 such that $\cO(\cM)\cong \cO(\deg \cM)$ as
   a nonequivariant line bundle.

   \subsubsection{}

   For $\bG=PSL(2)$, there are only Whittaker coefficients beyond the
   constant term, and those can be computed in the exact same fashion
   in the function field situation. Geometrically equivalent to
   Whittaker coefficient are counts of sections of $\bP(\cV)$ according to the order
   of their tangency with a distinguished section provided by $\cM
   \subset \cV$. Similar counts are well-studied in enumerative geometry
   and we are tempted to speculate that the duality interface \cite{AO}, which
   is a certain Borel-equivariant elliptic class on $\X \times
   \X^\vee$, may provide a general dictionary between insertions not
   just for K-theoretic counts but also
   in the arithmetic function field situation.

   In the number field situations, there is a lack of many basic
   building blocks and hence a really vast
   space for imagination. It is curious to note, however, that the
   Archimedean contributions to Whittaker coefficients of Eisenstein
   series for $\bG/\bB$ fully match the
   quantum connection for $\bG^\vee/\bB^\vee$, see \cites{Ger1,Ger2}.


\section{$\bP^1$-bundles over curves}\label{s_curves} 

\subsection{Overview }

\subsubsection{}

Our goal in this section is to illustrate the mechanism by which
the L-genera, and their natural categorification, appear
in the simplest automorphic and enumerative context, namely in the analysis of the
Eisenstein series for the group $\bG=PGL(2, \Bbbk(C))$.
Leaving the detailed discussion of definitions and results to
later sections, our goal in this overview is to explain them in
very broad strokes.

Let $\cP \to C$ be a $\bP^1$-bundle defined over $\Bbbk$.
By definition, the Eisenstein series
$\Eis(\cP,\chi)$ is the weighted count of 
$\Bbbk$-rational sections $s$ of $\cP$. The variable $\chi$ is a
character of the group $\Pic(C)(\Bbbk)$ which  determines
the weight of $s$. It involves
the factor $a^{-\deg(s)}$, where $a$ is a variable, and also
a character $\chi_0$ of the finite group $\Pic_0(C)(\Bbbk)$. For
$|a|\gg 1$, the series converges to a rational function of $a$.

\subsubsection{}\label{s_delicate}

For a bundle $\cP$ of the form $\bP(\cM \oplus \cO_C)$, where
$\cM\in \Pic(C)$, the count of sections contains the
count of sections of $\cM$. The latter is a famously delicate function
of
$C$ and $\cM$. This means that
values of $\Eis(\, \cdot \, , \chi)$ at points may be very difficult to
compute directly. 
Classically, it has been observed that other linear functionals
of $\Eis(\, \cdot \, , \chi)$ are easier to describe and open the
path to understanding Eisenstein series via Fourier analysis. 

One such functional, known as the constant term\footnote{See
  Appendix \ref{s_app} for an elementary discussion of the role
  played by constant terms in the spectral analysis of Eisenstein
  series.}, sums the Eisenstein
series over all $\cP = \bP(\cV)$ such that
\[
  0 \to \cM \to \cV \to \cO_C \to 0 \,, 
\]
where the line bundle $\cM$ is fixed. We denote this sum $\CT(\cM,\chi)$. 
One can further refine it by recording the
intersection of $s$ with the distinguished section of $\cP$ 
defined by $\cM$. These refined counts can also be done
following the logic below.

\subsubsection{}
Constant terms are very closely related to the distribution $\Psi_\Eis$, simply because
taking the constant term is the adjoint operator to the operator of
forming the Eisenstein series from $\chi$. 
Theorem \ref{t1} may be rephrased to give a formula for 
the constant term $\CT(\cM,\chi)$ in terms of $\bL_{\frac12}(\bX^\vee)$, where 
\[
  \bX^\vee = T^*\bP^1 = T^*\left(\bG^\vee / \bB^\vee\right) \,, 
\]
with the action of the Galois group $\Gamma$ that comes from a
homomorphism
\[
  \Gamma \xrightarrow{\quad (\chi,q^{-1}) \quad} \bA^\vee \times \bH
  \subset \bG^\vee \times \bH\subset \Aut(\bX^\vee) \,.
\]
Here $\bG^\vee=SL(2)$, $\bA^\vee \cong \Gm$ is the maximal torus of
$SL(2)$, and $\bH \cong \Gm$ scales the cotangent fibers. The group $\bB^\vee$ is a minimal
parabolic and hence $\bA\buu=\bA^\vee$. We recall that, by
construction, 
$ \bA\buu$ is the group that records the grading of Eisenstein
series by degree via its character $a^{-\deg(s)}$, $a\in \bA\buu$.
The action of $\bA\buu$ on $\bX^\vee$ commutes with the action of
$\Gamma$.

\subsubsection{}

From its construction, the genus $\bL_{\frac12}(\bX^\vee)$ can be interpreted as
the K-theory class of a certain spinor bundle $\Spin$ on
\begin{equation}
  \label{eq:221}
  \Fix = \left( \bX^\vee \right)^{\Gamma'} =
  \begin{cases}
    \bX^\vee\,, \quad & \chi_0^2=1\,, \\
    \{0,\infty\} \,,  \quad & \chi_0^2 \ne 1\,, 
  \end{cases}
\end{equation}
where $\Gamma' = \pi_{1,\textup{geom}}(C)\subset
\Gamma$ is the inertia subgroup.
Our goal in this section is to relate $\Spin$ to the constant
term directly and geometrically.

%
%

\subsubsection{}

The first steps towards bringing the Eisenstein series and the spinor
bundle $\Spin$ closer together are rather
standard. Eisenstein series are counts of $\Bbbk$-rational
points in suitable moduli spaces. Hence, they can be computed and
also categorified using compactly supported cohomology of
these moduli spaces, with their action of Frobenius.
In general, this cohomology is taken with coefficients in a
nontrivial local system determined by $\chi$. 

Further, counts of sections are easily
related to counts of quasisections, which is a useful technical
concept recalled in Section \ref{s_quasis}. In the Eisenstein series
context, counts of quasisections are known as the geometric
Eisenstein series, see \cite{BravGaits,Laumon}. We denote by $\QMN$ the
moduli spaces that appear in the computation of the constant
term of geometric Eisenstein series. In fact, the formulation in
terms of $\bL_{\frac12}(\bX^\vee)$, as opposed to
$\bL_{\frac12}(T^*(\bN^\vee\backslash \bG^\vee /\bB^\vee))$ already assumes passing form
the constant terms of Eisenstein series to the constant terms of
geometric Eisenstein series, see the discussion in Section
\ref{s_diagram3}. 

\subsubsection{}

Quasisections have
two important advantages. First, quasisection of $\cP$ of
a given degree $d$ form a proper moduli
space $\QM_d(\cP)$. Second, there are natural correspondences between quasisections of
different degrees.

These correspondence were introduced in \cite{FinkKuz} and, specialized to
our situation, 
generate an algebra $\cA_\chi(C)$ that may be described as a central
quotient of the universal enveloping
algebra of the Lie superalgebra $\Hd(C,\fsl_2^\chi)$. Here we compose
the map $\Gamma \xrightarrow{\,\, \chi \,\,}  \bA^\vee$ with the adjoint representation
to form the local system $\fsl_2^\chi$ on $C$.

\subsubsection{}

By construction, we have subalgefbras
\[
  \cA_\chi(C) \supset \cA^{\ge 1}_\chi(C) \supset \cA^{2}_\chi
\]
generated by the cohomology classes of degree $\ge 1$ and $2$,
respectively. For the computation of the constant term, it suffices
to know $\Hd_c(\QMN)$ as a Frobenius-equivariant module over
$\cA^{2}_\chi$, which is central. However, it is even more natural
to keep track of the action of $\cA^{\ge 1}_\chi(C)$, which is
a Clifford algebra over $\cA^{2}_\chi$.

\subsubsection{}

There is a Galois-equivariant map
\[
  \sigma: \bX^\vee \to \textup{quadric cone} = \Spec \cA^2_{\chi=1}\,. 
\]
It induces a pullback map
\begin{equation}
  \label{eq:46}
  \sigma^*: \cA^{\ge 1}_\chi(C) \to \Cliff_{\bX^\vee, \Gamma}\,, 
\end{equation}
where $\Cliff_{\bX^\vee, \Gamma}$ is the sheaf of Clifford algebras
acting on $\Spin$. We recall that $\Spin$ is distinguished by being
a self-dual bundle such that $ \Cliff_{\bX^\vee, \Gamma}
= \End(\Spin)$. We can twist it by line bundles on $\bX^\vee$ to obtain other
sheaves of Clifford modules of the same rank.

From \eqref{eq:46}, we get $\cA^{\ge 1}_\chi(C)$-modules by pushing
forward Clifford modules under $\sigma$. In this pushforward, we can additionally
specify support to create various kinds of modules.

\subsubsection{} 
Our main result in this section, Theorem \ref{t2} below, says that
\begin{equation}
  \label{eq:47}
  \textup{suitable centering of $\Hd_c(\QMN)$} = \sigma_{*,\bfL}
\Spin \otimes \cO(\cM)
\end{equation}
as a $\cA^{\ge 1}_\chi(C)$-module. Here centering means a 
combination of cohomological shifts, Tate twists, et cetera, 
$\sigma_{*,\bfL}$ is the pushforward with support in a certain
Lagrangian $\bfL \subset \bX^\vee$,  and $\cO(\cM)$ is as in
\eqref{eq:25*2}.

The line bundle $\cM\in \Pic(C)$ appears in the left-hand side of
\eqref{eq:47} as the point at which the constant term, and its
categorification $\Hd_c(\QMN)$, are computed. It appears in the
righ-hand side via the map \eqref{eq:25*2} that takes line bundles
on $C$ to equivariant line bundles on $\bX^\vee$. 
%

The identification in \eqref{eq:47} is $\Fr \times
\bA\buu$-equivariant and preserves cohomological degree, thus giving a categorial lift to the relation between the
constant term and L-genera.

\subsection{Eisenstein series}

\subsubsection{}

A global function field is a field of the form $\Bbbk(C)$, where 
$C$ is curve over a finite field
$\Bbbk$. The curve $C$ can be assumed smooth,  
projective, and irreducible over the algebraic closure
$\overline{\Bbbk}$. Eisenstein
series for the group $PGL(2,\Bbbk(C))$ are
functions on the countable set 
\[
\Bun_{PGL(2)} = \{ \textup{$\bP^1$-bundles $\cP \to C$
defined over $\Bbbk$}\}/\cong
\]
defined by counting all $\Bbbk$-rational sections $s: C \to \cP$ with a certain
weight. Namely, for each
section $s$, we may consider the isomorphism class of the
normal bundle
\begin{equation}
\NN_s= s^*(N_{\cP/s(C)}) \in \Pic(C)(\Bbbk) \,. \label{eq:207}
\end{equation}
Here
\begin{equation}
  0 \to \underbrace{\Pic_0(C)(\Bbbk)}_\textup{finite} \to  \Pic(C)(\Bbbk)
  \xrightarrow{\quad \deg \quad} \Z \to 0  \label{eq:208}
\end{equation}
is the group of $\Bbbk$-points of $\Pic(C)$. Given a character
\[
\chi : \Pic(C)(\Bbbk)  \to \Ct \,, 
\]
we form the series
\begin{equation}
\Eis(\cP,\chi) = \sum_{s: C \to \cP}  \chi(\NN_s) \,. \label{eq:26}
\end{equation}
This is easily seen to converge to a rational
function of $\chi$ if $|\chi(\textup{ample})| \ll 1$.
It is convenient to normalize $\chi$ by 
\begin{equation}
\chi(\cL) = (q^{1/2} a)^{-\deg \cL} \, \chi_0(\cL)\,,\label{eq:209}
\end{equation}
where $\chi_0$ has finite order and $a \in \Ct$ is a continuous
variable. For any $\chi_0$, the series \eqref{eq:26} converges to a rational
function of $a$ for $|a|\gg 1$. 

\subsubsection{}\label{s_modif} 

Two bundles $\cP$ and $\cP'$ are said to be related by an elementary
modification or a \emph{Hecke modification} if 
\begin{equation}
  \label{eq:160}
  \xymatrix{
    & \Bl_y \cP = \Bl_{y'} \cP' \ar[ld] \ar[rd] \\
    \cP && \cP'\,, 
    }
\end{equation}
where $y$ and $y'$ are points in the fiber over $x$. In other words,
go from $\cP$ to $\cP'$, we first blow up some point $y$ in the fiber
over $x$ and then blow down the fiber as in \eqref{Modif}
\begin{gather}
  \label{Modif}
  \raisebox{-60pt}{\includegraphics[height=120pt]{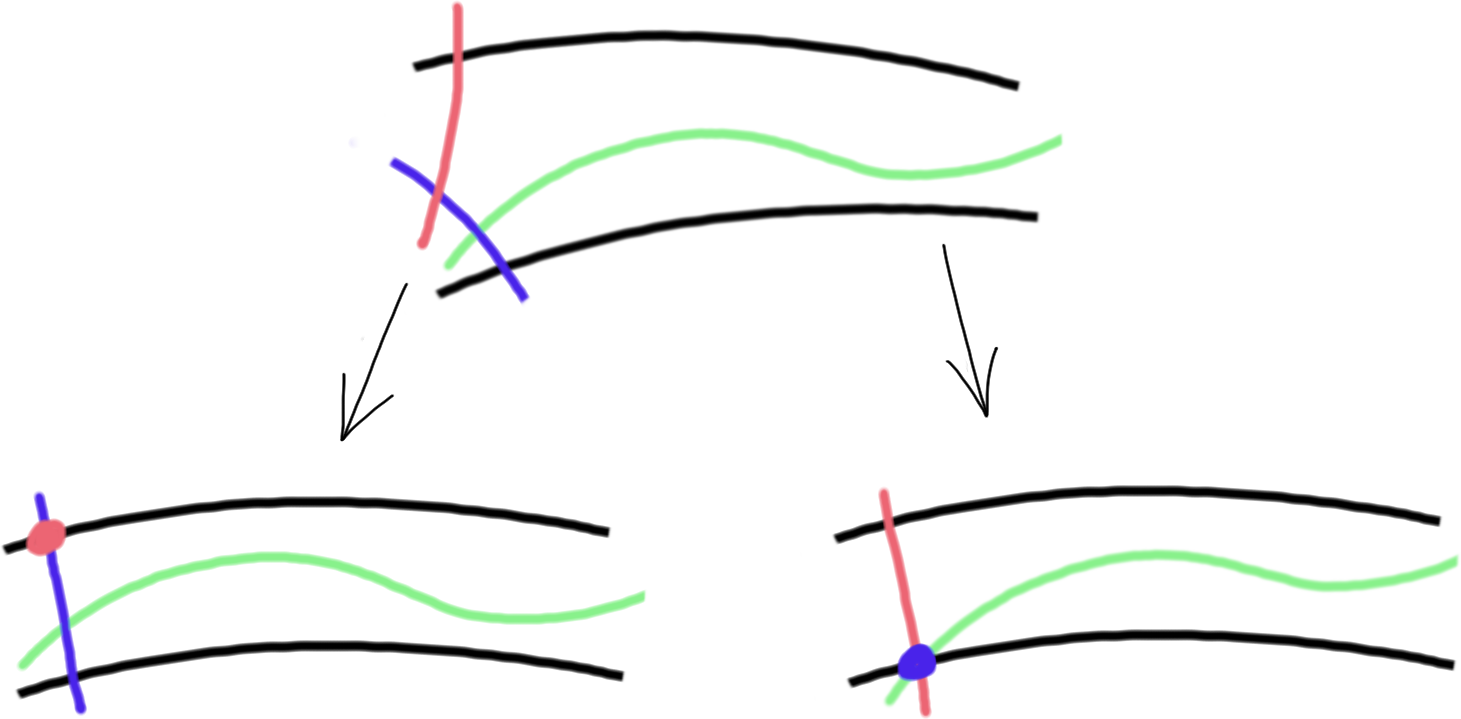}}\\
  \textup{\sl Modification of a $\bP^1$-bundle and its action on sections} \notag 
\end{gather}

\subsubsection{}

A section $s$ of $\cP$ in
\eqref{eq:160}
determines a unique section $s'$ of $\cP'$ and
\begin{equation}
  \label{eq:161}
  \NN_{s'} =
  \begin{cases}
    \NN_s(-x) \,, & y \in s \,, \\
    \NN_s(x) \,, & y \notin s \,. 
  \end{cases}
\end{equation}
The fiber of $\cP$ over $x\in \C$ is $\bP^1(\Bbbk_x)$, where
\[
[\Bbbk_x: \Bbbk] = \deg(\cO(x)) \overset{\textup{\tiny def}}= \deg x
\,.
\]
Therefore, there are 
$q^{\deg(x)}+1$ choices of $y$. Since only one is contained in $s$, we
infer that 
\begin{equation}
\sum_{\cP'=\textup{modification of $\cP$ at $x$}} \Eis(\cP',\chi) =
\lambda(\chi,x)\, \Eis(\cP,\chi) \,, \label{eq:211}
\end{equation}
where
\begin{equation}
  \label{eq:210}
  \lambda(\chi,x) = q^{\deg(x)/2}
\, 
  \tr
\begin{pmatrix}
  \chi_\rho(x) \\
  & \chi_\rho(x)^{-1} 
\end{pmatrix} \,, \quad \chi_\rho(x) = \chi(x) \, q^{\deg(x)/2} \,. 
\end{equation}
The relation \eqref{eq:211} and \eqref{eq:210} are the simplest
examples of the very general fact that Eisenstein series are
eigenfunctions of Hecke operators and that the corresponding
eigenvalues have something to do with Galois representations.
More precisely, since the \emph{geometric} Frobenius at the
point $x$ acts on $\Q_\ell(1)$ with weight $q^{-\deg(x)}$, we
can write
\begin{equation}
  \label{eq:302}
  \chi_\rho(x)  = \chi_0(x) a^{-\deg(x)} = \chi(x)
  \Tate{-\tfrac12}\,, 
\end{equation}
where the final bold parentheses denote minus half a of tensor product
with $\Q_\ell(1)$, that is, minus half of the Tate shift. The number
$\rho=1/2$ is the half sum of the positive roots for $PGL(2)$, which
is reflected in the notation in \eqref{eq:302}. Such shift is a very
general feature in theory of Eisenstein series. 

\subsection{Moduli of sections and quasisections} 

\subsubsection{}\label{s_quasis}
Recall that our goal in this section is to explain the relation between Eisenstein
series and L-genera on the \emph{categorical} level. On the Eisenstein
series side, the categorical lift is the obvious one. Namely,
$\Bbbk$-rational sections $s$ are $\Bbbk$-rational points of a certain 
moduli space of sections. Their counts are captured by the Frobenius
action on the compactly supported cohomology of the corresponding
moduli space (in general, with coefficients in a certain local
system). These cohomology groups, with the action of Frobenius,
provide the natural categorification of the Eisenstein series.

\subsubsection{}

Fix a rank 2 bundle $\cV$ over $C$ such that $\cP = \bP(\cV)$.
A section $s$ gives a rank $1$ subbundle $\cL \subset \cV$, that is,
an exact sequence of the form
\begin{equation}
  \label{eq:212}
   0 \to \cL \to \cV \to \textup{torsion free} \to 0  \,. 
 \end{equation}
 The requirement that the quotient in \eqref{eq:212} is
 torsion-free defines an open set in the larger moduli 
 space of quasimaps or quasisections. By definition, 
 \begin{equation}
 \QM(s: C \to \cP) = \{(\cL,s), \cL \in \Pic(C), s\in \bP(H^0(\cV \otimes
 \cL^{-1}))\} \,. 
\label{eq:214}
\end{equation}

\subsubsection{}

Clearly, for every quasisection $s$ there exists a unique
effective divisor $D$
 such that
 \begin{equation}
   \label{eq:213}
   \xymatrix{
     & \cL \ar[d] \ar[drr]^s\\
     0 \ar[r] &\cL(D)  \ar[rr]^{s(-D)\,\,\,\,} && \cV \ar[r] & \textup{torsion-free}
     \ar[r] & 0\,, 
     } 
 \end{equation}
 where the vertical arrow is the natural inclusion and the
 bottom row is exact. In other words, a quasisection is the same as
 a usual section and a collection of points of $C$. Thus, there
 is a very simple relation between counting sections and
 quasisections.

 Counts of quasisections are known as \emph{geometric} Eisenstein
 series \cite{BravGaits}. Their categorification, and its relation to
 categorification of L-genera, is easier to explore. We
 will focus on counts of quasisections and, correspondingly, on the cohomology of $\QM$.

 \subsubsection{}\label{s_bundle} 

 By construction, 
 \begin{equation}
   \label{eq:311}
   \xymatrix{
     \bP(H^0(\cV \otimes
     \cL^{-1})) \ar[r] \ar[d] & \QM \ar[d] \\
     \{ \cL\} \ar[r] & \Pic_{\deg \cL}(C)
     }
 \end{equation}
 is a projective space
 fibration over $\Pic(C)$. This fibration is empty for $\deg \cL \gg 0$
, smooth for $\deg \cL \ll 0$, and reflects interesting
 features of $C$ in the intermediate range. In the smooth
 range, we have
 \begin{equation}
   \label{eq:215}
   \dim \QM(s: C \to \cP)  = \deg \cV - 2 \deg \cL + 1-g =
   \deg \NN_s + 1 -g \,, \quad \deg \cL \ll 0 \,, 
 \end{equation}
 where $g$ is the genus of $C$ and
 \begin{equation}
   \label{eq:232}
   \NN_s = \det \cV/\cL 
 \end{equation}
 is the normal bundle to a quasisection $s$. Whenever nonempty, we
 define
 \begin{equation}
   \label{eq:215-2}
   \virdim \QM(s: C \to \cP)  = \deg \cV - 2 \deg \cL + 1-g \,. 
 \end{equation}
 The diagram \eqref{eq:311} equips $\QM$ and $\QMN$ with the
 tautological line bundes $\cO_{\QM}(1)$ and $\cO_{\QMN}(1)$

 \subsubsection{}
As already stressed in Section \ref{s_delicate}, Eisenstein series,
and hence also the cohomology groups $\Hd(\QM)$, depend
delicately on $C$ and $\cV$. There exist more robust moduli
spaces $\QMN$ which come up in the analysis of the constant term 
and will be discussed presently. 

In what follows, it will be convenient to normalize $\cV$ by requiring
that it fits into an exact sequence of the form
\begin{equation}
  \label{eq:226}
  0 \to \cM \to \cV \to \cO_C \to 0  \,. 
\end{equation}
We define 
\begin{align}
  \notag 
  \QMN= \{(\cV,\cL,s), &\textup{$\cV $ as in \eqref{eq:226}}\,,\\
  &\cL \in \Pic(C), s\in H^0(\cV \otimes
 \cL^{-1}), s\ne 0\}\Big/ \textup{iso fixing $\cM$ and $\cO_C$}\,, \label{eq:227}
\end{align}
where the line bundle $\cM$ is fixed and the isomorphisms by which
one is modding out in \eqref{eq:227} are required to be identity on
$\cM$ and $\cO_C$. The automorphisms of $(\cL,\cV)$ are
thus $\Aut(\cL) \times H^0(\cM)$, which act as a scalar  and as
unipotent transformations on the section $s$, respectively.
There is a natural action of $\Aut(\cM) = \Gm$
on $\QMN$. 

Possible choices of $\cV$ are parametrized
by $\Ext^1(\cO_C,\cM) = H^1(\cM)$ and from
\eqref{eq:215-2} we conclude 
 \begin{equation}
   \label{eq:215-3}
   \virdim \QMN= - 2 \deg \cL \,. 
 \end{equation}
 The last letter in the abbreviation $\QMN$ stand for the
 the unipotent freedom in choosing $\cV$ in \eqref{eq:226}. This is the
 geometric analog of the integration over the maximal
 unipotent subgroup in the automorphic setting.

\subsubsection{}\label{s_concreteQMN}

To have a more concrete description of \eqref{eq:227}, one may
note that the extension \eqref{eq:226} splits in the complement of
any point $y\in C$. Thus, $\cV$ is described by a 
transition function of the form
\[
T(t) = \begin{pmatrix}
 1 & T_{12}(t)  \\
 &1
\end{pmatrix} \,, \quad T_{12}(t) \in \Bbbk_y\pst\,, 
\]
between the formal neighborhood of $y$ and $C\setminus y$.
Here $t$ is a formal coordinate at $y$. The section $s$ is
described by 
\begin{equation}
s =
\begin{pmatrix}
  s_1(t) \\ 
  s_2(t) 
\end{pmatrix}\,,\quad s_1 \in H^0(C\setminus \{y\}, \cM \otimes \cL^{-1}
)\,,
s_2 \in H^0(C, \cL^{-1}
) \,, 
\label{eq:216}
\end{equation}
and the condition on $T$ and $s$ is that $Ts$ is regular at $t=0$.
Different choices of trivialization lead to the identification 
\[
(T,s) \cong (A T B^{-1}, B s) \,,
\]
where
\[
A \in  \left\{
    \begin{pmatrix}
      1 & \Bbbk_y\pTst\\
      0 & 1 
    \end{pmatrix} \right\}\,,
  \quad 
B \in  \left\{
    \begin{pmatrix}
      1 & H^0(C\setminus \{y\}, \cM) \\
      0 &1 
    \end{pmatrix} \right\}\,.
  \]
  Clearly, for $s_2(0) \ne 0$, the regularity of $Ts$ implies 
  \begin{equation}
    \label{eq:238}
    T_{12} = - \frac{s_1}{s_2}  \bmod \Bbbk_y\pTst \,, 
  \end{equation}
 which is well-defined modulo the action of $B$ and unique modulo the
  action of $A$. Since the choice of $y$ is arbitrary, we see that
  the locus $s_2 \ne 0$ in $\QMN$ is a smooth stack. The locus
  $s_2=0$ is also smooth, but $\QMN$ may not be smooth
  along that locus. 

  This nice feature of $\QMN$ can be observed for maps to more
  general targets $X$. Suppose that the attracting manifolds of a torus
  fixed point in $X$ ia an orbits of a certain unipotent group. If one allows
  the corresponding unipotent bundle to vary, the attracting manifold in the spaces
  of maps to $X$ becomes a smooth stacks with unipotent stabilizers.

 \subsection{Action of $\Hd(C,\fsl_2)$}\label{s_ef} 

 \subsubsection{}
 Consider the Lie superalgebra
 \begin{equation}
   \label{eq:37}
   \fsl_2 \otimes \Hd(C) = \Hd(C,\fsl_2)\,, 
 \end{equation}
 where in the right-hand side we consider the  cohomology of the
 trivial local system with fiber $\fsl_2$. Nontrivial local system will
 appear in Section \ref{s_Galois_e} below. 

 This Lie superalgebra acts by correspondences on 
 $\Hd(\QM)$ and $\Hd(\QMN)$. These correspondences 
 modify the line bundle $\cL$ and the section
 $s$. They do not affect the bundle $\cV$.

 The construction of $\fsl_2$-correspondences goes back to the classical work of
 Finkelberg and Kuznetsov \cite{FinkKuz} and forms the basic special case of
 many more advanced constructions in geometric representation theory,
 including the action of
Coulomb branches \cite{BFN1, BFN2, BFN3 ,Vermas, NakCoul} (as well as quantum groups of which Coulomb
branches tend to be quotients) on ordinary or extraordinary cohomology
of the moduli spaces of quasimaps.  We start by recalling the basic construction.




\subsubsection{}

Consider the value $s(x)$ of the section $s$ at a point $x\in C$. By
construction, this is a section of
 \begin{equation}
   \label{eq:325}
   \cN = \cV \otimes \cL^{-1} \otimes
\cO_{\QM}(1)\,, 
 \end{equation}
which is a rank 2 vector bundle over $C \times \QM$.
Here we interpret $\cL$ as the pullback to $C \times \QM$ of the
universal line bundle over $C \times \Pic(C)$. The
bundles $\cV$ and  $\cO_{\QM}(1)$ are   pulled back to $C \times \QM$ from the
first factor and the second factor, respectively.

The locus $s(x)=0$ is the image of the embedding 
 \begin{equation}
   \label{eq:320i}
   \iota: C \times \QM \hookrightarrow C \times \QM
 \end{equation}
 defined by
 \begin{equation}
   \label{eq:309}
   \iota((x,\cL,s))=(x,\cL(-x),s\otimes 1) \,. 
 \end{equation}
 Here $\cO_C\xrightarrow{\,\, 1 \,\,} \cO_C(x)$ is the canonical
 section. We may call $\iota^* \cN$ 
 the virtual normal bundle to \eqref{eq:320i}. It is the actual normal
 bundle for $\cL \ll 0$.

The inclusion \eqref{eq:320i} defines the functorial pullback map $\iota^*$ on
cohomology. The definition of the pushforward $\iota_*$ will require some
preparations.

\subsubsection{}
Consider the following abstract situation. Let $\iota: Y \to X$
be a closed inclusion with open complement $U$. Consider a relative
cohomology class $\varepsilon \in \Hd(X,U)$ and $\bar \varepsilon$ denote
its image in $\Hd(X)$. There is a factorization
\begin{equation}
  \label{eq:91}
  \xymatrix{
    \Hd(X) \ar[d]_{\iota^*} \ar[rr]^{\cup \bar \varepsilon} && \Hd(X)\\
    \Hd(Y) \ar[rr]^{\cup \varepsilon}  \ar@{.>}[rru]&& \Hd(X,U) \ar[u] 
    } 
\end{equation}
of the operator $\cup \bar \varepsilon$ of cup product by $\bar \varepsilon$,
which can be constructed as follows.

We denote by $\Lambda = \Q_\ell$
the constant sheaf of coefficients. Other sheaves of coefficients may
also be considered. By construction, we have
\begin{alignat}{2}
  \Hd(Y) &= \Hd(\iota^*\Lambda)&&= \Hd(\iota_* \iota^* \Lambda) \\
  \Hd(X,U) &= \Hd(\iota^!\Lambda)&&= \Hd(\iota_! \iota^! \Lambda)\,, 
\end{alignat}
where all maps are considered in the derived category of constructible
sheaves and $\iota_* =\iota_!$ since the map $\iota$ is proper. The
canonical maps $\Hd(X) \to \Hd(Y)$ and $\Hd(X,U) \to \Hd(X)$ come
from the adjunction maps $\Lambda \to \iota_* \iota^* \Lambda$ and
$\iota_! \iota^! \Lambda \to \Lambda$.

We have the following chain of maps
\begin{equation}
\Lambda \otimes \iota_! \iota^! \Lambda =
\iota_! ( \iota^* \Lambda \otimes \iota^! \Lambda) \to
\iota_! \iota^! \Lambda \to \Lambda \,,\label{eq:74}
\end{equation}
in which the first equality is the projection formula, see formula (2.6.19)
in \cite{KS} or Theorem 2.3.29 in \cite{Dimca}, while the
multiplication map $\iota^* \Lambda \otimes \iota^! \Lambda \to
\iota^! \Lambda$
is constructed in Proposition 3.1.11 in \cite{KS},
see also Theorem 3.2.13 in \cite{Dimca}. The sequence \eqref{eq:74}
yields the cup product map
\[
  \Hd(X) \otimes \Hd(X,U) \to \Hd(X)
\]
together with its factorization as in \eqref{eq:91}.

\subsubsection{} 

For our specific inclusion \eqref{eq:309}, we will construct the map
$\iota_*$ as the diagonal map in \eqref{eq:91} for the refined, also known
as localized, Euler class
 \begin{equation}
   \label{eq:233}
   \varepsilon = c_2(\cN) \in H^4( C \times \QM, U)\,. 
 \end{equation}
 Here $U$ is the complement of the image of $\iota$. The Euler class
  \eqref{eq:233} is constructed as follows.

  The bundle $\cN$ is a rank 2 vector bundle over $C\times \QM$ with a
  section $s(x)$. The data of a bundle with a section defines a map
  \begin{equation}
  \textup{clsf}: C\times \QM \to \mathbb{A}^2 / GL(2)\label{eq:90}
  \end{equation}
  to  the classifying stack for vector bundles with
  a section. By definition,
  \[
    \varepsilon = \textup{clsf}^*(c_2(\textup{Universal bundle})) = \textup{clsf}^*(
    [0/GL(2)]) \,. 
  \]
  Here 
  \[
   [0/GL(2)] \in H^4(\mathbb{A}^2 / GL(2), U_{\textup{universal}}) \,,
   \quad U_{\textup{universal}} = (\mathbb{A}^2\setminus 0) /
   GL(2)\,, 
 \]
 is the class of the origin in the target of \eqref{eq:90}. Since $U =
 \textup{clsf}^{-1}(U_{\textup{universal}})$, this gives the desired
 class \eqref{eq:233} and completes the construction of $\iota_*$. 

  

 An alternative definition of $\iota_*$ may be given as follows. From
 the computations in Section \ref{s_concreteQMN}, it is clear that
 allowing the bundle $\cV$ to vary gives a smooth total space of
 expected dimension. In
 this total space, $\iota$ is a regular embedding, hence defines a 
 pushforward in cohomology. The correspondence $\iota$ does not
 change $\cV$, therefore it defines an action on the cohomology of $\QM$ for
 fixed $\cV$. Since the image of $\iota$ is still the zero locus of
 the section $s(x)$ on the ambient space, the pushforward
 along a regular embedding may be equivalently defined in terms of the
 refined Euler class. As a result, the two definitions agree. 

 \subsubsection{}

 We define the maps 
 \begin{equation}
   \label{eq:313}
   \be, \bff : \Hd(C \times \QM) \xrightarrow{\quad
   } \Hd(\QM)
 \end{equation}
 by
 \begin{alignat}{2} 
  \be &= &&(\textstyle{\int_C} \otimes 1) \, \iota_* \,, \notag \\
   \bff &= -  &&(\textstyle{\int_C}\otimes 1) \, \iota^*\,,   \label{eq:322}
 \end{alignat}
 where $\textstyle{\int_C}$ denotes the pushforward along $C$. The
 operators \eqref{eq:322} increase the cohomological degree by $\pm
 2$, and the degree of $\cL$ by $\mp 1$, respectively. 

 As operators in \'etale cohomology, the operators $\be$ and
 $\bff$ also acquire a Galois action. We postpone this
 discussion to Section \ref{s_Galois_e}.

 \subsubsection{}

 Modifications at different points $x\ne x'$ commute, which means
 the commutator correspondence
\begin{equation}
  \label{eq:219}
  \xymatrix{
    \Hd( C \times C \times \QM) \ar[rr]^{ [ \be_{13}, \bff_{23}] }
    \ar[rd]_{\Delta^*_C}
    && \Hd(\QM) \\
    & \Hd( C \times \QM) \ar[ur]_{\bh} 
    }
\end{equation}
factors through the restriction $\Delta^*_C$ to the diagonal in
$C\times C$.
The subscripts in $ [ \be_{13}, \bff_{23}]$ mean that the operator
$\be$ is applied in the 1st and the 3d factor of the product, and
similarly for $\bff$. 
At this point, the diagram \eqref{eq:219} is the definition of the
operator $\bh$.

\subsubsection{}

It is clear that $\bh$ preserves cohomological degree and is given by a
correspondence that is supported over the diagonal in $\QM$. 
Therefore has the form
\begin{equation}
  \bh = (\textstyle{\int_C} \otimes 1) \circ
  \textup{cup product by a class in  $\bH^2(C \times \QM)$}
\,.\label{eq:223a}
\end{equation}
The divisor class in \eqref{eq:223a} will be computed presently. For
future reference, it will be convenient to list the tautological
divisor classes. We define
\begin{equation}
  \label{eq:235}
  \lambda = c_1(\cL)\,, \quad \mu = c_1(\cM)\,, \quad \tau =
  c_1(TC)\,, \quad \eta=c_1(\cO_{\QM}(1)) \,. 
\end{equation}
The exact sequence \eqref{eq:226} relates Chern classes of $\cV$ to
$c_1(\cM)$. In particular, we have
\begin{equation}
  \label{eq:38}
c_1(\cN) = \mu - 2 \lambda + 2 \eta \,. 
\end{equation}
 Being pulled back from $C$, the classes $\mu$ and $\tau$
 represent scalar operators in formulas like \eqref{eq:223a}.
 While in
a nonequivariant situation we have $\mu^2=\tau^2=0$, in an equivariant
situation both of these classes are nonzerodivisors. The ability to
divide by $\tau$ in a equivariant situation will come in handy in the
proof of Proposition \ref{p_sl2} below. 

\subsubsection{}

\begin{Proposition}\label{p_sl2}
We have
\begin{equation}
\bh = (\textstyle{\int_C} \otimes 1) \circ (\tau + c_1(\cN) ) 
\,.\label{eq:223}
\end{equation}
The operators $\be, \bff, \bh$ generate the Lie algebra
$\fsl_2$ in the sense that
\begin{equation}
  \label{eq:220}
  [\be, \bff] = \bh\, \Delta_C^* \,, \quad [\bh,\be] = 2 \be \, \Delta_C^* \,,
  \quad [\bh,\bff] = - 2 \bff\, \Delta_C^* \,. 
\end{equation}
\end{Proposition}

\begin{proof}
Consider the pushforward map
  \[
 \Delta_{C,*}:  \Hd( C \times \QM) \to \Hd( C \times C \times \QM)
 \,. 
 \]
 Since the composition of the pushforward and pullback is the multiplication
 by the Euler class of the normal bundle, we have 
 \[
 \Delta_{C}^* \circ \Delta_{C,*} = c_1(N_{C^2/\Delta C}) = c_1(TC) =
 \tau 
 \,. 
\]
Therefore
\begin{equation}
  \label{eq:25*3}
  \tau \bh = [ \be, \bff] \circ \Delta_{C,*} \,. 
\end{equation}
By construction of $\iota_*$, we have
\[
\iota^*  \iota_{*}  \circ \Delta_{C,*} =
\iota^*(\bar \varepsilon) \cup \,.
\]
where the refined Euler class $\varepsilon$ was defined in
\eqref{eq:233} and $\bar \varepsilon$ is the usual unrefined Euler class. To
compute the composition in the other direction, we use the factorization
\eqref{eq:91}, which gives
\[
\iota_*  \iota^* \circ \Delta_{C,*} =
\bar \varepsilon \cup \,.
\]
Therefore, 
  \begin{equation}
    [\iota^* , \iota_{*} ] \circ \Delta_{C,*} =
    (\iota^*(\bar\varepsilon) - \bar\varepsilon )  \cup \,, \label{eq:323}
 \end{equation}
 We have
 \begin{equation}
 \iota^*(\cN)  = \cN
 \otimes T C \,,\label{eq:26*2}
 \end{equation}
  whence 
  \begin{align}
 [\iota^* , \iota_{*} ] \circ \Delta_{C,*}  &= c_2(\cN
 \otimes T C) - c_2(\cN)\notag \\
    & = \tau(c_1(\cN)+\tau) \,. \label{eq:224}
\end{align}
To complete the proof of \eqref{eq:223}, it
remains to argue that we can divide by $\tau$ in
\eqref{eq:25*3} and \eqref{eq:224}. Clearly, the divisor in
\eqref{eq:223a} is  a linear combination of all possible divisors
\eqref{eq:235} involved, which is universal, that is, independent of the curve $C$.
Since our computations are fully equivariant,
the formula should also be true for $C=\A^1$ with the defining action
of $\Gm$. In this case, $\tau$ is the
 generator of $\Hd_{\Gm}(C,R)\cong R[\tau]$, where $R$ is the ring of
 coefficients, and is not a zerodivisor. This completes the proof of
 \eqref{eq:223}.

The other
commutators in \eqref{eq:220} follow from \eqref{eq:26*2}.

\end{proof}

\subsubsection{}
To convert the relations \eqref{eq:220} into a more familiar form, we
consider the map
\begin{equation}
 \fsl_2 \otimes \Hd(C)= \Hd(C,\fsl_2) \xrightarrow{\quad}  \End(\Hd({\textsf{QM}}))\label{eq:35}
 \end{equation}
such that
\begin{equation}
  \label{eq:34}
  \begin{pmatrix}
    0 & \gamma \\
    0 & 0 
  \end{pmatrix}
   \mapsto \be \of{\gamma}
  \overset{\textup{\tiny def}}=\be(\gamma \boxtimes \, \cdot \,) \,,
  \quad \gamma \in \Hd(C) \,, 
\end{equation}
and similarly for $\bff$ and $\bh$. Both sides in \eqref{eq:35} are
Lie superalgebras
with the $\Z/2$-grading determined by
the cohomological grading. From \eqref{eq:220} we deduce

\begin{Proposition}\label{p_sl22}
The map \eqref{eq:34} is a homomorphism of Lie superalgebras.
\end{Proposition}

\subsubsection{}\label{s_cAcE2}
Since the computations in Proposition \ref{p_sl2} are fully
equivariant, they determine the algebra $\cA(\A^1/\Gm)$ mentioned in Section
\ref{s_cAcE}  as the following quadratic extension of the
graded version of $\fsl_2\otimes R[\tau]$, where $R$ is the ring of
coefficients. 

To have a well-defined pushforward along $C$, we need to either
localize or work with compactly supported classes. Choosing the second
options, we fix the generator $\bpt \in H^2_c(\A^1/\Gm)$ such
that $\int_C \bpt = 1$. Since
\[
  \bpt^2 = \tau \bpt\,, 
\]
we see that $\fsl_2 \otimes H^2_c(\A^1/\Gm)$ is a graded version of
$\fsl_2$ over $R[\tau]$  in the sense that
\[
  [ \be\of{\bpt} , \bff\of{\bpt} ] = \tau \bh\of{\bpt}\,, 
\]
et cetera. Adding the element $\mu$, we see from Proposition
\ref{p_Cas} below that
\begin{equation}
  \label{eq:39}
  \cA(\A^1/\Gm) = \cU_\textup{graded}(\fsl_2) [\mu] \Big/
  \left(\mu^2 = \textup{Casimir element} + \tau^2 \right)\,. 
\end{equation}

\subsubsection{}

We note that Propositions \ref{p_sl2} and \ref{p_sl22} apply verbatim
to the moduli spaces $\QMN$ in place of $\QM$. Indeed, the
correspondences defining the $\fsl_2$-action act over the space
$H^1(\cM)$ of possible extensions $\cV$ and commute with the action of
$H^0(\cM)$. From now on, we will focus on the compactly
supported cohomology $\Hd_c(\QMN)$. 


\subsection{Casimir elements and the mirror $\bX^\vee$} 

\subsubsection{}

As an immediate corollary of the computation done in \eqref{eq:224},
we obtain the following

\begin{Proposition}\label{p_Cas} We have the equality 
  \begin{equation}
 \big( 2 \, \be_{13} \, \bff_{23} + 2 \, \bff_{23} \, \be_{13}  +
 \bh_{13} \bh_{23}  \big) \circ \Delta_{C,*} =
 \mu^2 - \tau^2\label{eq:36}
\end{equation}
of operators from $\Hd_c(C \times \QMN)$ to $\Hd_c(\QMN)$. 
\end{Proposition}

\noindent 
To convert this to a more familiar form, let $\bpt \in
H^2(C)$ be the generator that comes from the pushforward of $1\in
H^0(\pt)$ under 
$\pt \to C$. We may use the relations
\begin{align*}
  \Delta_{C,*} (\bpt) &= \bpt \otimes \bpt\,, \\
  \Delta_{C,*} (\alpha) &= \alpha \otimes \bpt + \bpt
                          \otimes \alpha \,, \quad \alpha \in H^1\,, 
\end{align*}
to deduce the following

\begin{Corollary}
  We have
 \begin{align}
    \label{eq:306}
    4 \, \be\of{\bpt} \, \bff\of{\bpt} + \bh\of{\bpt}^2 &=
                                                                            0
                                                                            \,,\\
                                                                            2 \, \be\of{\bpt} \, \bff\of{\alpha} + 2 \, \bff\of{\bpt} \, \be\of{\alpha}
                                                                            +
                                                                            \bh\of{\bpt}\,  \bh\of{\alpha}
   &=0 \,,       \label{eq:306-2}                                                  
 \end{align}
 where $\alpha \in H^1(C)$. 
\end{Corollary}

\noindent
In \eqref{eq:306-2}, we have divided by $2$ since for the point counts
we may assume that $2$ is invertible in the ring of coefficients $R$.


\subsubsection{}
It will be useful to have a formula for the action of the operators
$\bh$.

\begin{Proposition}\label{p_h} 
  We have
  \begin{equation}
    \label{eq:58}
    \bh\of\bpt = 2 \eta \,, \quad \bh\of\alpha = 2 \alpha\,, 
  \end{equation}
  viewed as operators of cup product by $\eta=c_1(\cO(1))$ and
  the pullback of $\alpha \in H^1(C) = H^1(\Pic C)$ to $\QMN$. 
\end{Proposition}

\begin{proof}
This follows from \eqref{eq:223}, \eqref{eq:38} and the following
formula
\begin{equation}
     \label{eq:56}
     \lambda = (\deg \cL)\cdot \bpt \otimes 1 +  \sum \gamma_i \otimes
\gamma_i^\vee \,, 
\end{equation}
for the class of $\lambda \in H^2(C \times \QMN)$. Here
$\{\gamma_i^\vee\}$ is the dual basis to a
basis $\{\gamma_i\}$ of
$H^1(C)$ in the sense that $\gamma_i \cup \gamma^\vee_j = \delta_{ij}
\bpt$. Indeed, for $\deg \cL =0$, the universal bundle $\cL$ is the pullback
via $C \hookrightarrow \Pic_0(C)$ of the Poincar\'e line bundle, whose first Chern
class is 
\[
\sum \gamma_i \otimes
\gamma_i^\vee = \Delta \Theta - \Theta \otimes 1 - 1 \otimes \Theta
\,.
\]
Here $\Theta$ is the theta divisor and $\Delta$ is the comultiplication in the cohomology of the
group $\Pic_0(C)$. To obtain other values of $\deg \cL$, we may
multiply by line bundles pulled back from $C$. 
\end{proof}

\subsubsection{}

Formulas \eqref{eq:306} and \eqref{eq:306-2} may be given the
following geometric interpretation. The algebra
\begin{equation}
  \label{eq:37*2}
  \cA(C) = \cU(\Hd(C,\fsl_2)) /  \textup{equations
    \eqref{eq:36}}\,, 
\end{equation}
acts in $\Hd(\QMN)$. Inside this algebra we have subalgebras
\begin{equation}\label{eq:38*2}
  \cA(C) \supset \cA^{\ge 1}(C) \supset \cA^{2}= \textup{center of
    $\cA^{\ge 1}(C)$}\,, 
\end{equation}
generated by elements of cohomological degree $\ge 1$ and $2$,
respectively. The notation here reflects the fact that the
subalgebra $\cA^2$ does not depend on the curve
$C$. Equation \eqref{eq:306} says that
\begin{align}
 \bX^\vee_0 &\eqdef  \Spec \cA^{2}  \notag \\
 & = \textup{quadric cone}  \,,  \label{eq:39*2}
\end{align}
while \eqref{eq:306-2} says that 
\begin{equation}
  \label{eq:40*2}
   \textup{$\cA^{2}$-span of  $H^1(C,\fsl_2)$} \cong
   H^1(C,\Omega^1(\bX^\vee_0)))\,, 
\end{equation}
where $\Omega^1(\bX^\vee_0)$ denotes the module of K\"ahler
differentials over the base field. For the time being, $\Omega^1(\bX^\vee_0)$ forms a trivial
local system on $C$. Nontrivial local system will appear in Section
\ref{s_Galois_e}.

\subsubsection{}

A shift is required to make the isomorphism \eqref{eq:40*2} preserve
cohomological dimension and the action of Frobenius. Indeed, we have 
\[
  \deg \be\of{\alpha} = \deg \be\of{\bpt}-1\,, \quad
   \textup{weight}\, \be\of{\alpha} = \textup{weight} \, 
   \be\of{\bpt}-\tfrac12\,, 
 \]
where the degree is the cohomological degree and weight refers to the
absolute values of the Frobenius eigenvalues. So, if we want to get
these operators as $H^1$, we need both a shift in cohomology and a
Tate twist. We define
\begin{equation}
  \label{eq:42}
  H\Omega_0 =  H^1(C,\Omega^1(\bX^\vee_0)))\sTate{2}\,, 
\end{equation}
where
\[
  \sTate{2}=[2]\Tate{1}
\]
denotes the pure combination of the cohomological shift and 
Tate twist, and 
where $\Omega^1$ was originally placed in degree $0$, compare with
\eqref{eq:6}. 
{}From the Poisson structure on the cone and multiplication in cohomology, we get a symmetric
pairing 
\begin{equation}
  \label{eq:41}
  S^2 H\Omega_0
  \xrightarrow{\quad  \cup \otimes \{ \, \cdot \, , \, \cdot \, \}
    \quad}
  H^2(C,\cO_{\bX^\vee_0}) \sTate{2}=\cO_{\bX^\vee_0}= \cA^2 \,. 
\end{equation}
We define
\[
\Cliff_{\bX^\vee_0} = \textup{Clifford}(H\Omega_0) \,.
\]
Formulas \eqref{eq:306} and \eqref{eq:306-2}, and the naturality of
all operations, give the following 

\begin{Proposition}\label{p_sl2_Cliff} 
  We have a degree-preserving
  isomorphism $ \cA^{\ge 1}(C)\cong \Cliff_{\bX^\vee_0} $. Moreover,
  this isomorphism is Frobenius-equivariant. 
\end{Proposition}

\subsubsection{}
We define
\begin{equation}
  \label{eq:42*2}
 \sigma: \bX^\vee \cong T^* \bP^1 \to \bX^\vee_0 
\end{equation}
as the blow-up at the vertex of the cone. This is the world's most
basic equivariant symplectic resolution. As above, $\bX^\vee$
gets its own sheaf of Clifford algebras $\Cliff_{\bX^\vee}$ with
a homomorphism
\[
  \sigma^*: \Cliff_{\bX^\vee_0}  \to \Cliff_{\bX^\vee} \,. 
\]
There is thus a well-defined pushforward $\sigma_*$ of Clifford modules.

In modern representation theory, understanding the representation of
quantizations of Poisson varieties via (derived) pullbacks and
pushforwards under maps like $\sigma$ is a very important technical
tool known as \emph{localization}. Our situation is simpler in that
we deal with sheaves of Clifford algebras instead of really
noncommutative algebras like the full $\cA(C)$. In the end,
however, the full structure of the $\cA(C)$-module is uniquely
determined.

We will show, and this will be the main result of this section, that
the module $\Hd(\QMN)$ is the pushforward of a suitable twist of the spinor bundle
$\Spin_{\bX^\vee}$. This will be, however, not the
ordinary pushforward, but rather pushforward with support in
a certain Lagrangian $\bfL \subset \bX^\vee$.

\subsubsection{}

It will be convenient to visualize various geometric notions
using the toric diagram of $\bX^\vee=T^*\bP^1$ as seen on the right in Figure
\eqref{eq:301} below. It is also
known as the Newton or moment polyhedron of $T^*\bP^1$. Its two vertices
correspond to the fixed points $0,\infty \in \bP^1$. Edges
correspond to the 1-dimensional torus orbits.

Canonically, diagrams like \eqref{eq:301} lie in the 
dual of the Lie algebra of a maximal torus
\[
\{(a,q)\} \subset \Aut(T^*\bP^1)
\]
where $a$ acts by $\diag(a^{-1},a) \in SL(2)$ and $q$ scales the
cotangent fibers with weight $q^{-1}$. Lattice points in toric
diagram are the weights of sections of line 
bundles. Specifically, in  \eqref{eq:301} we see the
weights $H^0(\cO_{\bX^\vee})$ and
$H^0(\cO_{\bX^\vee}(8))$. Here $\cO_{\bX^\vee}(8)$ is a random ample
line bundle. For the natural linearization, the weights of the corner
vertices for $\cO_{\bX^\vee}(8)$ are $a^{8}$ and $a^{-8}$,
respectively. In Figure \ref{eq:301} we indicated the weights of
neighboring vertices relative to the weights of the corner vertices. 
\begin{equation}
  \label{eq:301}
  \raisebox{-1.5cm}{
    \begin{tikzpicture}[scale=1]
      \draw [line width=1pt] (-8.2,2.2)--(-6,0)--(-3.8,2.2);
   \fill [fill=gray!20] (-8.2,2.2)--(-6,0)--(-3.8,2.2)--(-8.2,2.2);
    %
\node at (-6,0) {${\scriptstyle \bullet}$};
\node at (-13/2,1/2) {${\scriptstyle \bullet}$};
\node at (-6,1/2) {${\star}$};
\node at (-11/2,1/2) {${\scriptstyle \bullet}$};
\node at (-7,1) {${\scriptstyle \bullet}$};
\node at (-13/2,1) {${\scriptstyle \bullet}$};
\node at (-6,1) {${\scriptstyle \bullet}$};
\node at (-11/2,1) {${\scriptstyle \bullet}$};
\node at (-5,1) {${\scriptstyle \bullet}$};
\node at (-15/2,3/2) {${\scriptstyle \bullet}$};
\node at (-7,3/2) {${\scriptstyle \bullet}$};
\node at (-13/2,3/2) {${\scriptstyle \bullet}$};
\node at (-6,3/2) {${\scriptstyle \bullet}$};
\node at (-11/2,3/2) {${\scriptstyle \bullet}$};
\node at (-5,3/2) {${\scriptstyle \bullet}$};
\node at (-9/2,3/2) {${\scriptstyle \bullet}$};
\node at (-8,2) {${\scriptstyle \bullet}$};
\node at (-15/2,2) {${\scriptstyle \bullet}$};
\node at (-7,2) {${\scriptstyle \bullet}$};
\node at (-13/2,2) {${\scriptstyle \bullet}$};
\node at (-6,2) {${\scriptstyle \bullet}$};
\node at (-11/2,2) {${\scriptstyle \bullet}$};
\node at (-5,2) {${\scriptstyle \bullet}$};
\node at (-9/2,2) {${\scriptstyle \bullet}$};
\node at (-4,2) {${\scriptstyle \bullet}$};
\node[rectangle,anchor=east, inner sep=0] (2) at (-6.7,0.5)  {$\bff\of{\bpt}$};
    \node[rectangle,anchor=west,inner sep=0] (3) at (-5.3,0.5) {$\be\of{\bpt}$};
    \draw [line width=1pt] (-2.2,2.2)--(0,0)--(4,0)--(6.2,2.2);
    \fill [fill=gray!20] (-2.2,2.2)--(0,0)--(4,0)--(6.2,2.2)--(-2.2,2.2);
    \node[rectangle,anchor=east, inner sep=0] (2) at (-0.73,0.5)  {$qa^{2}$};
    \node[rectangle,anchor=west,inner sep=0] (3) at (4.7,0.5) {$qa^{-2}$};
   \node[rectangle,anchor=north,inner sep=0] (4) at (0.55,-0.1)
   {$a^{-2}$};
   \node[rectangle,anchor=north,inner sep=0] (5) at (3.5,-0.1)
   {$a^{2}$};
   \node[rectangle,anchor=north east,inner sep=1] (6) at (0,0)
   {$0$};
    \node[rectangle,anchor=north west,inner sep=1] (7) at (4,0)
    {$\infty$};
    %

    \draw[ultra thick, ->] (-3,0.7) arc (-60:-300:0.5); 
    
    \node at (0,0) {${\scriptstyle \bullet}$};
\node at (1/2,0) {${\scriptstyle \bullet}$};
\node at (1,0) {${\scriptstyle \bullet}$};
\node at (3/2,0) {${\scriptstyle \bullet}$};
\node at (2,0) {${\scriptstyle \bullet}$};
\node at (5/2,0) {${\scriptstyle \bullet}$};
\node at (3,0) {${\scriptstyle \bullet}$};
\node at (7/2,0) {${\scriptstyle \bullet}$};
\node at (4,0) {${\scriptstyle \bullet}$};
\node at (-1/2,1/2) {${\scriptstyle \bullet}$};
\node at (0,1/2) {${\scriptstyle \bullet}$};
\node at (1/2,1/2) {${\scriptstyle \bullet}$};
\node at (1,1/2) {${\scriptstyle \bullet}$};
\node at (3/2,1/2) {${\scriptstyle \bullet}$};
\node at (2,1/2) {${\scriptstyle \bullet}$};
\node at (5/2,1/2) {${\scriptstyle \bullet}$};
\node at (3,1/2) {${\scriptstyle \bullet}$};
\node at (7/2,1/2) {${\scriptstyle \bullet}$};
\node at (4,1/2) {${\scriptstyle \bullet}$};
\node at (9/2,1/2) {${\scriptstyle \bullet}$};
\node at (-1,1) {${\scriptstyle \bullet}$};
\node at (-1/2,1) {${\scriptstyle \bullet}$};
\node at (0,1) {${\scriptstyle \bullet}$};
\node at (1/2,1) {${\scriptstyle \bullet}$};
\node at (1,1) {${\scriptstyle \bullet}$};
\node at (3/2,1) {${\scriptstyle \bullet}$};
\node at (2,1) {${\scriptstyle \bullet}$};
\node at (5/2,1) {${\scriptstyle \bullet}$};
\node at (3,1) {${\scriptstyle \bullet}$};
\node at (7/2,1) {${\scriptstyle \bullet}$};
\node at (4,1) {${\scriptstyle \bullet}$};
\node at (9/2,1) {${\scriptstyle \bullet}$};
\node at (5,1) {${\scriptstyle \bullet}$};
\node at (-3/2,3/2) {${\scriptstyle \bullet}$};
\node at (-1,3/2) {${\scriptstyle \bullet}$};
\node at (-1/2,3/2) {${\scriptstyle \bullet}$};
\node at (0,3/2) {${\scriptstyle \bullet}$};
\node at (1/2,3/2) {${\scriptstyle \bullet}$};
\node at (1,3/2) {${\scriptstyle \bullet}$};
\node at (3/2,3/2) {${\scriptstyle \bullet}$};
\node at (2,3/2) {${\scriptstyle \bullet}$};
\node at (5/2,3/2) {${\scriptstyle \bullet}$};
\node at (3,3/2) {${\scriptstyle \bullet}$};
\node at (7/2,3/2) {${\scriptstyle \bullet}$};
\node at (4,3/2) {${\scriptstyle \bullet}$};
\node at (9/2,3/2) {${\scriptstyle \bullet}$};
\node at (5,3/2) {${\scriptstyle \bullet}$};
\node at (11/2,3/2) {${\scriptstyle \bullet}$};
\node at (-2,2) {${\scriptstyle \bullet}$};
\node at (-3/2,2) {${\scriptstyle \bullet}$};
\node at (-1,2) {${\scriptstyle \bullet}$};
\node at (-1/2,2) {${\scriptstyle \bullet}$};
\node at (0,2) {${\scriptstyle \bullet}$};
\node at (1/2,2) {${\scriptstyle \bullet}$};
\node at (1,2) {${\scriptstyle \bullet}$};
\node at (3/2,2) {${\scriptstyle \bullet}$};
\node at (2,2) {${\scriptstyle \bullet}$};
\node at (5/2,2) {${\scriptstyle \bullet}$};
\node at (3,2) {${\scriptstyle \bullet}$};
\node at (7/2,2) {${\scriptstyle \bullet}$};
\node at (4,2) {${\scriptstyle \bullet}$};
\node at (9/2,2) {${\scriptstyle \bullet}$};
\node at (5,2) {${\scriptstyle \bullet}$};
\node at (11/2,2) {${\scriptstyle \bullet}$};
\node at (6,2) {${\scriptstyle \bullet}$};
  \end{tikzpicture}}
\end{equation}
Note that $H^0(\cO_{\bX^\vee})=\cA^2$. We have marked the
weights of $\bff\of{\bpt}$ and $\be\of{\bpt}$, while
the weight of $\bh\of{\bpt}$ is marked by a star $\star$.
Since 
$H^0(\cO_{\bX^\vee}(8))$ is a module over $H^0(\cO_{\bX^\vee})$,
the left
part of \eqref{eq:301} acts on the right part by addition, as 
indicated by the curved arrow.

As customary in toric geometry, one may view the shapes in Figure
\eqref{eq:301} as a combinatorial equivalent of the geometric shapes
of the cone $\bX^\vee_0$ and its resolution $\bX^\vee$, respectively.

\subsubsection{}

We define 
\begin{align}
  \label{eq:328}
  \bfL &= \{ \bff\of{\bpt} = 0 \}  \\
  \label{eq:336}
  &= \{ x \in \bX^\vee, \exists \lim_{a\to 0} a\cdot x \} \,. 
\end{align}
This is the union of two attracting manifolds, namely 
\begin{alignat}{2}
  \label{eq:337}
  \bfL_0 &= \{ x \in \bX^\vee, \lim_{a\to 0} a\cdot x =
  0\}&&=\textup{zero section without $\infty$} \,, \\
  \bfL_\infty &= \{ x \in \bX^\vee, \lim_{a\to 0} a\cdot x =
  \infty\}&&=\textup{the fiber over $\infty$} \,,
\end{alignat}
see Figure \eqref{eq:301-2}.

\subsubsection{}\label{s_local coh} 

To get a feeling for what the pushforward with support in $\bfL$ looks
like, let us compute it for $\cO_{\bX^\vee}(8)$, which means we
compute 
consider the local homology groups $\Hd_{\bfL}(\cO_{\bX^\vee}(8))$.
The exact sequence of a pair for $\bfL_\infty \subset \bfL$ expresses
$\Hd_{\bfL}$ in terms of $\Hd_{\bfL_\infty}$ and
$\Hd_{\bfL_0}$. Both of these can be computed in a
single affine chart. This gives
 \[
  H^i_{\bfL_\infty}(\cO(8))= H^i_{\bfL_0}(\cO(8))=0 \,, \quad i \ne 1 \,,
  \] 
and 
\begin{equation}
  \label{eq:329}
  \xymatrix{
    0 \ar[r] & H^1_{\bfL_\infty}(\cO(8))  \ar@{=}[d] \ar[r] & H^1_{\bfL}(\cO(8))
    \ar[r] & H^1_{\bfL_0}(\cO(8)) \ar[r]
    \ar@{=}[d] & 0\,,\\
    &
  a^{8} (\be \bh^{-1})^{8}  \, R[\be, (\bff \bh^{-1})^{\pm 1}]/ R[\be , \bff \bh^{-1}]
    &&
    a^{8} R[\be \bh^{-1}, \bff^{\pm 1}]/ R[\be \bh^{-1}, \bff]
    }
  \end{equation}
  where $R$ is base ring of coefficients. See the
  picture of the corresponding modules in Figure
  \eqref{eq:301-2}, where the submodule is shaded in a darker shade of
  gray. The middle term \eqref{eq:329} is the unique
  nontrivial extension between them.
\begin{equation}
  \label{eq:301-2}
  \raisebox{-1cm}{\begin{tikzpicture}[scale=1]
      \draw [line width=0.5pt] (-2.2,2.2)--(0,0);
      \draw [line width=2pt] (0,0)--(3.85,0);
      \draw [line width=2pt] (4,0)--(6.2,2.2);
      \node at (0,0) {${\scriptstyle \bullet}$};
      \node at (4,0) {${\scriptstyle \bullet}$};
      \fill [fill=gray!40]
      (6.7,2.2)--(4.5,0)--(7.7,0)--(7.7,2.2)--(6.7,2.2);
      \fill [fill=gray!20]
      (2.2,-2.2)--(0.5,-0.5)--(7.7,-0.5)--(7.7,-2.2)--(2.2,-2.2); 

      \node[rectangle,anchor=south, inner sep=0] at (2.1,0.1)
      {$\bfL_0$};
      \node[rectangle,anchor=south east, inner sep=0] at (5,1.1) {$\bfL_\infty$}; 
      
    %
 \node at (1/2,-1/2) {${\scriptstyle \bullet}$};
\node at (1,-1/2) {${\scriptstyle \bullet}$};
\node at (3/2,-1/2) {${\scriptstyle \bullet}$};
\node at (2,-1/2) {${\scriptstyle \bullet}$};
\node at (5/2,-1/2) {${\scriptstyle \bullet}$};
\node at (3,-1/2) {${\scriptstyle \bullet}$};
\node at (7/2,-1/2) {${\scriptstyle \bullet}$};
\node at (4,-1/2) {${\scriptstyle \bullet}$};
\node at (9/2,-1/2) {${\scriptstyle \bullet}$};
\node at (5,-1/2) {${\scriptstyle \bullet}$};
\node at (11/2,-1/2) {${\scriptstyle \bullet}$};
\node at (6,-1/2) {${\scriptstyle \bullet}$};
\node at (13/2,-1/2) {${\scriptstyle \bullet}$};
\node at (7,-1/2) {${\scriptstyle \bullet}$};
\node at (15/2,-1/2) {${\scriptstyle \bullet}$};
\node at (1,-1) {${\scriptstyle \bullet}$};
\node at (3/2,-1) {${\scriptstyle \bullet}$};
\node at (2,-1) {${\scriptstyle \bullet}$};
\node at (5/2,-1) {${\scriptstyle \bullet}$};
\node at (3,-1) {${\scriptstyle \bullet}$};
\node at (7/2,-1) {${\scriptstyle \bullet}$};
\node at (4,-1) {${\scriptstyle \bullet}$};
\node at (9/2,-1) {${\scriptstyle \bullet}$};
\node at (5,-1) {${\scriptstyle \bullet}$};
\node at (11/2,-1) {${\scriptstyle \bullet}$};
\node at (6,-1) {${\scriptstyle \bullet}$};
\node at (13/2,-1) {${\scriptstyle \bullet}$};
\node at (7,-1) {${\scriptstyle \bullet}$};
\node at (15/2,-1) {${\scriptstyle \bullet}$};
\node at (3/2,-3/2) {${\scriptstyle \bullet}$};
\node at (2,-3/2) {${\scriptstyle \bullet}$};
\node at (5/2,-3/2) {${\scriptstyle \bullet}$};
\node at (3,-3/2) {${\scriptstyle \bullet}$};
\node at (7/2,-3/2) {${\scriptstyle \bullet}$};
\node at (4,-3/2) {${\scriptstyle \bullet}$};
\node at (9/2,-3/2) {${\scriptstyle \bullet}$};
\node at (5,-3/2) {${\scriptstyle \bullet}$};
\node at (11/2,-3/2) {${\scriptstyle \bullet}$};
\node at (6,-3/2) {${\scriptstyle \bullet}$};
\node at (13/2,-3/2) {${\scriptstyle \bullet}$};
\node at (7,-3/2) {${\scriptstyle \bullet}$};
\node at (15/2,-3/2) {${\scriptstyle \bullet}$};
\node at (2,-2) {${\scriptstyle \bullet}$};
\node at (5/2,-2) {${\scriptstyle \bullet}$};
\node at (3,-2) {${\scriptstyle \bullet}$};
\node at (7/2,-2) {${\scriptstyle \bullet}$};
\node at (4,-2) {${\scriptstyle \bullet}$};
\node at (9/2,-2) {${\scriptstyle \bullet}$};
\node at (5,-2) {${\scriptstyle \bullet}$};
\node at (11/2,-2) {${\scriptstyle \bullet}$};
\node at (6,-2) {${\scriptstyle \bullet}$};
\node at (13/2,-2) {${\scriptstyle \bullet}$};
\node at (7,-2) {${\scriptstyle \bullet}$};
\node at (15/2,-2) {${\scriptstyle \bullet}$};
\node at (9/2,0) {${\scriptstyle \bullet}$};
\node at (5,0) {${\scriptstyle \bullet}$};
\node at (11/2,0) {${\scriptstyle \bullet}$};
\node at (6,0) {${\scriptstyle \bullet}$};
\node at (13/2,0) {${\scriptstyle \bullet}$};
\node at (7,0) {${\scriptstyle \bullet}$};
\node at (15/2,0) {${\scriptstyle \bullet}$};
\node at (5,1/2) {${\scriptstyle \bullet}$};
\node at (11/2,1/2) {${\scriptstyle \bullet}$};
\node at (6,1/2) {${\scriptstyle \bullet}$};
\node at (13/2,1/2) {${\scriptstyle \bullet}$};
\node at (7,1/2) {${\scriptstyle \bullet}$};
\node at (15/2,1/2) {${\scriptstyle \bullet}$};
\node at (11/2,1) {${\scriptstyle \bullet}$};
\node at (6,1) {${\scriptstyle \bullet}$};
\node at (13/2,1) {${\scriptstyle \bullet}$};
\node at (7,1) {${\scriptstyle \bullet}$};
\node at (15/2,1) {${\scriptstyle \bullet}$};
\node at (6,3/2) {${\scriptstyle \bullet}$};
\node at (13/2,3/2) {${\scriptstyle \bullet}$};
\node at (7,3/2) {${\scriptstyle \bullet}$};
\node at (15/2,3/2) {${\scriptstyle \bullet}$};
\node at (13/2,2) {${\scriptstyle \bullet}$};
\node at (7,2) {${\scriptstyle \bullet}$};
\node at (15/2,2) {${\scriptstyle \bullet}$};
  \end{tikzpicture}}
\end{equation}
\noindent
In \eqref{eq:329}, we picked a specific linearization of
$\cO(8)$. Choosing a different linearization translates the whole
picture \eqref{eq:301-2}, while replacing $\cO(8)$ by a general $\cO(m)$
makes the two shaded pieces slide relative each other.

The diagram \eqref{eq:301-2} visualizes the structure of
$\Hd_{\bfL}(\cO_{\bX^\vee}(m))$ as a $\cA^2$-module. In particular,
$\bff$ and $\bh$ act locally nilpotently, while $\be$-torsion is
finite-dimensional. Evidently, one can also describe the result of
the computation as follows
\[
H^1_{\bfL}(\cO(8)) = H^0_{\bfL}(\cO(8))[\bff^{-1}] \big/
H^0_{\bfL}(\cO(8)) \,.
\]

\subsubsection{}

As a result of taking the cohomology with support in the
attracting manifold $\bfL$, the $a$-grading of all resulting modules
is bounded from above. In the Eisenstein series context, the variable $a$ weights
a quasisection $s$ by $a^{-\deg(s)}$. Since $\deg(s)$ is bounded below
for fixed $\cM$, it makes sense to look for a description of the
$\cA^{\ge 1}(C)$-module $\Hd_c(\QMN)$ as a
pushforward with support in $\bfL$.


\subsubsection{}
In what follows, we will always include a cohomological shift by the
number $\delta=\frac12 \dim \Fix$ in the
definition of $\sigma_{*,\bfL}$. With this definition, a line bundle
on $\Fix$ pushes forward to object in zero cohomological degree. 

\subsection{Galois action}\label{s_Galois_e}

\subsubsection{} 

For applications to Eisenstein series, we will need cohomology of $\QMN$
with coefficients in local systems. Via the natural map 
\begin{equation}
  \label{eq:296}
  \QMN(s: C \to \cP) \xrightarrow{\quad \NN_s \quad } \Pic(C) 
\end{equation}
we can pull back a local system $\chi$ from $\Pic(C)$. We associate a
local systems on $\Pic(C)$ to characters $\chi$ of $\Pic(C)$ by requiring
that the geometric
Frobenius acts on the fiber over $\cO_C(x)\in \Pic_1(C)$ with eigenvalue
$\chi(x)$.

\subsubsection{}\label{s_Galois_SL2}
We have
\begin{equation}
  \label{eq:297}
  \NN_{s(x)} = \NN_s \otimes \cO_C(x)^2 \,. 
\end{equation}
Therefore, we have a well-defined degree-preserving
Galois-equivariant correspondences
\begin{alignat}{2}
  &\Hd_c(C \times \QMN, &\chi^2\,&\boxtimes \chi) 
  \xrightarrow{\quad \be\quad } \Hd_c(\QMN, \chi)\sTate{2} \,, \notag \\
  &\Hd_c(C \times \QMN, & \Q_\ell&\boxtimes \chi)  
  \xrightarrow{\quad \bh\quad } \Hd_c(\QMN, \chi) \,,  \label{eq:298} \\ 
  &\Hd_c(C \times \QMN, &\chi^{-2} &\boxtimes \chi)  
  \xrightarrow{\quad \bff\quad } \Hd_c(\QMN, \chi)\sTate{-2} \,. \notag
\end{alignat}
The cohomological shifts and Tate twists in these formulas appear
because pushforwards in the construction of $\be$ and
$\bff$ introduce a shift by $\sTate{2\delta}$, where 
$\delta$ is the difference of
(virtual) dimensions between their target and domain.

\subsubsection{}

To get rid of the cohomological shifts in the right-hand side of
\eqref{eq:298}, we will center the cohomology of $\QMN$ around the
middle degree
 \begin{equation}
 \virdim \QMN = - 2 \deg \cL \,. \label{eq:308}
 \end{equation}
 For the eventual comparison with the spinor bundles, it is
 a good idea to also symmetrize the module $\Hd_c(\QMN)$ with
 respect to its grading by the degree $\deg \cL$ and also
 with respect to the Galois action. Neither of these shifts
 affects the $\fsl_2$-action.

We define 
\begin{equation}
  \label{eq:222}
  \widehat{H}(\QMN,\chi) = \Hd_c(\QMN,\chi)\sTate{-2 \deg \cL+g-1} \otimes
\chi(\cK_C)^{-1}  \,. 
\end{equation}
where the additional twist comes from 
 \begin{equation}
   \label{eq:70}
\left(\Lambda^{\textup{top}} \Hd(C,\chi^2) \right)^{-1/2} = 
\chi(\cK_C)^{-1} \sTate{g-1} \,. 
 \end{equation}
 This is a character of the group $\Fr \times \bA\buu$ via the
 action of Frobenius in the fiber of the local system $\chi$ on $\Pic(C)$ over
 the canonical bundle $\cK_C$. As a function of $a$, it is a multiple of $a^{2g-2}$ in
 accordance with \eqref{eq:302}. See 
 \eqref{eq:68} below for more on the equality in \eqref{eq:70}.

 Note that the implicit dependence of $\chi(\cM)$ on $a \in \bA\buu$ will
 also have a centering effect for the grading by $\bA\buu$, equivalently,
 the grading by $\deg \cL$, which one can see in Lemma \ref{l_w_gen}
 below.


%

\subsubsection{}

Consider the Galois representation
\begin{equation}
  \label{eq:299}
  \begin{pmatrix}
    \chi \\
    & \chi^{-1}
  \end{pmatrix} \in SL(2) \,,
\end{equation}
and the corresponding local system $\fsl_{2}^\chi$ obtained by
composing \eqref{eq:299} with the adjoint representation.
K\"unneth decomposition in \eqref{eq:298} results in the following

\begin{Proposition}
  We have a Lie superalgebra action
  \begin{equation}
    \label{eq:43}
    \Hd(C, \fsl_2^\chi) \otimes \widehat{H}(\QMN) \to
    \widehat{H}(\QMN)\,, 
  \end{equation}
  which is Frobenius-equivariant and degree-preserving. 
\end{Proposition}

\noindent
Note the small, but important difference between \eqref{eq:299} and
\eqref{eq:210}. With the symmetric normalization 
\eqref{eq:222}, the shift by $\rho$ as in  $\chi_\rho =
\chi\Tate{-\tfrac12}$ is not required. Of course, the shift will reappear in
any computation with actual Eisenstein series, because point counts
are given by Frobenius eigenvalues in $\Hd_c$, not in $\widehat{H}$.
One can also say that centering as in \eqref{eq:222} provides a more geometric way
to introduce the shift by $\rho$.

More generally, for moduli spaces of maps $f: C\to \bY$ to any
target $\bY$, we have
\[
  \virdim = \deg f^*(T \bY) + \const = \left\langle c_1(\bY), \deg f
  \right \rangle + \const\,, 
\]
with the same constant $(1-g)\dim \bY$ in both instances. Thus
the coefficient $2$ in \eqref{eq:222} should, in general, be replaced by
$c_1(\bY)$.

\subsubsection{}
We denote by $\cA_\chi(C)$ the algebra generated by $\Hd(C, \fsl_2^\chi)$. 
As before, we have subalgebras 
\begin{equation}\label{eq:38*3}
  \cA_\chi(C) \supset \cA^{\ge 1}_\chi(C) \supset \cA^{2}_\chi 
\end{equation}
generated by elements of cohomological degree $\ge 1$ and $2$,
respectively. Proposition \ref{p_sl2_Cliff} generalizes as follows 

  \begin{Proposition}\label{p_sl2_Cliff_chi} 
  We have a Frobenius-equivariant, degree-preserving
  isomorphism
  \begin{equation}
    \label{eq:44}
    \cA^{\ge 1}_\chi(C)\cong \Cliff_{\bX^\vee_0,\Gamma} \,. 
  \end{equation}
\end{Proposition}

\subsubsection{} 

When $\chi_0^2 \ne 1$, there is a dramatic simplification in the
structure of these algebras. Indeed, the raising operators are
limited to
\begin{equation}
  \label{eq:303}
  \be\of{\gamma} \,  = \be(\gamma \boxtimes \, \cdot \,) \,,
  \quad \gamma \in H^1(C,\chi^2_\rho) \,, 
\end{equation}
and similarly for the lowering operators
$\bff\of{\gamma}$. In principle, the commutators of these operators
may involve $\bh\of{\bpt}$, which satisfies $\bh\of{\bpt}^2=0$ as
a corollary of \eqref{eq:306-2}. We will see that, in fact,
$\bh\of{\bpt}$ acts by $0$ in \eqref{eq:43}, except for one
exceptional case. 

\subsubsection{}
The vanishing $\bh\of{\bpt}$ of admits the following geometric
interpretation and generalization. If $\chi^2_0\ne 0$ then scheme-theoretic fixed locus
$\Fix_0 = \left(\bX^\vee_0\right)^{\Gamma'}$ is a length two scheme with
coordinate ring generated by $\bh\of{\bpt}$ modulo $\bh\of{\bpt}^2$.
The $\Gamma'$-fixed locus $\Fix = \left(\bX^\vee\right)^{\Gamma'}$ is,
however, reduced. Therefore, $\bh\of{\bpt}$ acts by zero in any
module in the image of $\sigma_*$, where
  \begin{equation}
  \sigma^*: \Cliff_{\bX^\vee_0,\Gamma}  \to \Cliff_{\bX^\vee,\Gamma}  \label{eq:43*2}
  \end{equation}
  is the pullback map. More generally, we have
  \begin{equation}
    \label{eq:45}
    \Big( \chi^2_0\ne 0 \Big) \quad \Rightarrow \quad 
    \be\bff, \bff\be
    \in \Ker \sigma^*\,. 
  \end{equation}
  %
  In particular, $\bh\of{\bpt} \in \Ker \sigma^*$ as a
  commutator of classes in $H^1(C, \fsl_2^\chi)$.

\subsection{Categorification of the L-genus}

\subsubsection{}

The Clifford algebra $\Cliff_{\bX,\Gamma}$ has a natural spinor bundle
$\Spin$ that may be constructed as follows. There is a canonical
extension
\[
  0 \to
  \Omega^1_\hor \bX^\vee \to \Omega^1 \bX^\vee \to
  \Omega^1_\ver \bX^\vee \to 0
\]
where
\[
\Omega^1_\hor \bX^\vee \cong p^* \left(T^*\bP^1 \right) \subset \Omega^1 \bX^\vee
\]
is the Lagrangian subbundle
formed by $1$-forms that vanish on the fibers of the projection $p:
\bX^\vee \to \bP^1$. It gives rise to an isotropic subbundle
\begin{equation}
  H\Omega_{\bX^\vee,\hor} = H^1\left(C, \Omega^1_\hor \bX^\vee \right)\sTate{2} \subset
  H\Omega_{\bX^\vee}\,, \label{eq:49} 
\end{equation}
and hence to the spinor bundle
\begin{equation}
  \label{eq:48}
  \Spin = \sqrt{\det} \otimes \Cliff_{\bX^\vee, \Gamma} \Big/
  \left(H\Omega_{\bX^\vee,\hor}\right) \,, 
\end{equation}
where we mod out by the left ideal generated by \eqref{eq:49} and 
\begin{equation}
  \label{eq:66}
  \det = \Lambda^{\textup{top}} H\Omega_{\bX^\vee,\hor}
  = \left( \Lambda^{\textup{top}} H\Omega_{\bX^\vee,\ver}\right)^{-1}\,. 
\end{equation}
%
The existence of an equivariant square root is clear both when $\chi_0^2=1$ and
when $\chi_0^2\ne 1$. In the former case, we have the cohomology with
coefficients in a trivial local system, while in the latter case
$\dim \Fix=0$. See Section \ref{s_epsilon}
for further discussion.

\subsubsection{}

Following \eqref{eq:33}, we define
\begin{equation}
  \label{eq:51}
  \cOh_{\Fix,\vir} = \Spin \sTate{-\delta}\,,
\end{equation}
where
\begin{equation}
  \label{eq:52}
  \delta = \tfrac 12 \dim \Fix =
  \begin{cases}
    1\,, \quad &\chi_0^2 = 1 \,, \\
    0 \,, \quad  &\chi_0^2 \ne 1\,.
  \end{cases}
\end{equation}
By construction, we have
\[
  \left[ \cOh_{\Fix,\vir}  \right] = \bL(\bP^1) =
  \bL_{\tfrac12}(T^*\bP^1) \,,
\] 
which may be interpreted either as equality of classes in equivariant
K-theory of $\bX^\vee$, or as equality of distributions after pushing
forward to point modulo group. Our goal in this section is to lift
the relation between L-genera and constant terms to a relation between
$\widehat{H}(\QMN)$ and \eqref{eq:51}.

\subsubsection{}
Constant term is a function of $\cM\in \Pic(C)$ and in order to
introduce an $\cM$-dependent twist of $\Spin$ we need a homomorphism
   \begin{equation}
   \Pic(C)(\Bbbk) \owns \cM \mapsto \cO(\cM) \in 
   \Pic_{\Fr \times \bA\buu}(\bX^\vee)\,, \label{eq:25-2}
   \end{equation}
 such that the equivariant weights $\Fr \times \bA\buu$-weights of the
 restriction of $\cO(\cM)$ to the fixed points $\{0,\infty\}$ are
 $\chi(\cM)^{-1}$ and $\chi(\cM)$, respectively.

 In general, $\Gamma'$, or any other group, will act nontrivially on the
 restriction of the natural line bundles to the $\Gamma'$-fixed
 locus. In particular, consider the line bundle $\cO(1)$ with
 its natural linearization. This means that the group $\bA\buu =
 \{ \diag(a^{-1},a)\}$ acts with the weights $a^{\pm 1}$ over
 over the
 fixed points $\{0,\infty\}$,
 respectively. It follows from \eqref{eq:299} that an element
 $\gamma\in \Gamma$ will act by $\chi(\gamma)^{\mp1}$
 on the same fibers. Unless $\chi_0 =1$, this does not factor through
 $\Gamma/\Gamma'$.

Therefore, if we want $\Fr$ to act on $\cO(1)$, we need
 to pick its lift to $\Gamma/[\Gamma,\Gamma]$. We get such a lift if we fix a
 point $x\in C$ of degree $1$ or, more generally, a divisor $\cM$
 of degree $1$. We denote $\cO(\cM)$ the bundle $\cO(1)$ with the
 corresponding action of $\Fr$. This factors through $\Pic_1(C)(\Bbbk)$
 because $\chi_0$ is a character of $\Pic_0(C)(\Bbbk)$. By linearity,
 this assignment extends to \eqref{eq:25-2}. 

 \subsubsection{}
 Our main result in this section is the following

 \begin{Theorem}\label{t2}
  For any $\cM$ we have
     \begin{equation}
    \widehat{H}(\QMN) = \sigma_{*,\bfL} \left(\cOh_{\Fix,\vir} \otimes \cO(\cM)
    \right)
    \,,\label{eq:324}
  \end{equation}
  as a $\cA^{\ge 1}_\chi(C)$-module, unless 
  \begin{equation}
    \label{eq:63}
   (g,\cM,\chi_0^2) = (2,\cO_C, \textup{nontrivial}) \,. 
  \end{equation}
\end{Theorem}

\noindent
The proof of this theorem will occupy Section \ref{s_pt2}. The
exceptional case is analyzed in section \ref{s_exception} \,.
Among classical results to which Theorem \ref{t2} may be
directly comparable we mention \cites{Semiinf2, Semiinf1}. 




\section{Proof of Theorem 2}\label{s_pt2}

\subsection{Attracting manifolds and fixed loci} 

  \subsubsection{}
Consider the action of the group $\Aut(\cM) \cong \Gm$
  on the space $\QMN$. It is clear, for instance from the description in section
  \ref{s_concreteQMN}, that it has two
  basins of attraction as $z \in \Gm$ tends to $0 \in \overline{\Gm}$.
  Namely, 
  \begin{equation}
    \Attr(0) = \{ s_2 = 0 \} \,, \quad \Attr(\infty) = \{ s_2 \ne 0 \} \,.
    \label{eq:239}
\end{equation}
We denote the corresponding fix loci by
\[
\QMN^{\Gm} = \QM(0) \sqcup \QM(\infty) \,.
\]
On these fixed loci, we have $\cV = \cM \oplus \cO$
and $s$ is a quasisection of $\bP(\cM)$ or $\bP(\cO)$ respectively.
These are both $\bP^0$-bundles over $C$.

We observe that both $\QM(0)$ and  $\QM(\infty)$ are isomorphic
to the symmetric product  $\Sd = \bigsqcup \bS^d C$ or, equivalently,
the Hilbert scheme of points of $C$. Indeed, the isomorphism 
\begin{equation}
\Sd C \to  \QM(\bP(\cM)) = \{ \cL, s \in \bP(H^0(\cM\otimes \cL^{-1}))\}
  \label{eq:240}
\end{equation}
takes a divisor $D\in \Sd C$ to the quasimap
\[
\cL = \cM(-D) \xrightarrow{\quad} \cM \,.
\]

 \subsubsection{}

 The following table
 \begin{equation}
   \label{eq:305}
\begin{tabular}{ c| ccc} 
  & $\NN_s$ & $N_{\Attr/\textup{\textsf{fixed}}}$ & $\rk
                                                    N_{\Attr/\textup{\textsf{fixed}}}$
  \\ [0.5ex] 
\hline 
 $\QM(0)\phantom{\Big|}$ & $\cM^{-1}(2D)$ & $-\Hd(\cM)$ & $-\deg \cM +
  g-1$\\ 
 $\QM(\infty)$ & $\cM(2D)$ & $\Hd(\cM(D))-\Hd(\cM)$ & $\deg D$\\ 
\end{tabular}
 \end{equation}
collects the information about the normal bundle \eqref{eq:232}
to the fixed sections and about the normal bundle to the fixed
locus inside the attracting manifold. In particular,
%
\begin{alignat}{2}
  \notag 
  \Hd_c(\Attr(\infty)) &= \Hd(\QM(\infty)) &&\sTate{-2\deg D} \,, \\
  \Hd_c(\Attr(0)) &= \Hd(\QM(0)) &&\sTate{2\deg \cM
    +2-2g} \label{eq:244}
  \,.
 \end{alignat}

 \subsubsection{}
Consider the
 long exact sequence of a pair 
 \begin{equation}
   \label{eq:242}
   \xymatrix{
     & \Hd_c(\QMN) \ar[dr] \\
     \Hd_c(\Attr(\infty))  \ar[ur] && \Hd_c(\Attr(0))
     \ar@{.>}[ll]_{[1]} }
  \end{equation}
  for cohomology with compact supports. Since both
  cohomology groups in \eqref{eq:244} are pure, we
  conclude the following

  \begin{Lemma}
   The connecting dotted map in \eqref{eq:242} vanishes. 
  \end{Lemma}


\subsubsection{}

The fix loci have their own correspondences as in \eqref{eq:216},
which we denote by $\be^{(0)}$ and $\be^{(\infty)}$, respectively.
The main difference between quasisections of $\bP^k$-bundles
for different $k$ is  the rank of normal bundle $\cN$ as in
\eqref{eq:325}. This rank equals $k+1$ and hence $c_{k+1}(\cN)$ is
used 
in the definition of the correspondences. In particular, the
analog of \eqref{eq:224} for $\bP^0$-bundles reads
  \begin{equation}
 [\iota^* , \iota_{*} ] \circ \Delta_{C,*}  = c_1(\cN
 \otimes T C) - c_1(\cN)
    = \tau\,. \label{eq:224*2}
\end{equation}
 Thus
 \begin{equation}
   \label{eq:241}
  \left[ \be^{(0)}, \bff^{(0)} \right] = (\textstyle{\int_C}
  \otimes 1) \,  \Delta_C^*  \,, 
 \end{equation}
 meaning that these correspondences generate a Heisenberg--Clifford
 algebra $\Heis$ in place of $\fsl_2$.

 \subsubsection{}

 As before, a local system $\chi$ defines a local system $\Heis_{\chi}$, in which
 $\be$ and $\bff$ are twisted by $\chi$ and $\chi^{-1}$,
 respectively. Note that,  in accordance with the first column in
 \eqref{eq:305} and the Tate shift that
 is present in \eqref{eq:244}, we need the local system
 $\Heis_{\chi^2}$ on $\QM(0)$ and the local system $\Heis_{q\chi^2}$
 on $\QM(\infty)$ to make the triangle
 \eqref{eq:242} Galois-equivariant. 

 \subsubsection{}

 The following result is well-known and constitutes a $1$-dimensional analog
 of Nakajima's action on the Hilbert schemes of points of surfaces
 \cite{NakHilb}. 

 \begin{Proposition}\label{p_Verma_0}
$\Hd(\Sd C, \chi)$ is a Verma module for $\Hd(C,\Heis_{\chi})$. 
 \end{Proposition}

 \noindent
 For the trivial local system, this goes at least as far back as the
 work of Macdonald \cite{Macd}. For a nontrivial local, this is a
 consequence of representation theory of Clifford algebras and the
 following

 \begin{Lemma}
 If $\chi$ is nontrivial then $\Hd(\bS^d C, \chi)=0$ for $d>2g-2$. 
\end{Lemma}

\begin{proof}
  The map $S^d C \to \Pic_d(C)$ is projective bundle for $d > 2g-2$ and
  the pushforward of the constant sheaf under this map forms a trivial
  local system on the base. Since the local system $\chi$ is pulled
  back from the base, and the cohomology of an abelian variety with
  nontrivial coefficients vanishes, the claim follows. 
\end{proof}

\noindent 
As a corollary of the proof, we note that $\Hd(\bS^{2g-2} C, \chi)$
vanishes except for
\begin{equation}
  \label{eq:67}
   H^{2g-2}(\bS^{2g-2} C, \chi) = \chi(\cK_C) \otimes
   H^{2g-2}(\bP^{g-1}) \,, 
\end{equation}
where $\cK_C$ is the canonical bundle of $C$ and
$\bP^{g-1} \subset \bS^{2g-2} C$ is the canonical linear system. From
this and Proposition \ref{p_Verma_0} one may observe the
classical fact that
\begin{equation}
  \Lambda^{\textup{top}} \Hd(C,\chi) =
  \chi(\cK_C)\sTate{2-2g}\label{eq:68}
\end{equation}
for any character $\chi$, trivial or not.


 \subsection{Localization is localization}\label{two_local}

\subsubsection{}

Equivariant localization in topology 
relates  equivariant cohomology of a space, ordinary and
extraordinary, with the cohomology of the fixed loci.
In our case, this comparison is provided by exact
sequence \eqref{eq:242} and the isomorphisms
\eqref{eq:244}.

Using these isomorphisms, we may express the
$\fsl_2$-action in terms of the Heisenberg operators.
Our goal in this subsection is to identify this topological 
equivariant localization with localization in
representation theory as discussed in section
\ref{p_sl2_Cliff}.

  \subsubsection{}

  The map $\Hd_c(\QMN) \to \Hd_c(\Attr(0))$ is pullback under
  the embedding and thus commutes with the pullback $\iota^*$.
  For the pushforward, we need to look at the normal bundle $\cN$,
  which fits into the sequence
  \begin{equation}
    \label{eq:326}
  0 \to \underbrace{\cM \otimes \cL^{-1} \otimes
\cO_{\QM}(1)}_{\cN^{(0)}}\to  \cN \to  \underbrace{ \cL^{-1} \otimes
\cO_{\QM}(1) }_{\cN^{(\infty)}}\to 0 \,.
\end{equation}
The map $\iota^{(0)}_*$ lacks the quotient term in \eqref{eq:326},
whence 
  the following

  \begin{Proposition}
   The diagrams
  \begin{equation}
    \label{eq:243}
    \xymatrix{
      \Hd_c(C\times \QMN) \ar[r]  &\Hd_c(C\times \Attr(0)) \\
      \Hd_c(C\times \QMN) \ar[r]  \ar[u]^{\iota^*} & \Hd_c(C\times \Attr(0))
      \ar[u]_{\iota^{(0),*}}
    } \qquad
    \xymatrix{
      \Hd_c(C\times \QMN) \ar[r] \ar[d]_{\iota_*}  &\Hd_c(C\times \Attr(0))
      \ar[d]^{\square_*}\\
      \Hd_c(C\times \QMN) \ar[r]  & \Hd_c(C\times \Attr(0))
    } 
  \end{equation}
  where
  \[
    \square_* = c_1(\cN^{(\infty)}) \, \iota^{(0)}_*= \iota^{(0)}_*
    c_1(\cN^{(\infty)}\otimes TC)\,, 
    \] 
  commute. 
  \end{Proposition}

  \noindent
  Note these formulas are consistent with interpreting the
  virtual bundle 
  \[
  N_{\QMN/\Attr(0)} = \Hd(C,\cL^{-1}) \otimes \cO_{\QM}(1) \,. 
\]
as the virtual normal bundle to $\Attr(0)$ in $\QMN$.  The
correspondence $\iota^{(0)}$ changes it by
\[
  \iota^{(0),*} N_{\QMN/\Attr(0)}  - N_{\QMN/\Attr(0)} =
  \cN^{(\infty)}\otimes TC \,. 
\]

  \subsubsection{}

The normal bundle 
\begin{equation}
  \label{eq:245}
  N_{\QMN/\QM(\infty)} = \left( \Hd(\cM\otimes \cL^{-1}) -
  \Hd(\cM)\right) \otimes \cO_{\QM}(1) 
\end{equation}
is a vector bundle of  rank $-\deg \cL$ over
$\QM(\infty)$.
The map $ \Hd_c(\Attr(\infty)) \to \Hd_c(\QMN)$ commutes
 with $\iota_*$, since both maps are pushforwards. For pullbacks, we
 need to look at the effect of $\iota^{\infty}$ on the normal bundle
 \eqref{eq:245}. We deduce

 \begin{Proposition}
 The diagrams
 \begin{equation}
   \label{eq:246}
   \xymatrix{
      \Hd_c(C\times\Attr(\infty))  \ar[r] & \Hd_c(C\times\QMN)  \\
     \Hd_c(C\times\Attr(\infty))
     \ar[u]^{\square^*}  \ar[r]  & \Hd_c(C\times\QMN) \ar[u]^{\iota^*} 
    } \quad
    \xymatrix{
     \Hd_c(C\times\Attr(\infty))
      \ar[d]_{\iota^{(\infty)}_*} \ar[r] &  \Hd_c(C\times\QMN) \ar[d]^{\iota_*}  \\
     \Hd_c(C\times\Attr(\infty)) \ar[r]   & \Hd_c(C\times\QMN) 
    } 
  \end{equation}
  where
  \[
    \square^* = c_1(\cN^{(0)}\otimes TC) \, \iota^{(\infty),*} =\iota^{(\infty),*}
   c_1(\cN^{(0)})\,, 
    \] 
  commute. 
 \end{Proposition}

 \noindent 
 Note that the nontrivial squares in \eqref{eq:243} and \eqref{eq:246}
 also follow from the trivial ones and the relations \eqref{eq:229}
 below.

 \subsubsection{}
 One has the following analog of the Casimir element equation
 \eqref{eq:36}
 \begin{equation}
   \label{eq:55}
   \bff^{(0)}_{13} \, \be^{(0)}_{23} \circ \Delta_{C,*} = \bh^{(0)}
   \eqdef  (\textstyle{\int_C} \otimes 1) \circ c_1(\cN^{(0)} \otimes
   TC ) \,. 
 \end{equation}
 Indeed, \eqref{eq:55} is just the composition of pushforward and
 pullback by the same map. We can use \eqref{eq:56} to  expand out \eqref{eq:55} along the
 lines of \eqref{eq:306}.

 \begin{Corollary}
  We have
 \begin{align}
    \be^{(0)}\of{\bpt} \, \bff^{(0)}\of{\bpt} &=\eta \,, \notag\\
    \be^{(0)}\of{\bpt} \, \bff^{(0)}\of{\alpha} + \be^{(0)}\of{\alpha} \, \bff^{(0)}\of{\bpt}  &=\alpha 
       \label{eq:306-0}                                                  
 \end{align}
 where $\alpha \in H^1(C)=H^1(\Pic(C))$ acts on $\QMN$ by
 multiplication. 
\end{Corollary}
 
\noindent
For brevity, let us denote by $FE$ the following composition 
 \begin{equation}
   \label{eq:53}
   FE =  \bff_{13} \, \be_{23}  \circ \Delta_{C,*} = c_2(\cN \otimes TC) 
 \end{equation}
of pushforward $\be$ and the pullback $\bff$. Let
$FE^{(0)}$ denote the analogous expression appearing in \eqref{eq:55}. From \eqref{eq:326},
we conclude 
 \begin{alignat}{2}
   \label{eq:229}
   FE & = FE^{(0)} \,
   c_1(\cN^{(\infty)}\otimes TC) &&=
   c_1(\cN^{(\infty)}\otimes TC) \, FE^{(0)}
  \\
                  &= FE^{(\infty)} \, c_1(\cN^{(0)}\otimes TC)&&
                  =  c_1(\cN^{(0)}\otimes TC) \, FE^{(\infty)}   \,,   \notag 
\end{alignat}
which provides a different computations of the nontrivial squares in
\eqref{eq:243} and \eqref{eq:246}. 

\subsubsection{}
Writing out the equations \eqref{eq:243} using 
\eqref{eq:306-0}, we obtain
\begin{align}
  \bff\of{\bpt} &= \bff^{(0)}\of{\bpt} \notag = Y \\
  \be\of{\bpt} &= \eta \, \be^{(0)}\of{\bpt} = X^2 Y \notag \\
  \bff\of{\alpha} &= \bff^{(0)}\of{\alpha} = dY \notag \\
  \be\of{\alpha } &= \eta \, \be^{(0)}\of{\alpha} + \alpha \,
                    \be^{(0)}\of{\bpt} = 2 XY dX + X^2 dY \,, 
                   \label{eq:57}
\end{align}
where we use the shorthand
\[
  X = \be^{(0)}\of{\bpt}\,, \quad Y=\bff^{(0)}\of{\bpt}\,, \quad
  dX = \be^{(0)}\of{\alpha}\,, \quad dY=\bff^{(0)}\of{\alpha} \,.
\]
These formulas have a geometric interpretation which may be
illustrated using Figure \eqref{eq:301}. Let $X$ and $Y$ be the
coordinates around $0 \in \bX^\vee$ with weights $a^{-2}$ and $q a^2$,
respectively. Clearly, the first two lines in \eqref{eq:57} describe
the map $\sigma$ in the coordinates $X$ and $Y$, while the second two
lines describe the action of $\sigma^*$ on $\Omega^1$.
Let
\[
  \textup{Chart}_0= \bX^{\vee}\setminus T^*_\infty \,, \quad
  \textup{Chart}_\infty = \bX^{\vee}\setminus T^*_0 \,,
\]
be two affine charts covering $\bX^\vee$, where $T^*_0$ and $T^*_\infty$ denote the cotangent fibers over $0,
\infty \in \bP^1$. We rephrase \eqref{eq:57} as follows

\begin{Proposition}\label{p_loc_loc} 
  The equivariant localization homomorphism
  \[
    \cA^{\ge 1}(C) \to \cA^{(0),\ge 1}(C) \oplus \cA^{(\infty),\ge
      1}(C)
  \]
  given in \eqref{eq:243} and \eqref{eq:246} coincide with the
  natural maps of the Clifford algebras
  \[
    \Cliff_{\bX^\vee_0} \xrightarrow{\quad \textup{restriction} \, \circ \, 
      \sigma^*\quad }  \Cliff_{\textup{Chart}_0} \oplus
    \Cliff_{\textup{Chart}_\infty}\,. 
  \]
 \end{Proposition}

 \noindent
 This forms a direct link between equivariant localization and
 localization in representation theory.

 \subsection{The stable range}

 \subsubsection{}

 Our next goal is to prove Theorem \ref{t2} outside the vertex of the
 cone, that is, one the set where $\sigma$ is an isomorphism.
 This set is covered by two affine charts 
 \[
 \bX_0^\vee \setminus \{0\} = \{\textup{$\be\of\bpt$ is invertible} \}
   \cup
   \{\textup{$\bff\of\bpt$ is invertible} \} \,. 
   \]
   Since the degree of a quasimap is always bounded below, the lowering
   operators $\bff$ are locally nilpotent. Thus both sides of equality
   in Theorem \ref{t2} restrict to zero on the chart where
   $\bff\of\bpt$ is invertible.

   Since $\be\of\bpt$ is a raising operator, inverting it involves
   understanding of the $\fsl_2$-action in the stable range $\deg \cL
   \ll 0$. We will see
   in Section \ref{s_p_compl} that 
 all cohomology vanish in this range if $\chi_0^2\ne 1$. Therefore, our
 focus in this section is on cohomology with trivial local system
 coefficients.

 \subsubsection{}
If $\be\of\bpt$ is invertible then the Casimir relation \eqref{eq:36}
determine the action of $\bff$ from the action of $\be$ and $\bh$,
that is, from the action of the Borel subalgebra $\mathfrak{b} \subset
\fsl_2$. 
We recall from Proposition \ref{p_h} that the operators $\bh\of\bpt$
and $\bh\of\alpha$ act by a multiple of $\eta$ and $\alpha$,
respectively, where $\alpha\in H^1(C) = H^1(\Pic C)$ acts by cup 
product with the pullback of $\alpha$ to $\QMN$. The commutator 
\begin{equation}
  \label{eq:66*2}
\left[ \alpha_1, \be\of{\alpha_2} \right] = \be\of{\alpha_1 \cup \alpha_2}
\end{equation}
is the only nontrivial commutator in $\mathfrak{b}\otimes H^{\ge
  1}(C)$. 

\subsubsection{}

Let $\cF$ be any locally free Clifford module on $\bX_0^\vee$. It is
clear from the discussion in Section \ref{s_local coh} that 
\begin{equation}
  H^1_{\bfL}(\bX_0^\vee \setminus \{0\},\cF)
  \cong \left(\Q_\ell[\eta^{\pm 1}] \big/ \Q_\ell[\eta] \right) \otimes
  \underbrace{\cF \Big|_{\bfL_\infty \setminus \{0\}}}_{\Ker
    \eta} \,, \label{eq:67*2}
\end{equation}
as $\cA^{\ge 1}$-modules, while $H^i_{\bfL} =0$ for $i\ne 1$. One may
call vectors annihilated
by $\eta$ primitive. We say that a module of the form \eqref{eq:67*2} is
freely $\eta$-cogenerated by its primitive elements.

\begin{Proposition}\label{p_outside}
The $\cA^{\ge 1}$-module $\widehat{H}(\QMN)$ with $\be\of\bpt$ inverted
is freely $\eta$-cogenerated by its primitive elements. The primitive
elements are freely generated under the action of the operators
$\be\of{\bpt}^{\pm 1}$ and $\be\of\alpha$,
$\alpha \in H^{1}(C)$,
by any $1$-dimensional subspace of the form 
$H^{\textup{top}}_c(\QMN_d)$, where $d=\deg \cL \ll 0$. 
\end{Proposition}

\noindent
Evidently, $H^{\textup{top}}_c(\QMN_d)$ is annihilated by cup products
by $\alpha \in H^1(C)$. Thus, the above proposition and \eqref{eq:66*2}
completely determines the structure of this module. The shifts and
twists in definition of both sides of \eqref{eq:324} are 
arranged so that primitive elements equal the 
restriction of $\Spin \otimes \cO(\cM) \Tate{\tfrac{\deg
    \cM+\delta}{2}}$ to $\bfL_\infty \setminus \{0\}$. Therefore,
we have the following

\begin{Corollary}
Theorem \ref{t2} is true outside the vertex of the
 cone. 
\end{Corollary}

\noindent
The proof of Proposition \ref{p_outside} will occupy the rest of this
subsection.

 \subsubsection{}
 Let $\QM_d \subset \QM$ denote the component with $\deg \cL = d$, 
 and similarly for $\QMN_d$. We recall from Section \ref{s_bundle}
 that $\QM_d \to \Pic_d(C)$ is a projective bundle for $d\ll 0$. Similarly, $\QMN_d$ is
 the quotient of a projective bundle over $\Pic_d(C) \times
 H^1(\cM)$ by the action of the unipotent group $H^0(\cM)$.
 By the projective bundle theorem, we have 
 \begin{equation}
 H_c(\QMN_d)
 \cong
 \underbrace{\Ld \, H^1(C)}_{\textup{coefficients}} \, \big[\eta\big]
 \Big/ \bC_{\cO\oplus \cM} (\eta) 
 \,, \quad d \ll 0 \,, \label{eq:54}
\end{equation}
as a module over $\eta$ and $H^1(C)$, where
\[
\bC_{\cV} (\eta) = \textup{Chern
   polynomial}(H^0(C,\cV\otimes\cL^{-1}),\eta) \,. 
\]
While we can argue with just abstract properties of this Chern
polynomial, it is easy to describe it concretely.

\subsubsection{}

Since $H^0(C,\cV\otimes \cL^{-1})$ is obtained by pushforward
along $C$, which is $1$-dimensional, we conclude that
\[
\textup{ch}_i \left(H^0(\cV\otimes \cL^{-1})\right) =0 \,, \quad i>1 \,. 
\]
On the other hand, by the definition and symmetry of the theta divisor
$\Theta$, we have
\begin{equation}
  \label{eq:59}
  c_1 \left(H^0(\cV\otimes\cL^{-1})\right) = \rk(\cV) \, \Theta\,,
  \quad \Theta = \sum_{i=1}^g \gamma_i
  \gamma^\vee_i\,. 
\end{equation}
Therefore,
  \begin{equation}
    \label{eq:61}
    \bC_{\cV} (\eta)= \eta^{\rk H^0(\cV \otimes \cL^{-1})} 
 e^{-\textup{ch}_1/\eta} = \eta^{-rd + \const}
 \prod_{i=1}^g (\eta - r \gamma_i
  \gamma^\vee_i) \,, \quad r = \rk(\cV) \,, 
\end{equation}
where the constant equals $\deg \cV + r-(r+1)g$. Importantly, this
stabilizes in that it depends on $d$ only through the $\eta^{-rd}$
factor.

\subsubsection{}
It is easy to describe the action of the operators $\be$ and $\bff$ in
the stable range. Let $P=P(\eta, \alpha_1,\alpha_2, \dots)$ be an
element of the ring in the right-hand side of \eqref{eq:54}.

\begin{Lemma}
  We have
  \begin{alignat}{2}
    \label{eq:68-3}
    \bff\of\bpt \, P &= -P \,,& \qquad \bff\of\alpha\, P &= -
    \frac{\partial}{\partial \alpha^{\vee}} P \,, \\
    \be\of\bpt \, P &= \eta^2 P \,,& \qquad \be\of\alpha\, P &= 
    \eta^2 \frac{\partial}{\partial \alpha^{\vee}} P - 2 \eta \alpha
    P\,. \label{eq:68-2}
  \end{alignat}
\end{Lemma}
\begin{proof}
  The pullback $\iota^*$ is an algebra homomorphism which acts by  
\begin{alignat}{2}
  \label{eq:314}
  \iota^*(1\otimes \alpha) &= \alpha\otimes 1 + 1 \otimes \alpha\,, &
  \quad &\alpha \in H^1(C) \,, \\
  \iota^*(1\otimes\eta) & = \bpt \otimes 1 + 1 \otimes \eta \,,  \label{eq:314-2}
\end{alignat}
where $\bpt\in H^2(C)$ is the generator. The relation \eqref{eq:314}
was already discussed in the proof of Proposition \ref{p_h} . To see \eqref{eq:314-2},
note that $\eta$ is represented by the divisor where a linear function of
$s(x)$ vanishes for some fixed $x\in C$.

The equations \eqref{eq:68-3} follows from the definition of $\bff$ and
the above description of $\iota^*$. Formulas \eqref{eq:68-2}
follow from the Casimir relations \eqref{eq:36}. 
\end{proof}

\subsubsection{}

The line $H^{\textup{top}}_c(\QMN_d)$ is spanned by $\eta^{\max}
\Lambda^{\textup{top}} H^1(C)$, where the maximal power of $\eta$ is
one less than the degree of $\bC(\eta)$. It is clear that the action
of $\be\of\bpt^{\pm1}$ takes all these elements to each other.

By
looking at the lowest degree term in the odd variables, it is clear
that the action $\be\of\alpha$ generates a $2^{2g}$-dimensional
space of primitive vectors from $H^{\textup{top}}_c(\QMN_d)$. 
These have to exhaust all primitive vectors and this completes
the proof of Proposition \ref{p_outside}.

In fact, it is not difficult 
to see directly that the product of all $\be\of\alpha$, applied to $H^{\textup{top}}_c$ reproduces
precisely the relation $\bC(\eta)$.

 \subsection{Completion of the proof}\label{s_p_compl} 

 \subsubsection{}

 By Proposition \ref{p_Verma_0}, the exact sequence  \eqref{eq:242}
 translates into the exact sequence 
\begin{equation}
  \label{eq:310}
  0 \to \textup{Verma}\to
  \widehat{H}(\QMN) \to \textup{dual Verma} \to
  0 \,, 
\end{equation}
where we call a $\cA_\chi(C)$-module a Verma module if it is free as a module over
the commuting operators $\be$, while a dual Verma
module is dual to a free module as a module over the commuting
operators $\bff$. Note this definition also involves operators coming
from $H^0(C)$.

Our plan is to compare this sequence with the exact sequence
of the pair
\begin{equation}
  \label{eq:71}
  0 \to \sigma_{\bfL_\infty,*} \to \sigma_{\bfL,*} \to
  \sigma_{\bfL_0,*} \to 0 \,, 
\end{equation}
as in \eqref{eq:329}. We will first compare the submodules and
quotient modules, and then compare the extensions.

\subsubsection{}
It is clear from Proposition \ref{p_loc_loc} that the terms of the
filtration \eqref{eq:71} for $\sigma_{\bfL,*} \cOh_{\Fix,\vir} \otimes \cO(M)$
are similarly
a Verma
module and a dual Verma module. Therefore, it suffices to compare
the weights of the generators. This is done in the following 

\begin{Lemma}\label{l_w_gen}
  For both $\widehat{H}(\QMN)$ and $\sigma_{\bfL,*} \cOh_{\Fix,\vir} \otimes
  \cO(M)$, we have
  \begin{alignat}{2}
    \notag 
   \textup{generator of the submodule} & = \chi(\cM\otimes \cK_C^{-1})
                                      &&\sTate{g-1} \,, \\ 
    \textup{generator of the quotient} & = \chi(\cM^{-1}\otimes \cK_C^{-1})
                                     && \sTate{1-g}   \,.  \label{eq:72}
  \end{alignat}
\end{Lemma}

\begin{proof}
  For $\widehat{H}(\QMN)$, this is an immediate consequence of
  the definition \eqref{eq:222} and table \eqref{eq:305}. In the
  computations with the spinor bundle, we will need the number
  $\delta= \delta_{\{\chi_0^2=1\}} \in \{0,1\}$, where $
  \delta_{\{\chi_0^2=1\}}$ denotes the indicator function. From
  \eqref{eq:68}
  we deduce 
  \begin{equation}
    \label{eq:60}
    \left(\Lambda^{\textup{top}} H^1(C, \chi^2)\right)^{1/2} = \chi(\cK_C)
    \chi^{2\delta} \sTate{1-g-\delta} \,. 
  \end{equation}
  Consider the bundle $\Spin$. The generator of the submodule
  and the quotient modules come from the elements $1\in \Cliff$ and
  the product of all $\bff\of{\alpha_i}$,
  respectively. With the square root twists, these have the weight
  $ \left(\Lambda^{\textup{top}} H^1(C, \chi^2)\right)^{-1/2}$
  and $\left(\Lambda^{\textup{top}} H^1(C, \chi^{-2})\right)^{1/2}$,
  respectively. These weights form the first row in the following
  table 
 \begin{equation}
   \label{eq:305-2}
\begin{tabular}{ c| cc} 
 \textup{contributions from}  & \textup{quotient} & \textup{sub} 
  \\ [0.5ex] 
  \hline  
  $\Spin $  & $\chi(\cK_C)^{-1}
    \chi^{-2\delta} \sTate{1-g-\delta}$& $\chi(\cK_C)^{-1}
                                         \chi^{-2\delta} \sTate{\delta+g-1}$\\
  $\cOh_{\Fix,\vir} $  & $\sTate{-\delta}$& $\sTate{-\delta}$\\
   $\cO(\cM)$  & $\chi(\cM)^{-1}$& $\chi(\cM)$\\
  \textup{local cohomology}  & $\chi^{2\delta}\sTate{2\delta}$& $\chi^{2\delta}$\\
\end{tabular}
 \end{equation}
 where the shifts in local cohomology appear as in the computation done
 in section \ref{s_local coh}. Collecting terms in the above table
 concludes
 the proof. 
\end{proof}

\subsubsection{}
To finish the proof, we need to identify the classes of the
extension. If $\chi^2_0 = 1$, then the submodule is free as
a module over $\be\of\bpt$. This means the extension class
is uniquely determined by its localization with respect to
$\be\of\bpt$. By Proposition \ref{p_outside}, the extension
classes coinside after such localization, completing the
proof in the $\chi^2_0 = 1$ case.

\subsubsection{}
Consider the case $\chi^2_0 \ne 1$, in which case the
statement of the theorem is equivalent to the splitting of the
module $\widehat{H}(\QMN, \chi)$. The submodule
is a free module over $\be\of{\alpha_i}$, with the operators
$\bff\of{\alpha_j}$ acting by zero. For the quotient module,
the situation is reversed. Thus $\bff\of{\alpha_j} \, \be\of{\alpha_i} =
0$ for any $\alpha_i$ and $\alpha_j$, and the module splits
if and only if the product in the other order also vanishes. This is
equivalent to the vanishing of the commutator $\bh\of\bpt$.

\subsubsection{}
Both the sub and the quotient module are nonzero for a single
cohomological degree for given degree of the quasimap.
From the considerations of degrees, we see that 
$\rk \bh\of\bpt\le 1$. Since $\bh\of\bpt$ is central, we conclude
\[
  \Image \bff\of{\alpha_i} \subset \Ker \bh \of \bpt \,,
  \quad \Image \bh \of \bpt \subset \Ker \be\of{\alpha_i} \,, 
\] 
for any $\alpha_i$. 
Therefore, if nonzero, $\bh\of\bpt$ has to map the generator of the
quotient to the cogenerator of the submodule. Comparing the
degrees, we see that this is only possible if
\[
g=2 \,, \quad \deg \cM = 0 \,.
\] 
This case will be analyzed explicitly below and, indeed,
$\bh\of\bpt\ne 0$ when $\cM=\cO_C$.

\subsection{The exceptional case}\label{s_exception}

\subsubsection{}
Consider the case $g=2$, $\cM=\cO_C$. In this case, the generator of
the quotient module and the cogenerator of the submodule both
correspond to the class of $H^2(\bS^2 C,\chi)$ represented by
the canonical linear system $\bP^1 \subset \bS^2 C$. This cycle
lies in the fiber
\[
  \xymatrix{
 \QMN_{\cK^{-1}} \ar@{^{(}->}[r] \ar[d] & \QMN_{-2} \ar[d] \\ 
 \cK^{-1}_C \ar@{|->}[r] & \Pic_{-2}(C) \,, 
 } 
\]
over the point $\cL = \cK^{-1}_C $. The moduli space  $\QMN_{\cK^{-1}}$ is formed by all possible extensions
\[
  0 \to \cO \to \cV \to \cO \to 0
\]
together with a section
\[
  s \in H^0(\cK \otimes \cV)
\]
taken modulo the action of $\Gm \times H^0(\cO) \subset \Aut(\cV)$.
The operator $\bh\of\bpt$ acts by a multiple of $\eta = c_1(\cO(1))$.
It descends from the same action on the space $\widetilde{\QMN}_{\cK^{-1}}$, where we divide
by the action of $\Gm$ only.

\subsubsection{}
Consider the exact sequence
\[
  0 \to H^0(\cK)  \to H^0(\cK \otimes \cV) \to  H^0(\cK)
  \xrightarrow{\quad \beta_\cV \quad} H^1(\cK)
\]
where the map $\beta$ is the pairing with the class of
the extension
\[
\beta_\cV \in \Ext^1(\cO,\cO) = H^0(\cK)^*  \,. 
\]
Possible bundles $\cV$ are precisely parameterized by
their classes $\beta_\cV$. Denoting $K=H^0(\cK)$, we obtain
the following description of $\widetilde{\QMN}_{\cK^{-1}}$
as a subset of $\bP(K^{\oplus
   2}) \times K^*$ 
\begin{equation}
  \xymatrix{
 \bP(K\oplus \beta^\perp) \ar@{^{(}->}[r] \ar[d]
 &\widetilde{\QMN}_{\cK^{-1}} \ar[d] \\ 
 \beta \ar@{|->}[r] & K^* \,. } \label{eq:62}
\end{equation}
The action of $\Gm = \Aut \cM$ corresponds to scaling the first copy
of $K$ in $K^{\oplus 2}$. 

\subsubsection{}
In terms of the diagram \eqref{eq:62}, we generator of the quotient
module and the cogenerator of the submodule are dual to
the fundamental cycle in $H_8$ and
the cycle $\left[\bP(K)\right]\in H_6$, respectively. Similarly,
$\eta$ is dual to $\left[\bP(\textup{line} \oplus \beta^\perp) \right] \in
H_6$. One can move the line off the origin and then send it to
infinity by the action of $\Gm = \Aut \cM$. Thus
\[
\left[\bP(K)\right] = \left[\bP(\textup{line} \oplus \beta^\perp)
\right]
\]
and so $\bh\of\bpt\ne 0$ in this instance.

\subsubsection{}

If $\cM\not\cong \cO$ then the classes corresponding to $H^2(\bS^2 C,\chi)$
lie over different fibers of the projection to $\Pic_{-2}(C)$, thus
cannot possible be related via multiplication by $\eta$. This
concludes
the proof of Theorem \ref{t2}.

\section{Appendix.\ Rank 2 bundles over $\bP^1$}
\label{s_app} 

\subsection{$\Bun_\bG$ and $\Gr_\bG$}

\subsubsection{}

This appendix is elementary and written for readers with no background in
automorphic forms. Our main objective
here is to explain by means of the simplest example how spectral
decompositions are encoded by certain additive decompositions of
distributions. We also explain what these distributions have to do
with constant terms of Eisenstein series. Finally, we point out
how this elementary example fits into the general story told in our
paper.

\subsubsection{} 

Let $C=\bP^1$ be the projective line over a finite field $\Bbbk$ with
$q$ elements.
More generally, one may replace $\bP^1$ by a smooth
projective curve over 
$\Bbbk$ which remains irreducible over the algebraic
closure $\overline{\Bbbk}$, and that would be setting of
Section \ref{s_curves}. We denote by $\F =
\Bbbk(C)$ the field of rational functions on $C$. For
$C=\bP^1$, this simply means $\F= \Bbbk(t)$, where
$t$ is an indeterminate.

A closed point $x\in \bP^1$ other than $\infty \in \bP^1$ corresponds
to an irreducible monic polynomials in $t$. The degree $\deg_x$ is
the degree of this polynomial. We denote by $\Bbbk_x$ the
corresponding finite extension of $\Bbbk$. Consider
the completion $\F_x$ of $\F$ at $x$ and its subring 
\begin{equation}
  \label{eq:172}
 \Bbbk_x(\!(t)\!) \cong  \F_x \supset  \bO_x  \cong \Bbbk_x[\![t]\!] 
\end{equation}
formed
by functions regular at $x$. Here $\Bbbk_x(\!(t)\!)$ and $\Bbbk_x[\![t]\!]$
denote formal Laurent and Taylor series, respectively. 
One defines integral adeles and all adeles of $\F$,
respectively, by 
\begin{equation}
\bO = \prod_{x\in C} \bO_x \subset  {\prod_{x\in C}}' \, \F_x = \A \label{eq:152}
\end{equation}
where the prime in \eqref{eq:152} means that any adele is integral for
all but finitely many $x$. 

\subsubsection{} 

A vector bundle $\cV$ over $C$ may be trivialized in the formal
neighborhood of each closed point $x$ and also at the generic point.
The corresponding transition function identifies isomorphism classes
of vector bundles with 
double cosets 
\begin{equation}
\Bun_\bG = \bG(\F) \backslash \bG(\A) /\bG(\bO) \,, \label{eq:2}
\end{equation}
where $\bG=GL(n)$, $n= \rk \cV$.

For any reductive $\bG$, the
double cosets \eqref{eq:2} describe the $\Bbbk$-points of the stack $\cBun_\bG$ of
$\bG$-bundles on $C$. For
\[
\bG=  SL(n), PGL(n)
\]
this is the same as vector bundles with trivial
determinant, and vector bundles up to tensoring with a line bundle,
respectively. The latter can be also identified with
$\bP^{n-1}$-bundles $\cP$ over the curve $C$.

\subsubsection{}

A very special property of $C=\bP^1$ is that every vector bundle is a
direct sum of line bundles. As a result, the
structure of $\Bun_{\bG}$, and the corresponding theory of automorphic
forms, becomes almost trivial. Concretely,
\begin{equation}
\Bun_{SL(2)} = \left\{ \cO(m) \oplus \cO(-m) \right\}_{m\ge 0}
\hookrightarrow \Bun_{PGL(2)} = \left\{ \cP_k 
\right\}_{k\ge 0} \,, \quad \cP_k = \bP\left(\cO(k) \oplus \cO\right)
\,, \label{eq:153}
\end{equation}
where the map takes a rank 2 bundle to associated
$\bP^1$-bundle and its image consist of $\cP_k$ with $k$ even. 

For $C=\bP^1$, there are no cusp forms, no Galois
representations, Eisenstein series are equal to their constant term,
etc. Luckily, the structures that we want to illustrate survive even
in this simplest case and can give an idea about how things work in
general.

\subsubsection{}

Let $x \in C$ be a point of degree 1.  
The bundles in \eqref{eq:153} are trivial in the complement of $x$,
which means that
\begin{align}
  \Bun_\bG &= \bG(\cO_{C \setminus x}) \backslash \bG(\F_x)
             /\bG(\bO_x) \notag \\
 &= \bG(\Bbbk[t^{-1}]) \backslash \bG(\fst)
   /\bG(\fTst) \label{eq:154}\,,\\ 
   & = \bG(\Bbbk[t^{-1}]) \backslash \Gr_\bG \notag 
\end{align}
where
%
\begin{align}
  \label{eq:155}
  \Gr_\bG = \bG(\fst)
             /\bG(\fTst) \,, 
\end{align}
are the $\Bbbk$-points of the affine Grassmannian.
Among other
things, $\Gr_\bG$ describes $\bG$-bundles over $\bP^1$ trivialized
in the complement of a $\Bbbk$-point $x$.

\subsubsection{}

For $\bG = GL(n)$, we have
\[
\Gr_{GL(n)} = \{ \textup{$\fTst$-lattices $L\subset
  \fst^{\oplus n}$}\}\,. 
\]
The lattice $L(\cV)$ corresponding to a  vector bundle $\cV$ is
described as follows 
\begin{equation}
  \label{eq:157}
  \xymatrix{
    H^0(\bD_x,\cV) \, \ar@{^{(}->}[r] \ar@{=}[d]&
    H^0(\bD^\times_x,\cV) \ar@{=}[d] \\
    L(\cV) \, \ar@{^{(}->}[r] & \fst^{\oplus n} 
    } \,, 
  \end{equation}
  where
  \[
  \bD_x = \Spec \Bbbk[\![t]\!]\,, \quad 
\bD^\times_x = \Spec  \Bbbk(\!(t)\!) \,, 
\] 
are the formal neighborhood of $x$ and 
the punctured formal neighborhood of $x$.
For $\bG = SL(n)$ and $\bG =  PGL(n)$ one considers lattices $L$
such that $\Lambda^n L = \fTst$ and lattices up to scale,
respectively.

\subsubsection{}

For $\bG=PGL(2)$, the points of \eqref{eq:155} correspond to vertices
of a regular $(q+1)$-valent tree, see \cite{Serre}. This is the rank 1 case
of the general Bruhat-Tits building. Edges $L\textup{---}L'$ of this tree correspond to
pairs of lattices such that
\[
L \supsetneq L' \supsetneq t L \,. 
\]
Note this is a symmetric relation. 
For $q=2$, the tree is plotted in 
\eqref{eq:156}. 
\begin{gather}
  \label{eq:156}
  \raisebox{-100pt}{\includegraphics[height=200pt]{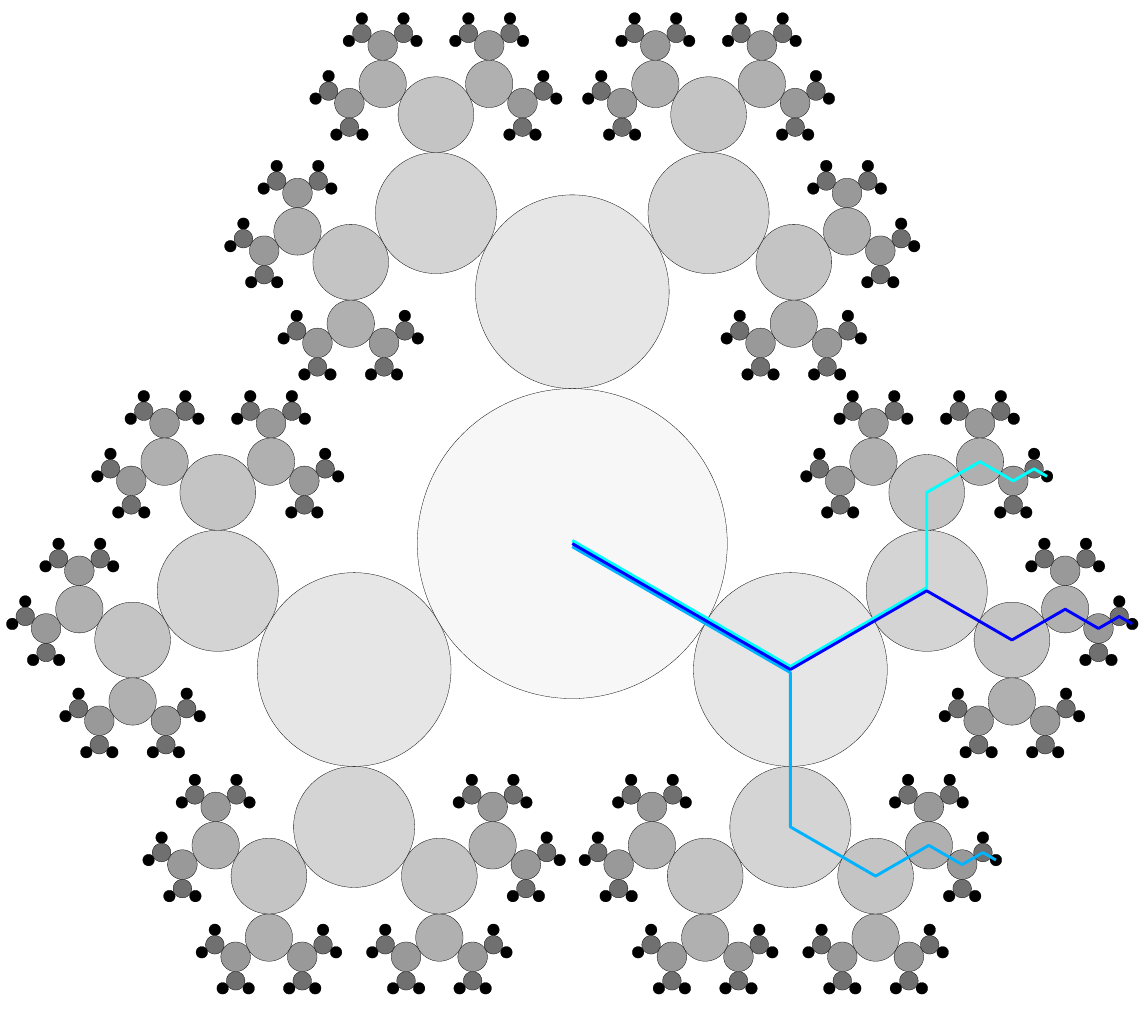}} \\
  \textup{\sl Orbits of $\bG(\fTst)$ on $\Gr_\bG$ for $q=2$} \notag 
\end{gather}
In \eqref{eq:156}, vertices of the tree are represented by circles and their color
reflects the distance from the root vertex ($=$ the trivial bundle).

The stabilizer of the root vertex in
$\bG(\fst)$ is $\bG(\fTst)$,
which acts transitively on the sphere of each radius around the
root. In other words, any path without backtracking
from the root to infinity gives a fundamental domain for $\bG(\fTst)$.
Three such paths are plotted in \eqref{eq:156}. The tree in
\eqref{eq:156} is the nonarchimedian analog of the unit disk,
equivalently, 
upper half-plane, and the action of $\bG(\fTst)$ should be compared
to the action of the maximal compact subgroup $SO(2,\R) \subset
SL(2,\R)$.

The subset
\[
\Gr_{SL(2)} \subset \Gr_{PGL(2)}
\]
is formed by vertices at even distance from the root.

\subsection{Hecke operators}

\subsubsection{}
The spherical Hecke algebra
\begin{align}
 \SH_x = \,\, & \textup{compactly supported functions} \notag \\
   &\textup{on $\bG(\fTst) \backslash \bG(\fst)/  \label{eq:159}
\bG(\fTst)$} 
\end{align}
acts by right convolution on function on $\Gr_\bG$. It is
clear from the description of orbits in \eqref{eq:156} that $\SH_x$
for $\bG = PGL(2)$ 
is an algebra of polynomials in one generator $\Delta$ which may be
chosen as the  graph Laplacian
\begin{equation}
\left[\Delta f \right] (L) = \sum_{L\textup{---}L'}
f(L') \label{eq:158} \,.
\end{equation}
The spherical Hecke algebra for $\bG=SL(2)$ is spanned by even powers of
$\Delta$.

\subsubsection{}
In terms of $\bP^1$-bundles $\cP \to C$, the action of $\Delta$ is
given by modifications as in \eqref{eq:160} in Section \ref{s_modif}. 
%
%
A section $s$ of $\cP$ in
\eqref{eq:160}
determines a unique section $s'$ of $\cP'$. Self-intersections of
$s$ and $s'$ are related by 
\begin{equation}
  \label{eq:161-2}
  (s',s') - (s,s) =
  \begin{cases}
    -1 \,, & y \in s \,, \\
    +1 \,, & y \notin s \,, 
  \end{cases}
\end{equation}
where $y$ is a point in the fiber of $\cP$ over $x\in C$
that one blows up in the
process of modification. Note there are $(q+1)$ choices for $y$ since
we assume $\deg x = 1$. 

The bundles $\cP_k$ for $k>0$ are distinguished by having a unique section
\begin{equation}
s_\infty: C= \bP\left(\cO(k) \right) \xrightarrow{\quad \quad}
\bP\left(\cO \oplus \cO(k) \right) =
\cP_k\label{eq:193}
\end{equation}
with a negative self-intersection $(s_\infty,s_\infty) = - k$.
According to  \eqref{eq:161-2},
the corresponding $(q+1)$ neighbors of $\cP_k$ look
like this:
\begin{gather}
\raisebox{-2.25cm}{\begin{tikzpicture}[scale=0.9]
\node (0) at (0,0) [circle,draw,minimum width = 1.3cm,fill=gray!3] {$0$}; 
\node (1) at (2,0) [circle,draw,minimum width = 1.3cm,fill=gray!7] {$1$};
\node (2) at (1.246979604, 1.563662965) [circle,draw,minimum width =
1.3cm,fill=gray!7] {$1$};
\node (3) at (1.246979604, -1.563662965) [circle,draw,minimum width =
1.3cm,fill=gray!7] {$1$};
\node (4) at (-0.4450418670, 1.949855825) [circle,draw,minimum width =
1.3cm,fill=gray!7] {$1$};
\node (5) at (-0.4450418670, -1.949855825) [circle,draw,minimum width =
1.3cm,fill=gray!7] {$1$};
\node (6) at (-1.801937736, 0.8677674786) [circle,minimum width =
1.3cm] {};
\node (7) at (-1.801937736, -0.8677674786) [circle,minimum width =
1.3cm] {};
\node (8) at (-2, 0) [circle,minimum width =
1.3cm] {$\dots$};
\draw (0)--(1); 
\draw (0)--(2); 
\draw (0)--(3);
\draw (0)--(4); 
\draw (0)--(5);
\draw (0)--(6); 
\draw (0)--(7); 
\draw (0)--(8); 
\end{tikzpicture}}\,, \quad 
\raisebox{-2.25cm}{\begin{tikzpicture}[scale=0.9]
\node (0) at (0,0) [circle,draw,minimum width = 1.3cm,fill=gray!40] {$k$}; 
\node (1) at (2,0) [circle,draw,minimum width = 1.3cm,fill=gray!56] {$k+1$};
\node (2) at (1.246979604, 1.563662965) [circle,draw,minimum width =
1.3cm,fill=gray!32] {$k-1$};
\node (3) at (1.246979604, -1.563662965) [circle,draw,minimum width =
1.3cm,fill=gray!32] {$k-1$};
\node (4) at (-0.4450418670, 1.949855825) [circle,draw,minimum width =
1.3cm,fill=gray!32] {$k-1$};
\node (5) at (-0.4450418670, -1.949855825) [circle,draw,minimum width =
1.3cm,fill=gray!32] {$k-1$};
\node (6) at (-1.801937736, 0.8677674786) [circle,minimum width =
1.3cm] {};
\node (7) at (-1.801937736, -0.8677674786) [circle,minimum width =
1.3cm] {};
\node (8) at (-2, 0) [circle,minimum width =
1.3cm] {$\dots$};
\draw (0)--(1); 
\draw (0)--(2); 
\draw (0)--(3);
\draw (0)--(4); 
\draw (0)--(5);
\draw (0)--(6); 
\draw (0)--(7); 
\draw (0)--(8); 
\end{tikzpicture}}\,, \quad k>0 \,, \label{kapusta} \\
 \textup{\sl The neighborhood of $\cP_k$ in $\Gr_\bG$.} \notag 
\end{gather}
where abbreviate $\cP_k$ to just $k$. Note that $\cP_{-k} \cong
\cP_{k}$; with this in mind one does not need to distinguish between
the two cases
in \eqref{kapusta}. 

\subsubsection{}

The bundles $\cP_k$ represents the the $\bG(\Bbbk[t^{-1}])$-orbits
on $\Gr_{\bG}$ for $\bG=PGL(2)$. The local rules \eqref{kapusta}
give a unique reconstruction of those up to an automorphism of the
tree, see Figure \eqref{tree2}. 
\begin{gather}
  \label{tree2}
  \raisebox{-100pt}{\includegraphics[height=200pt]{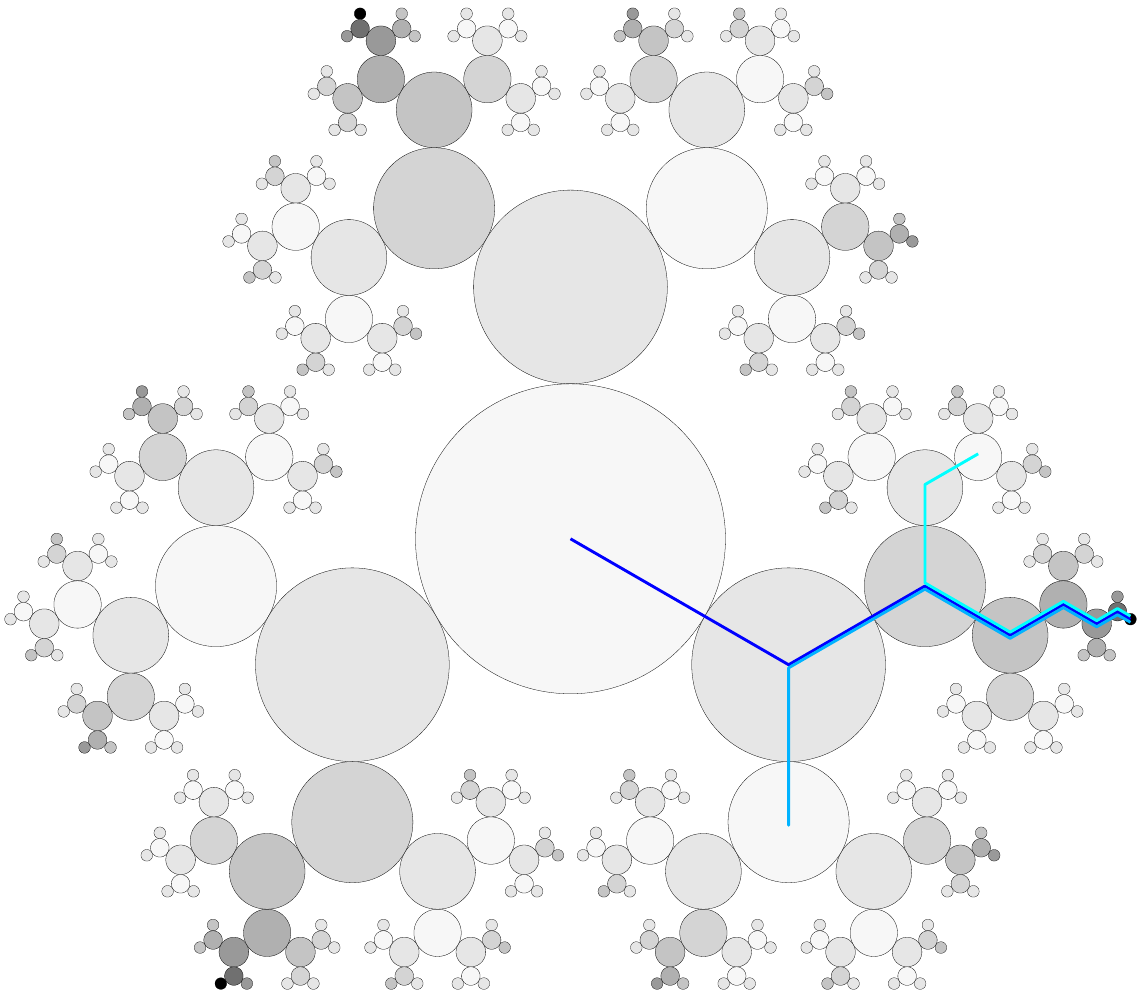}} \\
  \textup{\sl Orbits of $\bG(\Bbbk[t^{-1}])$ in $\Gr_\bG$.} \notag 
\end{gather}
The grayscale in \eqref{tree2} encodes the isomorphism type as
follows
\begin{equation}
  \label{eq:162}
  \raisebox{-0.6cm}{\begin{tikzpicture}
    \selectcolormodel{gray}
    \node (0) [circle,draw,minimum width = 1.3cm,fill=gray!3]  at
    (0,0) {$\cP_0$};
    \node (1)  [circle,draw,minimum width = 1.3cm,fill=gray!10]  at (2,0) {$\cP_1$}; 
 \node (2)  [circle,draw,minimum width = 1.3cm,fill=gray!17]  at (4,0) {$\cP_2$}; 
 \node (3)  [circle,draw,minimum width = 1.3cm,fill=gray!23]  at (6,0)
 {$\cP_3$};
  \node (4)  [circle,draw,minimum width = 1.3cm,fill=gray!31] at (8,0)
  {$\cP_4$};
  \node (5) [circle,draw,minimum width = 1.3cm,fill=gray!40]  at (10,0)
  {$\cP_5$};
  \node (6) [circle,minimum width = 1.3cm]  at (12,0)
  {$\dots$};
  \draw (0)--(1); \draw (1)--(2); \draw (2)--(3);  \draw (3)--(4);
   \draw (4)--(5); \draw (5)--(6); 
  \end{tikzpicture}} 
\end{equation}
Any chain of the form \eqref{eq:162}  is a 
fundamental domain for $PSL(2,\Bbbk[t^{-1}])$. Three such domain
are plotted in \eqref{tree2}.  They differ by shifts by element of the 
subgroup 
\begin{equation}
\Bbbk^{\infty} \cong \left\{
  \begin{pmatrix}
    1 & \Bbbk[t^{-1}] \\
    0 & 1 
  \end{pmatrix} \right\} \subset
PSL(2,\Bbbk[t^{-1}])\,, \label{eq:163}
\end{equation}
which fixes the corresponding cusp.

One may compare and
contrast this action with its Archimedean analog, namely,  the action
of $PSL(2,\Z)$ on the hyperbolic plane. 
The subgroup \eqref{eq:163} is the nonarchimedian analog of
\[
\left\{
  \begin{pmatrix}
    1 & \Z \\
    0 & 1 
  \end{pmatrix} \right\} \subset
PSL(2,\Z) \,.
\]
Together with $PSL(2,\Bbbk)$, the subgroup \eqref{eq:163} generates $PSL(2,\Bbbk[t^{-1}])$.

\subsubsection{} 

Summarizing the above discussion, one can say that for $\bG = PGL(2)$,
the action of the
spherical  Hecke algebra on functions on $\Bun_\bG$ amounts to the action
of one operator $\Delta$ on functions $f(k)$ of one variable $k \in \Z_{\ge
  0}$ given by the formula
\begin{equation}
  \label{eq:173}
  \left[\Delta f\right] (k) = q f(k-1) + f(k+1) \,, \quad f(-1) = f(1)
  \,. 
\end{equation}
This operator is self adjoint with respect to the natural Hermitian
inner 
product in $L^2(\Bun_\bG)$ given by
\begin{equation}
  \label{eq:174}
  \| f \|_{L^2(\Bun_\bG)} = \sum_{k \ge 0}
  \frac{|f(k)|^2}{|\Aut(\cP_k)|}  \,, 
\end{equation}
which is elementary to see directly using
\begin{equation}
  \label{eq:166}
  |\Aut(\cP_k)| =
  \begin{cases}
    q (q^2-1) \,, &  k =0 \,, \\
    q^{k+1} (q-1) \,, & k > 0 \,. 
  \end{cases}
\end{equation}

\subsection{Spectrum of $\Delta$}

\subsubsection{}

The operator \eqref{eq:173} is a second-order 
difference operators with constant coefficients. Ignoring the
boundary conditions, we find two eigenfunctions
\[
f_{\pm}(k;a) = \left(q^{1/2} a^{\pm 1} \right)^k, \quad k \in \Z, a
\in \Ct\,, 
\]
with eigenvalue
\begin{equation}
\Delta f_{\pm}(k;a) = \lambda(a) \, f_{\pm}(k;a)
\,, \quad \lambda(a) = q^{1/2} (a + a^{-1})  \,. \label{eq:182}
\end{equation}
The boundary condition in \eqref{eq:173} singles out one
particular linear combination of those
\begin{equation}
  \label{eq:175}
  \Eis(k,a) = \left(q^{1/2} a\right)^k + \frac{q a^2 -1}{a^2 - q} \, 
  \left(q^{1/2} a^{-1} \right)^k\,, 
\end{equation}
which is the Eisenstein series for $\Bun_{PSL(2)}(\bP^1)$, see below.

\subsubsection{}

By construction, $\Eis(k,a)$ and $\Eis(k,a^{-1})$ are proportional. In
fact, the elementary relation
\begin{equation}
  \label{eq:180}
   \Eis(k,a^{-1}) = \frac {a^2 - q}{q a^2 -1} \, \Eis(k,a) 
\end{equation}
is the simplest instance of the general functional equation for
Eisenstein series.

\subsubsection{} 

In the spectral theory of differential and difference operators,
one views \eqref{eq:175} as a superposition of a wave and its
reflection by the boundary. One reads off properties of the
system, such as its spectrum,  from the properties of the
coefficient in \eqref{eq:175} as a function of $a$, see e.g.\
\cite{Faddeev, Leon}. This analysis is
elementary in the case at hand. 

\subsubsection{}

Note that none of the functions \eqref{eq:175} lies in $L^2$. Indeed,
for large $k$ it grows at least as $q^{k/2}$, making the sum
\eqref{eq:174} diverge. Also note that
\[
\Res_{a= \pm \sqrt{q}} \Eis(k,a) \, \frac{da}{a}
= \frac{q-q^{-1}}{2 } \, (\pm
1)^{k} 
\]
are $L^2$-eigenfunctions with eigenvalues $\pm (q+1)$. As we will
see momentarily
\begin{equation}
  \label{eq:176}
  \Spec(\Delta) = \big[-2\sqrt{q},2 \sqrt{q}\big] \cup \{\pm (q+1)\} \,, 
\end{equation}
with multiplicity one. The spectral measure on the continuos spectrum in \eqref{eq:176}  is in the Lebesgue measure
equivalence class. 

\subsubsection{}

One popular starting point for spectral decompositions is the 
identity
\begin{equation}
f = \frac{1}{2\pi i} \oint_{|z|\gg 1} dz \, \frac{1}{\Delta-z} 
\, f\label{eq:165}
\end{equation}
valid for any function or vector $f$ and a bounded linear operator
$\Delta$. The operator $(\Delta-z)^{-1}$ in \eqref{eq:165} is called the
resolvent of the operator $\Delta$. It is analytic for $|z| \gg 1$ and
its residue at $z=\infty$ is the identity operator, whence
\eqref{eq:165}.

If $\Delta$ is self-adjoint then by the abstract spectral theorem
one has the following integral representation for the matrix
coefficients of the resolvent
\begin{equation}
  \label{eq:321-2}
  \left( (\Delta -z )^{-1} f_1, f_2\right) = \int_\R \frac{f_1(\lambda)
    \overline{f_2(\lambda)}}{\lambda - z} \,
  \mu_{\textup{spectral}}(d\lambda) \,.
\end{equation}
Here we identify elements $f_i$ of our Hilbert space with
functions\footnote{in principle, vector-valued, but scalar-valued in our case} on the spectrum of $\Delta$ and
$\mu_{\textup{spectral}}$ is the spectral
measure. It is clear that the atoms of $\mu_{\textup{spectral}}$
produce poles in \eqref{eq:321-2}, while the absolutely continuous
part of $\mu_{\textup{spectral}}$ produces a jump across the real
axis in \eqref{eq:321-2}. The magnitude of this jump is proportional to the
density of $\mu_{\textup{spectral}}$. A concrete spectral decomposition can thus
be obtained by moving the contour of integration in \eqref{eq:165} so
that to make it encircle the singularities of $(\Delta -z )^{-1}$. 

\subsubsection{}\label{s_compl} 

A good set of functions $f$ to plug into \eqref{eq:165} is not the
Eisenstein series themselves, but their linear combinations of the
form
\begin{align}
  \label{eq:177}
  \Eis_\omega(k) &= \oint_{|a|\gg 1} \Eis(k;a) \, \omega(a) \, \frac{da}{2\pi
                   i a}  \,, \\
  &= \oint_{|a|=1} \Eis(k;a) \, \omega(a) \, \frac{da}{2\pi
                   i a} + \frac{q-q^{-1}}{2} \sum_{\pm } \omega(\pm
    q^{1/2}) \, (\pm 1)^k \,, \label{eq:177bis}
\end{align}
where $\omega(a) \in \C[a^{\pm 1}]$. 
These are 
called pseudo-Eisenstein series in the theory of automorphic forms.
In other situations, they may be called wave packets, etc., that is,
linear combinations of eigenfunctions that decay at infinity sufficiently
rapidly. 

Since \eqref{eq:177} picks out the coefficients of the $a\to \infty$
expansion of $\Eis(k;a)$, it is easy to see that each $\Eis_\omega(k)$ is
compactly supported and that the map
\begin{equation}
  \label{eq:186}
  \C[a^{\pm 1}] \xrightarrow{\quad \omega \mapsto \Eis_\omega
    \quad} \, \, 
  \begin{matrix}
    \textup{compactly supported} \\
    \textup{functions on $\Bun_\bG$} 
  \end{matrix} 
\end{equation}
is surjective. The right hand side is dense in $L^2(\Bun_\bG)$ and
a very useful perspective on the spectral decomposition is to describe
$L^2(\Bun_\bG)$ as the completion of $\C[a^{\pm 1}]$ with respect to
the seminorm  $\| \Eis_\omega \|_{L^2(\Bun_\bG)}$ (which includes
quotienting out by the seminorm's kernel). This seminorm 
can be computed 
abstractly in great generality, see Section \ref{s_const_norm}. 

\subsubsection{}\label{s_spectrum1} 
By construction
\begin{equation}
  \label{eq:178}
  \frac{1}{\Delta-z} 
  \, \Eis_\omega =
 \oint_{|z| \gg |a| \gg 1}
  \frac{da}{2\pi i a}
  \,
  \frac{\omega(a)}{
    \lambda(a)- z} \, \Eis(a) \,. 
\end{equation}
Moving $z$ away from $z=\infty$ means shrinking the contour of
integration in \eqref{eq:178} to the unit circle $|a|=1$ as in
\eqref{eq:177bis}. 
Note the function $\lambda(a)$ is one-to-one for $|a|>1$ and
and two-to-one for $|a|=1$, creating the expected discontinuity of the
resolvent along
\[
\lambda(\{|a|=1\}) = \big[-2\sqrt{q},2 \sqrt{q}\big]  \, .
\]
This proves \eqref{eq:176}.

\subsubsection{}

For $\bG=SL(2)$, we need to consider only even $k$ and the operator
$\Delta^2$. This means we will have $b=a^2$ as the new parameter
for Eisenstein series
\begin{equation}
  \label{EisSL2}
  \Eis_{SL(2)}(\cO(m)\oplus \cO(-m),b) = \left(q b\right)^m + \frac{q b -1}{b - q} \, 
  \left(q b^{-1} \right)^m\,.  
\end{equation}
In particular, the eigenfunctions $(\pm 1)^k$ get identified, while
the continuous spectrum retains its shape.

\subsection{Same spectrum, as seen in the dual group} 

\subsubsection{}
The interpretation of this spectrum from the point of the general theory
is the following. One considers the Langlands dual group $\bG^\vee$
and identifies, following Satake,
\[
\lambda: \SH_x  (\bG) \otimes_\Z \Z[q^{\pm 1/2}]
\xrightarrow{\quad \sim \quad} K_{\bG^\vee} (\pt) \otimes_\Z
\Z[q^{\pm 1/2}]  \,, 
\]
where the isomorphism is an isomorphism of algebras. In particular,
\[
PGL(2)^\vee=SL(2)\,, 
\]
and vice versa, and
\begin{equation}
  \label{eq:181}
  \lambda ( \Delta) = q^{1/2} \otimes \textup{defining representation
    of $SL(2)$} \,. 
\end{equation}
To find an agreement with \eqref{eq:182} above, we observe that
\[
\lambda(a) = q^{1/2} \tr
\begin{pmatrix}
  a \\
  & a^{-1} 
\end{pmatrix}
\]
is the trace in the defining representation of $a$, viewed as a
conjugacy class in $SL(2)$. Viewed from this angle,
\begin{equation}
  \label{eq:183}
  \Spec(\SH_x  (\bG)) \subset \textup{conjugacy classes in
    $\bG^\vee(\C)$}\,. 
\end{equation}

\subsubsection{}

We hope the reader will enjoy checking that for both $PGL(2)$ and
$SL(2)$ the spectrum computed above can be described as
follows
\begin{equation}
  \label{eq:184}
  \Spec(\SH_x  (\bG)) = \bigsqcup_{\pi: SL(2,\C)\to \bG^\vee(\C)}
  \pi \left( \begin{pmatrix}
  q^{1/2} \\
  & q^{-1/2} 
\end{pmatrix} \right) \cdot \bG^\vee(\C)^\pi_{\textup{compact}}
\end{equation}
where the union is over all homomorphisms
$\pi$, the group $\bG^\vee(\C)^\pi$ is the centralizer of the image of
$\pi$, and $\bG^\vee(\C)^\pi_{\textup{compact}}$ is its maximal
compact subgroup.

Indeed, in the case at hand, there are only two possibilities for $\pi$. We have the
trivial representation, which gives
\begin{equation}
  \label{eq:185b}
  \Spec(\SH_x  (PGL(2)))_{\textup{continous}} = \left\{
    \begin{pmatrix}
  a \\
  & a^{-1}
\end{pmatrix} , |a|=1 \right\} \,, 
\end{equation}
and the defining
representation which gives 
\begin{equation}
  \label{eq:185}
  \Spec(\SH_x  (PGL(2)))_{\textup{discrete}} = \left\{
    \pm  \begin{pmatrix}
  q^{1/2} \\
  & q^{-1/2} 
\end{pmatrix}  \right\} \,. 
\end{equation}
For $\bG = SL(2)$, one takes the image of \eqref{eq:185b} and
\eqref{eq:185} in $PGL(2,\C)$.

\subsubsection{}

Formula \eqref{eq:184} admits extremely far-reaching generalizations of the
form
\begin{equation}
  \label{eq:184bb}
  \Spec(\SH_x  (\bG)) = \bigsqcup_{\pi:
    \Gamma \times SL(2,\C)\to \bG^\vee(\C)}
  \pi \left( \Fr_x \times \begin{pmatrix}
  \| \varpi_x \|^{1/2} \\
  & \| \varpi_x \|^{-1/2} 
\end{pmatrix} \right) \cdot \bG^\vee(\C)^\pi_{\textup{compact}} \,, 
\end{equation}
where the spectrum is the joint spectrum of the spherical Hecke
algebras for all nonarchimedian
places of a global field $\F$, $\Gamma$ is some
version of the unramified Galois group of $\F$, $\Fr_x$ is the
conjugacy class of the Frobenius element at $x$, and
$\| \varpi_x \|$ is the norm of a uniformizing element of $\F_x$. 


\subsection{Eisenstein series}

\subsubsection{}

The space $\Bun_\bG$ depends covariantly on the group $\bG$, meaning
that a group homomorphism $\bG \to \bG'$ induces a map $\Bun_\bG \to
\Bun_{\bG'}$. In particular, let $\sP \to \bG$ be an inclusion of a
parabolic subgroup and let $\sP \to \bM$ be the quotient by the unipotent
radical. Both Eisenstein and pseudo-Eisenstein series on $\Bun_\bG$
are produced by the
pull-push $\pi_{\bG,*} \pi_\bM^*$ in the diagram 
\begin{equation}
  \label{eq:187}
  \xymatrix{
    & \Bun_{\sP}(C) \ar[ld]_{\pi_\bM} \ar[rd]^{\pi_\bG}\\
    \Bun_\bM(C) && \Bun_\bG(C) \,,}
\end{equation}
which appears in many places in our narrative. 
Here we have discrete sets with the canonical measure $|\Aut|^{-1}$,
which provides a natural bijection between functions and measures. The
pushforward of a functions is really the pushforward of the
corresponding measure, in other words
\begin{equation}
  \label{eq:197}
  \big[\pi_* f \big](x) = \sum_{x' \in \pi^{-1}(x)}
  \frac{|\Aut(x\phantom{{}'})|}{|\Aut(x')|} \, f(x') \,. 
\end{equation}
By construction, $\pi_*$ and $\pi^*$ are transpose with respect to the
natural duality between functions and measures.

\subsubsection{}
In our running example, we have
\begin{equation}
  \label{eq:179}
  \bM \cong GL(1) = \{ z\} \xleftarrow{\quad}
  \sP = \left\{
  \begin{pmatrix}
    z & * \\
    0 & 1 
  \end{pmatrix}
\right\}
\xrightarrow{\quad}
\bG = PGL(2) \,. 
\end{equation}
A $\sP$-bundle over a curve $C$ is an extension
\begin{equation}
  \label{eq:198}
  0 \to \cM \to \cV \to \cO_C \to 0 \,, \quad \cM \in \Bun_{GL(1)} 
  = \Pic(C)(\Bbbk) \,. 
\end{equation}
Up to isomorphism, it is uniquely encoded by the $\bP^1$-bundle
$\cP = \bP(\cV)$ with a section
\[
s_0: C = \bP(\cM) \to \cP \,.
\]
We have $\cM=s_0^*(\cO_\cP(-1))$ and, in particular,
\begin{equation}
\deg s_0 = - \deg \cM = (s_0,s_0) \,.\label{eq:199}
\end{equation}

\subsubsection{}
The numerator in \eqref{eq:197} means that, in the
corresponding count, we \emph{do not} identify sections that
differ by the automorphism of $\cP$.
In other
words, every one of the the $\frac{|\Aut(\cP)|}{|\Aut(\cP,s)|}$ many
point in each $\Aut(\cP)$-orbit in 
\begin{equation}
  \label{eq:188}
  \left[\pi_{\bG,*} \pi_\bM^* \, \phi \right] (\cP) =
  \sum_{\textup{sections $s: C\to \cP$}} \phi(s^* \cO(-1)) \,, 
\end{equation}
contributes the same value $\phi(s^* \cO_{\bP^1}(1))$. Here $\phi$ is an
arbitrary function on
\begin{equation}
\Pic(C)(\Bbbk) = \Z \oplus \Pic_0(\C)(\Bbbk) \label{eq:200} \,,
\end{equation}
where the $\Z$-coordinate is given by \eqref{eq:199}. 
In our running example $C=\bP^1$, the second summand in \eqref{eq:200} is zero.

\subsubsection{}

As we will see momentarily, choosing $\phi(\deg s)$ to be compactly
supported gives pseudo-Eisenstein series on
$\Bun_{PGL(2)}$. Equivalently, we will check that taking
\begin{equation}
\phi(\cM) = \left(q^{1/2} a\right)^{\deg \cM} \label{eq:191}
\end{equation}
recovers \eqref{eq:175}, which means that 
\begin{equation}
  \label{eq:192}
  \Eis(\cP,a) =
   \sum_{\textup{sections $s: C\to \cP$}} \left(q^{1/2}
     a\right)^{-\deg(s)}  \,. 
\end{equation}
The shift by $q^{1/2}$ is the shift by
$q^{\rho}$ where $\rho=1/2$ is the half-sum of weights of $\bM$ in
the unipotent radical of $\sP$.

\subsubsection{}
It is obvious from \eqref{eq:161-2} that the right-hand side of
\eqref{eq:192} is
an eigenfunctions of $\Delta$ with eigenvalue $\lambda(a)$.
Further, the normalization is correct since $\cP_k$ contains
a unique section \eqref{eq:193} of degree $-k$. Thus, \eqref{eq:192}
is established.

Nonetheless, it will in instructive to do the actual counting.
Actually, we will do 
explicitly a middle-school version of the counting we
did just slightly more abstractly in Section \ref{s_curves}.

\subsection{Counting quasisections}

\subsubsection{}

We
have $\cP_k \setminus s_\infty \cong \cO_{\bP^1}(k)$ and hence
all other sections can be identified with rational sections of
$\cO(k)$. In other words, they have the form
\[
s(t) = \frac{F_{d+k}(t)}{G_{d}(t)}, \quad d=0,1,2,\dots\,,
\]
where $F_{d+k}$ and $G_{d}$ are polynomials of the indicated
degree. The degree of this section is $k+2d$.

\subsubsection{}

One way to count these, is to first count pairs $(F_{d+k},G_{d})$
with $G_d \ne 0$ up to an overall constant
multiple, ignoring the fact that they may have a common nonconstant
factor. The corresponding pairs $(F_{d+k}(t),G_{d}(t))$ describe 
quasisections of $\cO_{\bP^1}(k)$, see Section \ref{s_curves}. They 
form a complement to a linear subspace in a projective space of dimension
$2d+k+1$. Hence we get the count 
\begin{equation}
  \label{eq:194}
  \sum_{d\ge 0} \left(q^{1/2}
    a\right)^{-k-2d} \frac{q^{2d+k+2}-q^{d+k+1}}{q-1} =
   \frac{q}{(1-q a^{-2})(1-a^{-2})} \, \left(q^{1/2}
    a^{-1}\right)^{k} \,. 
\end{equation}
For the common multiple up to a constant factor, we similarly get the
count
\begin{equation}
  \label{eq:194b}
  \sum_{d\ge 0} \left(q^{1/2}
    a\right)^{-2d} \frac{q^{d+1}-1}{q-1} =
   \frac{1}{(1-a^{-2})(1-q^{-1}a^{-2})} \,. 
\end{equation}

\subsubsection{}

By definition, counting nonzero polynomials up to a constant factor is
the same as counting closed points of $\bP^1$. This explains why the $\zeta$-function
\begin{equation}
  \label{eq:195}
  \zeta_{\bP^1}(t) = \frac{1}{(1-t)(1-qt)} 
\end{equation}
appears in the answer. Introducing the symmetrized version of \eqref{eq:195},
also known as the completed $\zeta$-function
\begin{equation}
  \label{eq:195-2}
  \xi_{\bP^1}(t) = \frac{q^{1/2} t}{(1-t)(1-qt)} \,, 
\end{equation}
we can write the ratio of \eqref{eq:194} and \eqref{eq:194b} as
follows
\begin{equation}
  \label{eq:196}
  \Eis(\cP_k, a) = \left(q^{1/2}
    a\right)^{k} + \frac{\xi(a^{-2})}{\xi(a^2)} \left(q^{1/2}
    a^{-1}\right)^{k} \,. 
\end{equation}
Eisenstein series for $C=\bP^1$ are proportional to the their constant
terms, see \eqref{eq:201} below. The formula \eqref{eq:196} is thus
equivalent to the simplest instance of Langlands' constant term
formula.

\subsection{Constant term and $L^2$-norms}\label{s_const_norm} 

\subsubsection{}

By definition, the constant term
\[
\CT = \pi_{\bM,*} \pi_{\bG}^*
\]
is the transposed operator in the diagram \eqref{eq:187}.
For $C = \bP^1$ and $\deg \cM \ge -1$, any extension in \eqref{eq:198}
is trivial, thus
\begin{equation}
\CT (\phi) (\cM) = \frac{1}{q^{\deg \cM +1}} \, \phi( \cM
\oplus \cO) \,, \quad \deg \cM \ge -1 \,. \label{eq:201}
\end{equation}
The prefactor in \eqref{eq:201} comes from the ratio 
\[
\frac {|\Aut (\cM)|} {|\Aut ( \cM \subset \cM\oplus\cO)|}=
\frac1{|H^0(\cM)|} = \frac{1}{q^{\deg \cM +1}} \,. 
\]
in formula \eqref{eq:197}. 
Thus
\begin{equation}
\CT(\Eis(a))(\cO(k)) = \frac1q \left[\left(q^{-1/2}
    a\right)^{k} + \frac{\xi(a^{-2})}{\xi(a^2)} \left(q^{-1/2}
    a^{-1}\right)^{k} \right]\label{eq:189} \,. 
\end{equation}
Langlands gave a very general formulation and very general proof of
such constant term formulas in \cite{L}, further investigated by Harder
in \cite{Harder} and many other authors.

\subsubsection{}

For a general curve $C$ and Eisenstein series associated
to $\bG/\bB$, where $\bG$ is a split reductive group and $\bB\subset
\bG$ is the Borel subgroup, formula \eqref{eq:189} generalizes as follows. In this case
$\bM$ is the maximal torus of $\bG$, hence
\[
\Bun_\bM = \Cochars(\bM) \otimes_\Z \Pic(C)
\xrightarrow{\quad \deg \quad} \Cochars(\bM) \,. 
\]
The variable $a$ in the formula
\begin{equation}
\phi(\cM) = (q^{\rho} a)^{\deg \cM} \label{eq:198-2}
\end{equation}
should thus be interpreted 
as an element of the dual torus 
\[
a \in \bM^\vee  = \Chars(\bM) \otimes_\Z \Ct \,.
\]
The Weyl group $W$ acts on functions of $a$, such as \eqref{eq:198-2}
in the natural way. Consider the operator
\begin{align}
\boldsymbol{\Sigma} & = \sum_{w\in W} \prod_{
  \substack{\alpha > 0 \\ w^{-1} \alpha < 0}}
  \frac{\xi(a^{-\alpha})}
  {\xi(a^{\alpha})} \,   w \label{eq:199-2} \,, \\ 
                    & =  \boldsymbol{\Pi}^{-1} \left(\sum_{w\in W}  w
                      \right) \boldsymbol{\Pi}
                      \label{eq:199bis}
\end{align}
where $\boldsymbol{\Pi} = \prod_{\alpha > 0} \xi(a^{-\alpha})$. It
follows from \eqref{eq:199bis} that
\begin{equation}
  \label{eq:190}
  \boldsymbol{\Sigma}^2  = |W| \boldsymbol{\Sigma} \,, 
\end{equation}
which means that $\Sigma$ is proportional to a projector when
acting on rational functions of $a$. With these notations, 
\begin{equation}
  \label{eq:197-2}
  \CT (\Eis(\, \cdot \, )) \,  \propto \, \boldsymbol{\Sigma}
\end{equation}
when acting on functions of the form \eqref{eq:198-2}. For functions
$\phi$ that involve a nontrivial character of $\Pic_0(C)(\Bbbk)$, the function
$\xi$ in \eqref{eq:199-2} should be replaced by the corresponding
Dirichlet L-function.

\subsubsection{}

One immediate application of constant term formulas is a
formula of the norm $\| \Eis_\omega \|_{L^2(\Bun_\bG)}$ of a
pseudo-Eisenstein series. This is because we can 
we can express norms on $\Bun_\bG$ in terms of norms on $\Bun_\bM$
using the relation 
\begin{equation}
\left(\pi_{\bG,*} \pi_\bM^* \right)^* = \CT. \label{eq:11}
\end{equation}
The space $\Bun_\bM$ for $\bM=GL(1)$ is just $\Z$, each point of which has
the automorphism group $GL(1,\Bbbk)$ of order $q-1$. We can
go between functions
\[
\omhat: \Z \to \C
\] 
and functions $\omega(a)$ on $\C^\times = \bM^\vee$ by means of the Fourier
transform, which we normalize by
\begin{equation}
\omega(a)  =\sum_k \omhat(k) \,  (q^{1/2} a)^{-k} \,.\label{eq:9}
\end{equation}
The shift in \eqref{eq:9} is introduced so that there is no shift
in
\begin{equation}
  \Eis_\omega =\pi_{\bG,*} \pi_\bM^* \,\,  \omhat(\deg(\, \cdot \,))\,.  \label{eq:203}
\end{equation}
We have
\begin{equation}
  \| \omhat \|_{L^2(\Bun_\bM)} = \tfrac{1}{q-1} \sum_{k\in \Z}
  |\omhat(k)|^2 \,\label{eq:10} \,,
  \end{equation}
where the prefactor accounts for the automorphisms.

\subsubsection{}

Formulas \eqref{eq:9} and \eqref{eq:10} imply 
\begin{equation}
  \label{eq:205}
  \Big( (q^{-1/2} a)^{\deg}, \omhat\Big)_{L^2(\Bun_\bM)} =
  \tfrac{1}{q-1} \, \bar{\omega}(a^{-1}) \,, 
\end{equation}
where $\bar{\omega}(a) = \overline{\omega(\overline{a})}$. 
From this, \eqref{eq:11} and \eqref{eq:203}, we conclude 
\begin{equation}
 \notag 
  \Big(\Eis(a), \Eis_{\omega}\Big)_{L^2(\Bun_\bG)} = \Big(\!\CT \Eis(a),
                                              \omhat\Big)_{L^2(\Bun_\bM)}
  = \tfrac{1}{q(q-1)} \boldsymbol{\Sigma} \, \bar{\omega}(a^{-1})
        \,, 
         \label{eq:204}
\end{equation}
where 
\begin{equation}
  \label{eq:206}
  \boldsymbol{\Sigma} \, \omega (a) =  \omega(a) +  \frac{\xi(a^{-2})}{\xi(a^{2})}
  \omega(a^{-1})  \,. 
\end{equation}
It follows that 
\begin{equation}
  \label{eq:202}
 \| \Eis_\omega \|_{L^2(\Bun_\bG)} = 
\tfrac1{q(q-1)} \oint_{|a|\gg 1} \frac{da}{2\pi i a}       \, 
                                                   \omega(a) 
                                                   \boldsymbol{\Sigma}
                                     \, \bar{\omega}(a^{-1})\,.  
\end{equation}

\subsubsection{}
Moving the contour of integration in \eqref{eq:202} as in
\eqref{eq:177bis}, we obtain 
\begin{alignat}{2} 
  \notag 
 \| \Eis_\omega \|_{L^2(\Bun_\bG)} 
  &= 
\tfrac1{q(q-1)} \oint_{|a|=1} \frac{da}{2\pi i a}       \, 
                                                   \omega(a) 
                                                   \boldsymbol{\Sigma}
    \, \overline{\omega(a)} &&+
    r_1 |\omega(q^{1/2})|^2 + r_2 |\omega(-q^{1/2})|^2\,, \\
  &= 
\underbrace{\tfrac1{2q(q-1)}  \,\,\, \Big\| \overline{\boldsymbol{\Sigma}} \, \omega
    \Big\|^2_{L^2(|a|=1)} } _{\textup{continuous spectrum}}&&+ 
  \underbrace{ r_1 |\omega(q^{1/2})|^2 + r_2 |\omega(-q^{1/2})|^2}
  _{\textup{discrete spectrum}}\,,  \label{eq:202b}
\end{alignat}
where the coefficients $r_1,r_2 >0$ are obtained from the residues of
$\boldsymbol{\Sigma}$ at $a=\pm q^{1/2}$, and where in \eqref{eq:202b}
we used the fact that $\frac12\boldsymbol{\Sigma}$ is a self-adjoint
projector acting on $L^2(|a|=1)$.

\subsubsection{}

Let us now go back to the discussion in Section \ref{s_compl}. Formula
\eqref{eq:202b} describes a seminorm on $\C[a^{\pm 1}]$. Its three terms
involve values of $\omega$ at disjoint closed sets, hence the
completion $L^2(\Bun_\bG)$ 
splits into a direct sum of two 1-dimensional eigenspaces and
their orthogonal complement. The latter is a direct integral of
Eisenstein series $\Eis(a)$ for $a$ on the unit circle $|a|=1$ modulo
$a\mapsto a^{-1}$. The $a\mapsto a^{-1}$ quotient appears as
the quotient by the kernel of the spectral
projector $\frac12\overline{\boldsymbol{\Sigma}}$. It reflects
the functional equation for $\Eis(a)$. 

\subsection{A geometric interpretation}

\subsubsection{}
To conclude this Appendix, we want to explain our geometric
point of view on spectral decomposition in the 
running example. The relevant geometric object is yet another
projective line
\[
  \bP^1 = \bG^\vee/\bB^\vee\,,
\]
which is now interpreted as the flag variety for the Langlands
dual group. It has the natural action of
\[
  \bM^\vee = \{\diag(a^{-1},a) \} \subset \bG^\vee  \,. 
\]
In fact, the group $\bM^\vee$ will play two roles because it both
acts on $\bP^1$ and is also the Levi subgroup of the group
$\bB^\vee$ which appears in the quotient construction of $\bP^1$.
To account for these two roles, one can consider quotient stacks like
$\bM^\vee \backslash \bG^\vee/\bB^\vee$ or even $\bB^\vee \backslash
\bG^\vee/\bB^\vee$, in which $\bM^\vee$ appears  on the left and 
on the right on an equal footing. 
This results in the
identification
\begin{equation}
  \label{eq:12}
  K_{\bM^\vee}\!(\bP^1) = K_{\bB^\vee}\!(\bP^1) = K(\bB^\vee
  \backslash \bG^\vee/\bB^\vee) = \dots = \Z[\bM^\vee] \boxtimes_{\Z[\bM^\vee]^W}
  \Z[\bM^\vee]
  \,,
\end{equation}
where dots denote many other possible geometric realizations of
this algebra, the Weyl group $W\cong S(2)$ acts by $a\mapsto a^{\pm
  1}$, and the $\boxtimes$ symbol is used to distinguish this
tensor product from the tensor product in equivariant K-theory (which
is the algebra product in \eqref{eq:12}). The
identification may be fixed so that
\begin{align*}
  \label{eq:13}
  a \boxtimes 1 &= \textup{character $a$} \otimes \cO_{\bP^1} \,,\\
  1 \boxtimes a & = \cO_{\bP^1}(1) \,. 
\end{align*}
We will write elements of $K_{\bM^\vee}\!(\bP^1) \otimes_\Z\C$ in
the form $\omega_1(a) \boxtimes \omega_2(a)$. 

\subsubsection{}

There is a distinguished linear functional $\int_{\bP^1}$, given by the horizontal
arrow in the following diagram
\begin{equation}
  \label{eq:15}
  \xymatrix{
    K_{\bM^\vee}\!(\bP^1) \ar[rrr]^{\quad\int_{\bP^1}=(\bM^\vee \backslash \bP^1 \to \pt)_*\quad}\ar[rd]_{\chi}&&& \Z\,, \\
    & \Z[\bM^\vee] \ar[rru]_{\textup{\,\, invariants}}
    } 
\end{equation}
where $\chi$ is the Euler characteristic, equivalently pushforward to
$\pt/\bM^\vee$. It may be computed by the localization formula in equivariant
K-theory, or the Weyl character formula for $SL(2)$, or the formula
for the sum of a geometric progression.  The result is
\begin{equation}
  \label{eq:16}
  \int_{\bP^1} \omega_1(a) \boxtimes \omega_2(a) =
\int_{|a|=1} \frac{da}{2\pi i a} \, \omega_1(a) \left( \sum_{w \in W} w \right) 
    \frac{ \omega_2(a)}{1-a^{-2}}  \,. 
\end{equation}
%
Since the symmetrization eliminates the poles in
\eqref{eq:16}, the integration contour $|a|=1$
may be moved arbitrarily.

\subsubsection{}

Consider the cotangent bundle
\[
  \bX^\vee = T^*\bP^1 = T^*(\bG^\vee/\bB^\vee)\,. 
\]
Using the identification
\[
K_{\bM^\vee}\!(\bP^1) = K_{\bM^\vee}\!(\bX^\vee) \,, 
\]
one may similarly consider pushing forward along $\bX^\vee \to \pt$.
Since this pushforward is not proper, a certain care needs to be
exercised. Concretely, let $\bfL \subset \bX^\vee$ be the set of
points that have a limit under the action of $a$ as $a\to 0$. This is
the union of the zero section $\bP^1 \subset T^*\bP^1$ and the cotangent
fiber over one of the fixed
points. Let $\chi_\bfL$ denote Euler characteristic with support in
$\bfL$. It takes values in formal Laurent series in $a^{-1}$. 
Consider the diagram
\begin{equation}
  \label{eq:15-2}
  \xymatrix{
    K_{\Ct_q \times \bM^\vee}\!(\bX^\vee)
    \ar[rrr]^{\quad\int_{\bX^{\!\vee}}\quad}\ar[rd]_{\chi_\bfL}&&& \Z[q^{\pm 1}]\,, \\
    & \Z[q^{\pm1}][a,a^{-1}]\!] \ar[rru]_{\textup{\,\, invariants}}
    } 
\end{equation}
in which we enlarged the equivariance by making $\Ct$ with coordinate
$q$ act by scaling the cotangent directions with weight $q^{-1}$.
Again, by localization or geometric series, we compute
\begin{equation}
  \label{eq:16-2}
  \int_{\bX^{\!\vee}} \omega_1(a) \boxtimes \omega_2(a) =
  \int_{|a|\gg 1} \frac{da}{2\pi i a} \, \omega_1(a)
  \left( \sum_{w \in W} w \right) 
    \frac{\omega_2(a)}{(1-a^{-2})(1-q a^{2})}  \,. 
\end{equation}
Note the similarity between the weight here and the $\zeta$-function
\eqref{eq:195} of $\bP^1$.

\subsubsection{}

Now consider the following  Hamiltonian reduction of $\bX^\vee$ 
\[
 \cT = T^*(\bB^\vee \backslash \bG^\vee/\bB^\vee)  = \bB^\vee
 \backslash \mu_{\bB^\vee}^{\!-1}(0) \,, 
\]
where
\[
  \mu_{\bB^\vee}: \bX^\vee \to \Lie(\bB^\vee)^*\,,
\]
is the moment map for the action of $\bB^\vee$. The 
map $\mu_{\bB^\vee}$
is equivariant for the action of $\Ct_q \times \bM^\vee$ provided
$\Ct_q$ scales the target with weight $q^{-1}$. Again, we have
the identification 
\[
K_{\bM^\vee}\!(\bP^1) = K(\cT) \,, 
\]
and we would like to define a linear functional $\int_\cT$ as before.
Since the map $\cT \to \pt$ is not proper, this does require an
explanation, but the explanation is easier than the one for
$\int_{\bX^\vee}$. 

We can choose $\bB^\vee\supset \bM^\vee$ so that $\bM^\vee$ acts
with weight $a^2$ on its unipotent radical $\bU^\vee\subset \bB^\vee$. Importantly, the points
in the unipotent radical have a limit as $a\to 0$. This means that
one of the two equations in $\mu_{\bB^\vee}=0$ defines 
the Lagrangian $\bfL$. Therefore, the Euler characteristic 
 is well-defined as an element of 
$\Z[q^{\pm1}][a,a^{-1}]\!]$.

The equation $\mu_{\bB^\vee}=0$ cuts the Euler class of a trivial
vector 
bundle with a nontrivial group action. Similarly, the
Chevalley-Eilenberg complex for $\bU^\vee$-invariants cuts out
the Euler class of a trivial line bundle with a nontrivial weight.
This gives 
\begin{equation}
  \label{eq:17}
  \int_{\cT} \omega_1(a) \boxtimes \omega_2(a)  =
  \int_{\bX^{\!\vee}} \cE \omega_1(a) \boxtimes \omega_2(a)  \,, 
\end{equation}
where $\cE$ is the Euler class
\[
  \cE = \underbrace{(1-q)(1-q a^2)}_{\mu_{\bB^\vee}=0}
  \underbrace{(1-a^{-2})}_{\textup{$\bU^\vee$-invariants}} \,. 
\]
This means 
\begin{equation}
  \label{eq:16-3}
  \int_{\cT} \omega_1(a) \boxtimes \omega_2(a) =
 {\scriptstyle{(1-q)}} 
  \int_{|a|\gg 1} \frac{da}{2\pi i a} \,
  \left(\omega_1(a) \omega_2(a) + \frac{\xi(a^{-2})}{\xi(a^{2})} \,
    \omega_1(a) \omega_2(a^{-1})\right)\,, 
\end{equation}
Comparing this with \eqref{eq:202}, we see that 
\begin{equation}
  \label{eq:18}
  \left( \Eis_{\omega_1}, \Eis_{\omega_2}\right)_{L^2(\Bun_\bG)}
  \propto \int_{\cT} \omega_1(a) \boxtimes \omega^*_2(a)\,, 
\end{equation}
where
\begin{equation}
  \label{eq:19}
  \omega^*(a) = \bar{\omega}(a^{-1}) \,. 
\end{equation}
The proportionality constant in \eqref{eq:18} is  interesting
and has to do with volumes of $\Bun_\bG$ and other things. It is, however, immaterial
for the spectral decomposition. 

\subsubsection{}

Now, how does one see the spectral decomposition and specifically
\eqref{eq:184} from this ? For any reductive group $\bG^\vee$, the moment
map
\begin{equation}
  \label{eq:20}
  \mu_{\bG^\vee}: T^*(\bG^\vee/\bB^\vee)  \to \Lie(\bG^\vee)^* \cong
   \Lie(\bG^\vee)
\end{equation}
is the resolution of the cone of nilpotent elements in
$\Lie(\bG^\vee)$. The fibers of this map are called Springer fibers.
The map \eqref{eq:20} induces a map
\begin{equation}
  \label{eq:8}
  \bmu:  \cT\to
\Lie(\bG^\vee)/\bG^\vee\,, 
\end{equation}
the fibers of which are Cartesian squares of the Springer
fibers.

Consider the diagram
\begin{equation}
  \label{eq:21}
  \xymatrix{
    \cT\big/\Ct_q \ar[rr] \ar[dr]_{\bmu}&& \pt\big/\Ct_q \\
    &\Lie(\bG^\vee)\big/(\Ct_q \times \bG^\vee) \ar[ru]} \,. 
\end{equation}
The pushforward $\bmu_*$ in \eqref{eq:21} gives a K-theory class
on the stack of nilpotent conjugacy
classes $e$ for the dual group $\bG^\vee$. These 
correspond to representations $\pi$ in \eqref{eq:184}
via
\begin{equation}
  \label{eq:14}
  e = \pi\left(\!\left(
    \begin{matrix}
    0  & 1\\ 0 & 0
    \end{matrix} \right)\!\right)  \,.
\end{equation}
We have
\begin{equation}
  \label{eq:22}
  \textup{Orbit}(e)\big/ (\Ct_q \times \bG^\vee) \cong
  \pt \big/  (\Ct_q \times \bC(e)) \,, 
\end{equation}
where $\bC(e)$ is the centralizer of $e$ and $q\in \Ct_q$ acts by
$\pi(\diag(q^{1/2},q^{-1/2}))$ in the fiber over $e$. We see
this action of $q$ explicitly in \eqref{eq:184}. Restricted to
\eqref{eq:22}, the northeast
going map in \eqref{eq:21} is the integration over the maximal
compact subgroup of the complex group $\bC(e)$. This
group coincides with $\bG^\vee(\C)^\pi_{\textup{compact}}$ in
\eqref{eq:184}. This identifies the spectrum as a set of conjugacy
classes in $\bG^\vee$.

Further analysis of \eqref{eq:21}
describes spectral projectors as pushforwards along the Spinger
fibers and computes the spectral measure from the
$\Ct_q \times \bC(e)_{\textup{compact}}$-action on the
unipotent radical of $\bC(e)$. See \cite{KO1} for many
technical details required for the complete match. One key
property that needs to be checked is that the pieces corresponding
to each $e$ are positive in the same way as
\eqref{eq:202b} is a sum of orthogonal seminorms. Here the
Hermitian structure is introduced by \eqref{eq:19}. 
It would be
interesting to have an abstract argument showing such positivity.


\begin{bibdiv}
	\begin{biblist}


\bibitem{AO} M.~Aganagic and A.~Okounkov, 
  \emph{Duality interfaces in 3 dimensional theories},
  String-Math 2019, available from
  \url{https://www.math.columbia.edu/~okounkov/papers.html}. 
  


\bib{Atiyah}{book}{
   author={Atiyah, Michael Francis},
   title={Elliptic operators and compact groups},
   series={Lecture Notes in Mathematics, Vol. 401},
   publisher={Springer-Verlag, Berlin-New York},
   date={1974},
   pages={ii+93},
}

\bib{Bez}{article}{
   author={Bezrukavnikov, Roman},
   title={Cohomology of tilting modules over quantum groups and
   $t$-structures on derived categories of coherent sheaves},
   journal={Invent. Math.},
   volume={166},
   date={2006},
   number={2},
   pages={327--357},
}


\bib{BezPos}{article}{
   author={Bezrukavnikov, Roman},
   author={Positselski, Leonid},
   title={On semi-infinite cohomology of finite-dimensional graded algebras},
   journal={Compos. Math.},
   volume={146},
   date={2010},
   number={2},
   pages={480--496},
}


  




\bib{BFN1}{article}{
   author={Braverman, Alexander},
   author={Finkelberg, Michael},
   author={Nakajima, Hiraku},
   title={Towards a mathematical definition of Coulomb branches of
   3-dimensional $\mathcal{N}=4$ gauge theories, II},
   journal={Adv. Theor. Math. Phys.},
   volume={22},
   date={2018},
   number={5},
   pages={1071--1147},
}

\bib{BFN2}{article}{
   author={Braverman, Alexander},
   author={Finkelberg, Michael},
   author={Nakajima, Hiraku},
   title={Coulomb branches of $3d$ $\mathcal{N}=4$ quiver gauge theories and
   slices in the affine Grassmannian},
   note={With two appendices by Braverman, Finkelberg, Joel Kamnitzer,
   Ryosuke Kodera, Nakajima, Ben Webster and Alex Weekes},
   journal={Adv. Theor. Math. Phys.},
   volume={23},
   date={2019},
   number={1},
   pages={75--166},
 }

 \bib{BFN3}{article}{
   author={Braverman, Alexander},
   author={Finkelberg, Michael},
   author={Nakajima, Hiraku},
   title={Ring objects in the equivariant derived Satake category arising
   from Coulomb branches},
   note={Appendix by Gus Lonergan},
   journal={Adv. Theor. Math. Phys.},
   volume={23},
   date={2019},
   number={2},
   pages={253--344},
}

\bib{BravGaits}{article}{
   author={Braverman, A.},
   author={Gaitsgory, D.},
   title={Geometric Eisenstein series},
   journal={Invent. Math.},
   volume={150},
   date={2002},
   number={2},
   pages={287--384},
}


\bib{Vermas}{article}{
   author={Bullimore, Mathew},
   author={Dimofte, Tudor},
   author={Gaiotto, Davide},
   author={Hilburn, Justin},
   author={Kim, Hee-Cheol},
   title={Vortices and Vermas},
   journal={Adv. Theor. Math. Phys.},
   volume={22},
   date={2018},
   number={4},
   pages={803--917},
}

\bib{Ionuts}{article}{
   author={Ciocan-Fontanine, Ionut},
   author={Favero, David},
   author={Gu\'{e}r\'{e}, J\'{e}r\'{e}my},
   author={Kim, Bumsig},
   author={Shoemaker, Mark},
   title={Fundamental factorization of a GLSM Part I: Construction},
   journal={Mem. Amer. Math. Soc.},
   volume={289},
   date={2023},
   number={1435},
   pages={iv+96},
}
	










\bib{DeformRed}{article}{
   author={Davison, Ben},
   author={P\u{a}durariu, Tudor},
   title={Deformed dimensional reduction},
   journal={Geom. Topol.},
   volume={26},
   date={2022},
   number={2},
   pages={721--776},
}



  


\bibitem{DHLloc} D.~Halpern-Leistner,
  \emph{A categorification of the Atiyah-Bott localization formula},
  available from \texttt{math.cornell.edu/~danielhl}. 

\bibitem{DHO} Marcelo De Martino, Volker Heiermann, Eric Opdam,
\emph{On the unramified spherical automorphic spectrum}, 
\texttt{arXiv:1512.08566}

\bibitem{DHO2} Marcelo De Martino, Volker Heiermann, Eric Opdam,
\emph{Residue distributions, iterated residues, and the spherical automorphic spectrum}, 
\texttt{ arXiv:2207.06773}

\bib{Dimca}{book}{
   author={Dimca, Alexandru},
   title={Sheaves in topology},
   series={Universitext},
   publisher={Springer-Verlag, Berlin},
   date={2004},
   pages={xvi+236},
}

\bib{Faddeev}{article}{
   author={Faddeev, L. D.},
   title={The inverse problem in the quantum theory of scattering},
   language={Russian},
   journal={Uspehi Mat. Nauk},
   volume={14},
   date={1959},
   number={4(88)},
}

\bib{Fan}{article}{
   author={Fantechi, Barbara},
   author={G\"{o}ttsche, Lothar},
   title={Riemann-Roch theorems and elliptic genus for virtually smooth
   schemes},
   journal={Geom. Topol.},
   volume={14},
   date={2010},
   number={1},
   pages={83--115},
}


\bib{Semiinf2}{article}{
   author={Feigin, Boris},
   author={Finkelberg, Michael},
   author={Kuznetsov, Alexander},
   author={Mirkovi\'{c}, Ivan},
   title={Semi-infinite flags. II. Local and global intersection cohomology
   of quasimaps' spaces},
   conference={
      title={Differential topology, infinite-dimensional Lie algebras, and
      applications},
   },
   book={
      series={Amer. Math. Soc. Transl. Ser. 2},
      volume={194},
      publisher={Amer. Math. Soc., Providence, RI},
   },
   date={1999},
   pages={113--148},
}

\bib{FinkKuz}{article}{
   author={Finkelberg, Michael},
   author={Kuznetsov, Alexander},
   title={Global intersection cohomology of quasimaps' spaces},
   journal={Internat. Math. Res. Notices},
   date={1997},
   number={7},
   pages={301--328},
 }

 \bib{Semiinf1}{article}{
   author={Finkelberg, Michael},
   author={Mirkovi\'{c}, Ivan},
   title={Semi-infinite flags. I. Case of global curve $\mathbf{P}^1$},
   conference={
      title={Differential topology, infinite-dimensional Lie algebras, and
      applications},
   },
   book={
      series={Amer. Math. Soc. Transl. Ser. 2},
      volume={194},
      publisher={Amer. Math. Soc., Providence, RI},
   },
   date={1999},
   pages={81--112},
}

\bib{Gaits}{article}{
   author={Gaitsgory, D.},
   title={Eisenstein series and quantum groups},
   language={English, with English and French summaries},
   journal={Ann. Fac. Sci. Toulouse Math. (6)},
   volume={25},
   date={2016},
   number={2-3},
   pages={235--315},
}


\bib{Ger1}{article}{
   author={Gerasimov, Anton},
   author={Lebedev, Dimitri},
   author={Oblezin, Sergey},
   title={From Archimedean $L$-factors to topological field theories},
   journal={Lett. Math. Phys.},
   volume={96},
   date={2011},
   number={1-3},
   pages={285--297},
}

\bib{Ger2}{article}{
   author={Gerasimov, Anton},
   author={Lebedev, Dimitri},
   author={Oblezin, Sergey},
   title={Archimedean $L$-factors and topological field theories II},
   journal={Commun. Number Theory Phys.},
   volume={5},
   date={2011},
   number={1},
   pages={101--133},
}

\bib{Harder}{article}{
   author={Harder, G.},
   title={Chevalley groups over function fields and automorphic forms},
   journal={Ann. of Math. (2)},
   volume={100},
   date={1974},
   pages={249--306},
}



\bib{Mirror}{book}{
   author={Hori, Kentaro},
   author={Katz, Sheldon},
   author={Klemm, Albrecht},
   author={Pandharipande, Rahul},
   author={Thomas, Richard},
   author={Vafa, Cumrun},
   author={Vakil, Ravi},
   author={Zaslow, Eric},
   title={Mirror symmetry},
   series={Clay Mathematics Monographs},
   volume={1},
   note={With a preface by Vafa},
   publisher={American Mathematical Society, Providence, RI; Clay
   Mathematics Institute, Cambridge, MA},
   date={2003},
   pages={xx+929},
}

\bib{KS}{book}{
   author={Kashiwara, Masaki},
   author={Schapira, Pierre},
   title={Sheaves on manifolds},
   series={Grundlehren der mathematischen Wissenschaften [Fundamental
   Principles of Mathematical Sciences]},
   volume={292},
   note={With a chapter in French by Christian Houzel},
   publisher={Springer-Verlag, Berlin},
   date={1990},
   pages={x+512},
}






\bibitem{KO1} D.~Kazhdan and A.~Okounkov, 
  \emph{On the unramified Eisenstein spectrum},
  \texttt{arXiv:2203.03486}. 

\bibitem{KO2} \bysame, 
  \emph{Unramified Langlands spectral decomposition for function field},
in preparation. 
  





\bib{Labesse}{article}{
   author={Labesse, Jean-Pierre},
   title={The Langlands spectral decomposition},
   conference={
      title={The genesis of the Langlands Program},
   },
   book={
      series={London Math. Soc. Lecture Note Ser.},
      volume={467},
      publisher={Cambridge Univ. Press, Cambridge},
   },
   date={2021},
   pages={176--214},
}

\bib{Lvol}{article}{
   author={Langlands, R. P.},
   title={The volume of the fundamental domain for some arithmetical
   subgroups of Chevalley groups},
   conference={
      title={Algebraic Groups and Discontinuous Subgroups},
      address={Proc. Sympos. Pure Math., Boulder, Colo.},
      date={1965},
   },
   book={
      publisher={Amer. Math. Soc., Providence, R.I.},
   },
   date={1966},
   pages={143--148},
}

\bib{L}{book}{
   author={Langlands, Robert P.},
   title={On the functional equations satisfied by Eisenstein series},
   series={Lecture Notes in Mathematics, Vol. 544},
   publisher={Springer-Verlag, Berlin-New York},
   date={1976},
}

\bib{Lmarch}{article}{
   author={Langlands, R. P.},
   title={Automorphic representations, Shimura varieties, and motives. Ein
   M\"{a}rchen},
   conference={
      title={Automorphic forms, representations and $L$-functions},
      address={Proc. Sympos. Pure Math., Oregon State Univ., Corvallis,
      Ore.},
      date={1977},
   },
   book={
      series={Proc. Sympos. Pure Math.},
      volume={XXXIII},
      publisher={Amer. Math. Soc., Providence, RI},
   },
   date={1979},
   pages={205--246},
}


\bib{Laff}{article}{
   author={Lafforgue, Vincent},
   title={Chtoucas pour les groupes r\'{e}ductifs et param\'{e}trisation de
   Langlands globale},
   language={French},
   journal={J. Amer. Math. Soc.},
   volume={31},
   date={2018},
   number={3},
   pages={719--891},
}

\bib{Laumon}{article}{
   author={Laumon, G.},
   title={Faisceaux automorphes li\'{e}s aux s\'{e}ries d'Eisenstein},
   language={French},
   conference={
      title={Automorphic forms, Shimura varieties, and $L$-functions, Vol.
      I},
      address={Ann Arbor, MI},
      date={1988},
   },
   book={
      series={Perspect. Math.},
      volume={10},
      publisher={Academic Press, Boston, MA},
   },
   date={1990},
   pages={227--281},
}


\bib{Macd}{article}{
   author={Macdonald, I. G.},
   title={Symmetric products of an algebraic curve},
   journal={Topology},
   volume={1},
   date={1962},
   pages={319--343},
}





		
\bib{Moe1}{article}{
   author={M\oe glin, C.},
   author={Waldspurger, J.-L.},
   title={Le spectre r\'{e}siduel de ${\rm GL}(n)$},
   journal={Ann. Sci. \'{E}cole Norm. Sup. (4)},
   volume={22},
   date={1989},
   number={4},
   pages={605--674},
}

\bib{Moe2}{article}{
   author={M\oe glin, C.},
   title={Orbites unipotentes et spectre discret non ramifi\'{e}: le cas des
   groupes classiques d\'{e}ploy\'{e}s},
   journal={Compositio Math.},
   volume={77},
   date={1991},
   number={1},
   pages={1--54},
}

\bib{MW}{book}{
   author={M\oe glin, C.},
   author={Waldspurger, J.-L.},
   title={Spectral decomposition and Eisenstein series},
   series={Cambridge Tracts in Mathematics},
   volume={113},
   note={Une paraphrase de l'\'{E}criture [A paraphrase of Scripture]},
   publisher={Cambridge University Press, Cambridge},
   date={1995},
}



\bib{NakHilb}{article}{
   author={Nakajima, Hiraku},
   title={Heisenberg algebra and Hilbert schemes of points on projective
   surfaces},
   journal={Ann. of Math. (2)},
   volume={145},
   date={1997},
   number={2},
   pages={379--388},
   issn={0003-486X},
   review={\MR{1441880}},
   doi={10.2307/2951818},
}


\bib{NakCoul}{article}{
   author={Nakajima, Hiraku},
   title={Towards a mathematical definition of Coulomb branches of
   3-dimensional $\mathcal{N}=4$ gauge theories, I},
   journal={Adv. Theor. Math. Phys.},
   volume={20},
   date={2016},
   number={3},
   pages={595--669},
}




\bibitem{O2} A.~Okounkov, 
  \emph{Nonabelian stable envelopes, vertex functions with descendents, and integral solutions of q-difference equations},
  \texttt{arXiv:2010.13217}. 

  \bib{OP}{article}{
   author={Okounkov, A.},
   author={Pandharipande, R.},
   title={Gromov-Witten theory, Hurwitz numbers, and matrix models},
   conference={
      title={Algebraic geometry---Seattle 2005. Part 1},
   },
   book={
      series={Proc. Sympos. Pure Math.},
      volume={80, Part 1},
      publisher={Amer. Math. Soc., Providence, RI},
   },
   date={2009},
   pages={325--414},
}





\bib{Serre}{book}{
   author={Serre, Jean-Pierre},
   title={Arbres, amalgames, ${\rm SL}_{2}$},
   language={French},
   series={Ast\'{e}risque, No. 46},
   note={Avec un sommaire anglais;
   R\'{e}dig\'{e} avec la collaboration de Hyman Bass},
   publisher={Soci\'{e}t\'{e} Math\'{e}matique de France, Paris},
   date={1977},
   pages={189 pp. },
}






\bib{Leon}{book}{
   author={Takhtajan, Leon A.},
   title={Quantum mechanics for mathematicians},
   series={Graduate Studies in Mathematics},
   volume={95},
   publisher={American Mathematical Society, Providence, RI},
   date={2008},
   pages={xvi+387},
   isbn={978-0-8218-4630-8},
   review={\MR{2433906}},
   doi={10.1090/gsm/095},
}


\bib{VinOnish}{article}{
   author={Vinberg, \`E. B.},
   author={Gorbatsevich, V. V.},
   author={Onishchik, A. L.},
   title={Structure of Lie groups and Lie algebras},
   book={
      series={Itogi Nauki i Tekhniki},
      publisher={Akad. Nauk SSSR, Vsesoyuz. Inst. Nauchn. i Tekhn. Inform.,
   Moscow},
   },
   date={1990},
   pages={5--259},
}

\bib{VinPop}{article}{
   author={Vinberg, E. B.},
   author={Popov, V. L.},
   title={Invariant theory},
   book={
      series={Itogi Nauki i Tekhniki},
      publisher={Akad. Nauk SSSR, Vsesoyuz. Inst. Nauchn. i Tekhn. Inform.,
   Moscow},
   },
   date={1989},
   pages={137--314, 315},
}







	\end{biblist}
\end{bibdiv}
\end{document}